\documentclass[reqno,10pt]{amsart}
\usepackage{amscd,amssymb,amsmath,color,url}
\usepackage{latexsym}
\usepackage[all]{xy}
\usepackage{defs}
\usepackage[colorlinks=true]{hyperref}

\def\..{{,\dots ,}}

\def\env{{\rm env}}
\def\mon{{\rm mon}}

\def\rmwt{{\rm wt}}

\def\rmlog{{\rm log}}
\def\whOmega{{\wh\Omega}}
\def\whd{{\wh d}}

\def\nonexp{{\rm nonexp}}
\def\tf{{\rm tf}}

\def\di{{\diamond}}

\def\expchar{{\rm exp.char}}
\def\adic{{\rm adic}}

\begin{document}

\title[Metrization of differential pluriforms on Berkovich spaces]{Metrization of differential pluriforms on Berkovich analytic spaces}

\author{Michael Temkin}
\address{Einstein Institute of Mathematics, The Hebrew University of Jerusalem, Giv'at Ram, Jerusalem, 91904, Israel}
\email{temkin@math.huji.ac.il}
\thanks{This work was supported by the European Union Seventh Framework Programme (FP7/2007-2013) under grant agreement 268182.}
\keywords{Berkovich analytic spaces, K\"ahler seminorms, pluriforms, skeletons.}

\begin{abstract}
We introduce a general notion of a seminorm on sheaves of rings or modules and provide each sheaf of relative differential pluriforms on a Berkovich $k$-analytic space with a natural seminorm, called K\"ahler seminorm. If the residue field $\tilk$ is of characteristic zero and $X$ is a quasi-smooth $k$-analytic space, then we show that the maximality locus of any global pluricanonical form is a PL subspace of $X$ contained in the skeleton of any semistable formal model of $X$. This extends a result of Musta\c{t}\u{a} and Nicaise, because the K\"ahler seminorm on pluricanonical forms coincides with the weight norm defined by Musta\c{t}\u{a} and Nicaise when $k$ is discretely valued and of residue characteristic zero.
\end{abstract}

\maketitle

\section{Introduction}

\subsection{Motivation}
It often happens that a Berkovich space $X$ possesses natural skeletons, which are, in particular, deformational retracts of $X$ of finite topological type, see \cite[Section~4.3]{berbook} and \cite{bercontr}. In some special cases, such as the case of curves of positive genus or abelian varieties, there is a canonical (usually, minimal) skeleton. As a rule, skeletons are obtained from nice formal models, e.g. polystable or, more generally, log smooth ones, though we should mention for completeness that a different and very robust method of constructing skeletons was developed very recently by Hrushovski and Loeser, see \cite{HL}.

In \cite{KS}, Kontsevich and Soibelman constructed a canonical skeleton of analytic K3 surfaces over $k=\bfC((t))$ by use of a new method: the skeleton is detected as the extremality locus of the canonical form. In \cite{Mustata-Nicaise}, this method was extended by Musta\c{t}\u{a} and Nicaise as follows. If $k$ is discretely valued and $X$ is the analytification of a smooth and proper $k$-variety, they constructed norms on the pluricanonical sheaves $\omega^{\otimes n}_X$ and showed that the maximality locus of any non-zero pluricanonical form $\phi$ is contained in the skeleton associated with any semistable formal model of $X$. The union of the maximality loci of non-zero pluricanonical forms is called the essential skeleton of $X$ in \cite{Mustata-Nicaise}. It is an important ``combinatorial" subset of $X$, see \cite{Nicaise-Xu}, although it does not have to be a skeleton of $X$.

An advantage of the above approach is that it constructs a valuable combinatorial subset, the essential skeleton, in a canonical way. In particular, the only input is the metrization of pluricanonical sheaves, and no choice of a formal model is involved. Slightly ironically, one heavily exploits formal models to metrize $\omega_X^{\otimes n}$. On the one hand, existence of nice global formal models is not needed since it suffices for any so-called divisorial point $x$ to find a sufficiently small domain that contains $x$ and possesses a regular formal model. The latter problem is much easier and the construction works fine when $\cha(\tilk)>0$ and existence of nice global models is a dream. But on the other hand, this still leads to technical restrictions, including the assumption that $k$ is discretely valued. In addition, the construction of the norm is not so geometric: first one defines it at divisorial points by use of formal models, and then extends it to the whole $X$ by continuity.

The original aim of this project was to provide a natural local analytic construction of Musta\c{t}\u{a}-Nicaise norm, which applies to all points on equal footing and eliminates various technical restrictions of their method. The basic idea is very simple: provide $\Omega_X$ with the maximal seminorm making the differential $d\:\calO_X\to\Omega_X$ a non-expansive map, and induce from it a seminorm on $\omega_X^{\otimes n}=(\wedge^d\Omega_X)^{\otimes n}$. Moreover, the same definition makes sense for any morphism $f\:X\to S$, so the construction generalizes to the relative situation and no assumption on $k$ and $f$ is needed. This increases the flexibility even in the original setting; for example, one obtains a way to work with analytic families of proper smooth varieties. Unfortunately, implementation of the basic idea is not so simple due to lack of various foundations. So, a large part of this project is devoted to developing basic topics, including a theory of K\"ahler seminorms on modules of differentials, its application to real-valued fields, metrization of sheaves of modules, etc. On the positive side, we think that this foundational work will be useful for future research in non-archimedean geometry and related areas.

Once K\"ahler seminorms will have been defined, we will study the maximality locus of pluricanonical forms. Unfortunately, the assumption that $\cha(\tilk)=0$ seems unavoidable with current technique, but we manage to treat the non-discrete case as well. In particular, we only use a result \`a la de Jong (see \cite[Theorem~3.4.1]{altered}) instead of the existence of semistable model (a result \`a la Hironaka). Finally, under the assumptions of \cite{Mustata-Nicaise}, we compare the K\"ahler seminorm to the Musta\c{t}\u{a}-Nicaise norm on the pluricanonical sheaves. Surprisingly, they coincide only when $\cha(\tilk)=0$, and in general they are related by a factor which up to a constant coincides with the log different of $\calH(x)/k$.

\subsection{Methods}
At few places in the paper, including the definition of K\"ahler seminorms, we have to choose one method out of few possibilities. Let us discuss briefly what these choices are.

\subsubsection{An approach via unit balls}
One way to define a seminorm on a $k$-vector space $V$ is by using its unit ball $V^\di$, which is a $\kcirc$-module. Technically, this leads to the easiest way to metrize a coherent $\calO_X$-sheaf $\calF$: just choose an $\calOcirc_X$-submodule $\calF^\di_X$ (subject to simple restrictions). For example, if the valuation is not discrete one can simply define the K\"ahler seminorm $\lVert\ \rVert_\Omega$ on $\Omega_{X/S}$ as the seminorm associated with the sheaf $\calOcirc_Xd_{X/S}(\calOcirc_X)$, the minimal $\calOcirc_X$-submodule of $\Omega_{X/S}$ containing $d_{X/S}\calOcirc_X$. Unfortunately, this definition is problematic when the valuation is discrete or trivial (though, see Remark~\ref{unitballrem}). In order to consider ground fields with discrete or trivial valuations on the equal footing, we have to develop all basic constructions in terms of seminorms themselves. For example, $\lVert\ \rVert_\Omega$ can be characterized as the maximal seminorm such that the map $d\:\calO_X\to\Omega_{X/S}$ is non-expansive (see Lemma~\ref{univkahlem}). Still, some implicit use of unit balls is made in Section~\ref{Kahlersec}, e.g. see Theorem~\ref{unitballth}.

\subsubsection{Seminormed algebras versus Banach algebras}
Most seminormed rings in Berkovich geometry are Banach. Nevertheless, more general seminormed rings also show up, with the main example being the local rings $\calO_{X,x}$ and their residue fields $\kappa(x)$. For this reason, it is technically much more convenient to work with seminormed rings and modules rather than their Banach analogues throughout the paper. In addition, it turns out to be important to consider only non-expansive homomorphisms, while classical Banach categories contain all bounded homomorphisms.

\subsubsection{Metrization of sheaves}
There exist two ways to define seminorms on sheaves of rings or modules. In this paper we implement a sheaf theoretic approach, which applies to any site. We introduce the notions of (pre)sheaves of seminorms on a sheaf of abelian groups (resp. rings or modules) $\calA$. In fact, this is equivalent to introducing (pre)sheaves of seminormed abelian groups with the underlying sheaf $\calA$. Various operations, such as tensor products, are defined using sheafification. A slight technical complication of the method is that one has to consider unbounded seminorms.

A simpler ad hoc method to metrize sheaves on topological spaces or sites with enough points is to metrize the stalks in a semicontinuous way: for a section $s\in\calF(U)$ let $\lvert s\rvert \:U\to\bfR_{\ge 0}$ denote the function sending $x\in U$ to $\lvert s\rvert_x$, then all functions $\lvert s\rvert$ should be upper semicontinuous. All operations are then defined stalkwise. The main problem with applying this method to our case is that one has to work with all points of the $G$-topology site $X_G$, which is not a standard tool in Berkovich geometry. We describe these points in the end of the paper, but they are not used in our main constructions.

\begin{rem}
(i) It is more usual to consider only continuous metrics. For example, the definition of metrization in \cite{CLD} requires that all sections $\lvert s\rvert$ are continuous. Nevertheless, it is the semicontinuity that encodes the condition that $\lvert \ \rvert $ is a sheaf, see Theorem~\ref{conservlem}.

(ii) In the case of K\"ahler seminorms, non-continuous functions $\lVert \phi\rVert _\Omega$ arise in the simplest cases. For example, already the function $\lVert dt\rVert_\Omega$ on the disc $\calM(k\{t\})$ is not continuous; in fact, $\lVert dt\rVert_\Omega$ is the radius function, see \S\ref{discsec}.
\end{rem}

\subsection{Overview of the paper and main results}
In Section~\ref{adhocsec} we fix our notation and study seminorms on vector spaces over real-valued fields and modules over real valuation rings. Most of the material is probably known to experts but some of it is hard to find in the literature, especially the material on index of semilattices and content of torsion modules. Section~\ref{metrsheafsec} deals with metrization of sheaves. First, we study sheaves on arbitrary sites and then specialize to the case of $G$-sheaves on analytic spaces. In Section~\ref{difsec}, we extend the theory of K\"ahler differentials to seminormed rings. In particular, given a non-expansive homomorphism of seminormed rings $A\to B$ we show that $d_{B/A}\:B\to(\Omega_{B/A},\lVert \ \rVert_\Omega)$ is the universal non-expansive $A$-derivation, where the K\"ahler seminorm $\lVert \ \rVert_\Omega$ is defined as the maximal seminorm making the homomorphism $d_{B/A}$ non-expansive.

In Section \ref{Kahlersec} we study the K\"ahler seminorm on a vector space $\Omega_{K/A}$, where $K$ is a real-valued field. In Theorem~\ref{unitballth} we show that $\Omega^\rmlog_{\Kcirc/\Acirc}$ modulo its torsion is an almost unit ball of the K\"ahler seminorm on $\Omega_{K/A}$. Thus, the study of K\"ahler seminorms is tightly related to study of $\Omega^\rmlog_{\Kcirc/\Acirc}$ and ramification theory. Although this is classical in the discretely valued case, \cite[Chapter~6]{Gabber-Ramero} is the only reference for such material when $\lvert K^\times\rvert$ is dense. Unfortunately, even loc.cit. does not cover all our needs, so we have to dig into the theory of real-valued fields, that makes this section the most technical in the paper. One of our main results there is that if $K$ is dense in a real-valued field $L$ then $\hatOmega_{K/A}\toisom\hatOmega_{L/A}$, see Theorem~\ref{denseth} and its corollary.

In Section~\ref{metrsec} we define the K\"ahler seminorm $\lVert \ \rVert_\Omega$ on $\Omega_{X/S}$ for any morphism $f\:X\to S$. This is done by sheafifying the presheaf of K\"ahler seminorms on affinoid domains and it follows easily from the definition that $\lVert \ \rVert_\Omega$ is the maximal seminorm on $\Omega_{X/S}$ making the map $d\:\calO_{X_G}\to\Omega_{X/S}$ non-expansive. In Theorem~\ref{isometryth} we show that the completed stalk $\wh{\Omega_{X/S,x}}$ is isomorphic to $\hatOmega_{\calH(x)/\calH(s)}$, where $s=f(x)$. This provides the main tool for explicit work with the seminorm $\lVert \ \rVert_\Omega$ and its stalks. Note that the main ingredient in proving Theorem~\ref{isometryth} is that $\hatOmega_{\kappa(x)/\kappa(s)}=\hatOmega_{\calH(x)/\calH(s)}$ by Theorem~\ref{denseth}. Another fundamental property of K\"ahler seminorms is established in Theorem~\ref{analyticomega}: K\"ahler seminorms is determined by the usual points of $X$ (the $G$-analyticity condition from \S\ref{sheavesansec}) and the functions $\lVert s \rVert_\Omega$ are semicontinuous with respect to the usual topology of $X$ (and not only the $G$-topology). Some examples of K\"ahler seminorms are described in Section~\ref{examsec}, including the ones demonstrating that K\"ahler seminorms behave weirdly when $k$ possesses wildly ramified extensions. Also, we study the compatibility of $\lVert \ \rVert_\Omega$ with base change in Section~\ref{changesec}. In particular, we show that it is compatible with restriction to the fibers and tame extensions of the ground field (Theorem~\ref{compatth} and its corollaries), but is incompatible with wild extensions of the ground field. Finally, we use $\lVert \ \rVert_\Omega$ to metrize the sheaves $S^m(\Omega^n_{X/S})$ of relative pluriforms on $X$.

Section~\ref{PLsec} is devoted to recalling basic facts about formal models and skeletons of Berkovich spaces. Then we study in Section~\ref{lastsec} metrization of the sheaves $\omega_X^{\otimes m}$ of pluricanonical forms on a rig-smooth space $X$. In Corollary~\ref{monomnormcor2} we obtain a simple formula that evaluates K\"ahler seminorms at monomial points, and we deduce in Theorem~\ref{rplth} that the restrictions of geometric K\"ahler seminorms onto PL subspaces of $X$ are PL. The main ingredient here is Theorem~\ref{tamechart}. In Theorem~\ref{comparth} we establish the connection between K\"ahler seminorm on pluricanonical sheaves and the weight norm of Musta\c{t}\u{a} and Nicaise. Finally, in Section~\ref{maxsec} we study the maximality locus of a non-zero pluricanonical form with respect to the geometric K\"ahler seminorm. We prove that it is contained in the essential skeleton of $X$, see Theorem~\ref{semistableprop}, and, if $\cha(\tilk)=0$, it is a PL subspace of $X$, see Theorem~\ref{plmaxth}. When $\cha(\tilk)=0$ this extends the results of Musta\c{t}\u{a} and Nicaise to the case of non-discrete $\lvert k^\times\rvert$ and arbitrary quasi-smooth $X$, not necessarily algebraizable or even strictly analytic. (If $\cha(\tilk)>0$, our norm differs from the weight norm.)

Finally, Section~\ref{kansec} is devoted to study the topological realizations of the $G$-topologies on analytic spaces and PL spaces. Our description of the topological space $\lvert X_G\rvert$ seems to be new, see Section~\ref{Gsec} and Remark~\ref{Ocircrem}(ii) concerning the connection to adic and reified adic spaces. In particular, we interpret the points of $\lvert X_G\rvert$ in terms of the graded reductions of germs. Also, we interpret points of the PL topologies in terms of combinatorial valuations (or valuations on lattices), see Section~\ref{Psec}, and for a PL subspace $P\subseteq X$ we describe the embedding $\lvert P_G\rvert \into\lvert X_G\rvert$. This is used to prove a strong result on the structure of $P$: locally $P$ possesses a residually unramified chart, see Theorem~\ref{tamechart}. It seems that our usage of the space $\lvert P_G\rvert$ is more or less equivalent to the use of model theory in \cite{Ducpl2} and \cite{skeletons}, see Remark~\ref{PGrem}. Note also that Theorem~\ref{tamechart} is the only result of Section~\ref{kansec} used in the main part of the paper (in the proof of Theorem~\ref{rplth}).

\subsection{Acknowledgements}
This work originated at Simons Symposium on Non-Archimedean and Tropical Geometry in 2013, where I learned about the work of Musta\c{t}\u{a} and Nicaise and found an analytic approach to metrize pluriforms. I want to thank Simons Foundation and the organizers of the Symposium, Matt Baker and Sam Payne. I owe much to the anonymous referees for pointing out numerous inaccuracies and making various suggestions on improving the quality and readability of the paper. Also, I am grateful to J. Nicaise for answering questions related to \cite{Mustata-Nicaise}, to A. Ducros for showing his work with Thuillier \cite{skeletons} and answering various questions, and to V. Berkovich, I. Tyomkin and Y. Varshavsky for useful discussions.

\setcounter{tocdepth}{1}
\tableofcontents

\section{Real-valued fields}\label{adhocsec}
In this section we study seminormed vector spaces over a real-valued field $K$ and modules over the ring of integers $\Kcirc$.

\subsection{Conventions}\label{convsec}
First, let us fix basic terminology and notation on valued fields and analytic spaces.

\subsubsection{Valued fields}
By a {\em valued field} we mean a field $F$ provided with a non-archimedean valuation $\lvert \ \rvert \:F\to\{0\}\cup\Gamma$, where $\Gamma$ is an ordered group. The conditions $\lvert x\rvert \le 1$ and $\lvert x\rvert <1$ define the {\em ring of integers} $\Fcirc$ and its maximal ideal $\Fcirccirc$, respectively. In addition, $\tilF=\Fcirc/\Fcirccirc$ denotes the {\em residue field} of $F$.

\subsubsection{The real-valued case}\label{realvalsec}
Assume that a valued field $K$ is {\em real-valued}, i.e. $\Gamma=\bfR_{>0}$. Then $\lvert \ \rvert $ is a norm and hence defines a topology on $K$. We will use the notation $\lvert \Kcirccirc\rvert =\sup_{\pi\in\Kcirccirc}\lvert \pi\rvert $. Thus, $\lvert \Kcirccirc\rvert =0$ if the valuation is trivial, $\lvert \Kcirccirc\rvert =\lvert \pi\rvert $ if $K$ is discretely valued with uniformizer $\pi$, and $\lvert \Kcirccirc\rvert =1$ otherwise.

Let, now, $\pi$ be any element of $\Kcirccirc\setminus\{0\}$ if the valuation is non-trivial, and set $\pi=0$ otherwise. In particular, the induced topology on $\Kcirc$ is the $\pi$-adic one. Given a $\Kcirc$-module $M$ we say that an element $x\in M$ is {\em divisible} if it is infinitely $\pi$-divisible. In particular, if the valuation of $K$ is trivial then $0$ is the only divisible element. A $\Kcirc$-module is {\em divisible} if all its elements are divisible.

\subsubsection{Analytic spaces}
Throughout this paper, $k$ is a non-archimedean {\em analytic field}, i.e. a real-valued field which is complete with respect to its valuation. Trivial valuation is allowed. All analytic spaces we will consider are $k$-analytic spaces in the sense of \cite[\S1]{berihes}.

In addition, we fix a divisible subgroup $H\subseteq\bfR_{>0}$ such that $H\neq 1$ and $\lvert k^\times\rvert \subseteq H$ and consider only $H$-strict analytic spaces in the sense of \cite{descent}. For shortness, we will often call them {\em analytic spaces}.

\subsubsection{The $G$-topology}
The usual topology of an analytic space $X$ can be used for working with coherent sheaves only when $X$ is good. In general, one has to work with the $G$-topology of analytic domains whose coverings are the set-theoretical coverings $U=\cup_{i\in I} U_i$ such that $\{U_i\}$ is a quasi-net on $U$ in the sense of \cite[\S1.1]{berbook}. These coverings are usually called {\em $G$-coverings} or {\em admissible} coverings. By $X_G$ we denote the associated site: its objects are ($H$-strict) analytic domains and coverings are the $G$-admissible ones. The structure sheaf $\calO_{X_G}$ of $X$ is a sheaf on $X_G$, and by $\calO^\circ_{X_G}$ we denote the subsheaf of $\kcirc$-algebras whose sections have spectral seminorm bounded by 1, i.e. $\calO_{X_G}^\circ(V)=\calA_V^\circ$ for an affinoid domain $V=\calM(\calA_V)$.

\subsection{Seminormed rings and modules}\label{normringsec}
In Section \ref{normringsec} we recall well-known facts and definitions concerning seminorms. All seminorms we consider are non-archime\-dean. Ring seminorms will be denoted $\lvert \ \rvert$, $\lvert \ \rvert'$, etc., and module seminorms will be denoted $\lVert \ \rVert$, $\lVert \ \rVert'$, etc.

\subsubsection{Seminormed abelian groups}
Throughout this paper, a {\em seminorm} on an abelian group $A$ is a function $\lVert \ \rVert \:A\to\bfR_{\ge 0}$ such that
\begin{itemize}
\item[(0)] $\lVert 0\rVert =0$,
\item[(1)] the inequality $\lVert x-y\rVert \le\max(\lVert x\rVert ,\lVert y\rVert )$ holds for any $x,y\in A$.
\end{itemize}
Note that (1) is a short way to encode the more standard conditions that $\lVert x+y\rVert \le\max(\lVert x\rVert ,\lVert y\rVert )$ and $\lVert-x\rVert=\lVert x\rVert$. The pair $(A,\lVert \ \rVert )$ will be called a {\em seminormed group}. If $\lVert \ \rVert$ has trivial kernel then it is called a {\em norm}.

\subsubsection{Bounded and non-expansive homomorphisms}
A homomorphism $\phi\:A\to B$ is called {\em bounded} with respect to $\lVert \ \rVert_A$ and $\lVert \ \rVert_B$ if there exists $C=C(\phi)$ such that $\lVert a\rVert \le C\lVert \phi(a)\rVert$ for any $a\in A$. If $C=1$ then $\phi$ is called {\em non-expansive}.

%\begin{rem}
%It would be more precise to say ``non-expansive" instead of ``non-expansive", but we follow the usual terminology.
%\end{rem}

\subsubsection{The non-expansive category}
As in \cite[Section~5.1]{BBK}, the category of seminormed abelian groups with non-expansive homomorphisms will be called the {\em non-expansive category} (of seminormed abelian groups).

\begin{rem}
(i) Often one works with the larger category whose morphisms are arbitrary bounded morphisms; let us call it the {\em bounded category}. In fact, it is equivalent to the localization of the non-expansive category by bounded maps that possess a bounded inverse.

(ii) Working with the bounded category is natural when one wants to study seminorms up to equivalence; for example, this is the case in the theory of Banach spaces. On the other side, working with the non-expansive category is natural when one distinguishes equivalent seminorms, so this fits the goals of the current paper.

(iii) A serious advantage of working with the non-expansive category is that one can describe limits and colimits in a simple way, see \cite[Section~5.1]{BBK}.
\end{rem}

\subsubsection{Seminormed rings}
A {\em seminorm} on a ring $A$ is a seminorm $\lvert \ \rvert$ on the underlying group $(A,+)$ that satisfies $\lvert 1\rvert =1$ and $\lvert xy\rvert \le\lvert x\rvert \lvert y\rvert$ for any $x,y$. A ring with a fixed seminorm is called a {\em seminormed ring}. Usually it will be denoted by calligraphic letters and the seminorm will be omitted from the notation, e.g. $\calA=(\calA,\lvert \ \rvert_\calA)$. An important example of a normed ring is a real-valued field.

\subsubsection{Seminormed modules}
A seminormed $\calA$-module $M$ is an $\calA$-module provided with a seminorm $\lVert \ \rVert_M$ such that $\lVert am\rVert_M\le\lvert a\rvert_\calA\lVert m\rVert_M$ for any $a\in\calA$ and $m\in M$. The notions of non-expansive homomorphisms of seminormed rings and modules are defined in the obvious way.

\subsubsection{Quotient seminorms and cokernels}
If $\phi\:A\to B$ is a surjective homomorphism of seminormed abelian groups, rings or modules then the {\em quotient seminorm} on $B$ is defined by $\lVert x\rVert_B=\inf_{y\in\phi^{-1}(x)}\lVert y\rVert_A$. It is the maximal seminorm such that $\phi$ is non-expansive, hence, in the case of groups and modules, $(B,\lVert \ \rVert_B)$ is the cokernel of any non-expansive homomorphism $C\to A$ whose image is $\Ker(\phi)$.

\subsubsection{Strictly admissible homomorphisms}
Recall that a homomorphism $\psi\:C\to D$ of seminormed abelian groups (resp. rings or modules) is called {\em admissible} if the quotient seminorm on $\psi(C)$ is equivalent to the seminorm induced from $D$. In the non-expansive category, it is natural to consider the following more restrictive notion: $\psi$ is {\em strictly admissible} if the quotient seminorm on $\psi(C)$ equals to the seminorm induced from $D$.

\subsubsection{Tensor products}\label{tensorsec}
Given a seminormed ring $\calA$, by {\em tensor product} of seminormed $\calA$-modules $M$ and $N$ we mean the module $L=M\otimes_\calA M$ provided with the {\em tensor seminorm} $\lVert \ \rVert_\otimes$ such that $\lVert l\rVert_\otimes=\inf(\max_i \lVert m_i\rVert\cdot\lVert n_i\rVert)$, where the infimum is taken over all representations $l=\sum_{i=1}^n m_i\otimes n_i$. Note that $\lVert \ \rVert_\otimes$ is the maximal seminorm such that the bilinear map $\phi\:M\times N\to L$ is {\em non-expansive}, i.e. satisfies $\lVert \phi(m,n)\rVert_\otimes\le\lVert m\rVert \cdot\lVert n\rVert$. Obviously, $M\times N\to(L,\lVert \ \rVert_\otimes)$ is the universal non-expansive bilinear map.

\subsubsection{Exterior and symmetric powers}\label{powersec}
If $M$ is a seminormed $\calA$-module then the modules $S^nM$ and $\bigwedge^n M$ acquire a natural seminorm as follows: both are quotients of $\otimes^n M$, so we consider the tensor seminorm on $\otimes^n M$ and endow $S^nM$ and $\bigwedge^n M$ with the quotient seminorms. The latter can be characterized as the maximal seminorms such that $\lVert m_1\otimes\dots\otimes m_n\rVert_{S^n}$ and $\lVert m_1\wedge\dots\wedge m_n\rVert_{\bigwedge^n}$ do not exceed $\prod_{i=1}^n\lVert m_i\rVert$ for any $m_1{,\dots ,}m_n\in M$.

\subsubsection{Filtered colimits}\label{normcolimsec}
Assume that $\{(\calA_\lam,\lVert \ \rVert_\lam)\}_{\lam\in\Lambda}$ is a filtered family of seminormed abelian groups (resp. rings or $\calA$-modules) with non-expansive transition homomorphisms, $\calA=\colim_\lam\calA_\lam$ is the filtered colimit, and $f_\lam\:\calA_\lam\to\calA$ are the natural maps. We endow $\calA$ with the {\em colimit seminorm} given by $\lVert a\rVert =\inf_{\lam\in\Lambda,b\in f^{-1}_\lam(a)}\lVert b\rVert_\lam$. Obviously, this is the maximal seminorm making each homomorphism $f_\lam$ non-expansive, and hence $\calA$ is the colimit of $\calA_\lam$ in the category of seminormed rings.

\begin{lem}\label{colimlem}
Filtered colimits of seminormed rings and modules are compatible with quotients, tensor products, symmetric and exterior powers.
\end{lem}
\begin{proof}
For concreteness, consider the case of tensor products. Assume that $\{M_i\}$ and $\{N_i\}$ are two filtered families and set $L_i=M_i\otimes N_i$, $M=\colim_i M_i$, $N=\colim_i N_i$ and $L=\colim_i L_i$. By the universal properties of tensor and colimit seminorms we obtain a non-expansive homomorphism $\phi\:L\to M\otimes N$, whose underlying homomorphism is the classical isomorphism. It remains to show that $\phi^{-1}$ is non-expansive too. If $l\in M\otimes N$ satisfies $\lVert l\rVert _{M\otimes N}<r$ then $l=\sum_{j=1}^r m_j\otimes n_j$ with $\max_j(\lVert m_j\rVert _M\cdot\lVert n_j\rVert _N)<r$. Choosing $i$ large enough we can achieve that $m_j$ and $n_j$ come from elements $m'_j\in M_i$ and $n'_j\in N_i$ whose norms are so close to the norms of $m_j$ and $n_j$ that the inequality $\max_j(\lVert m'_j\rVert _{M_i}\cdot\lVert n'_j\rVert _{N_i})<r$ holds. Then $l'=\sum_j m'_j\otimes n'_j\in L_i$ is a lifting of $l$ satisfying $\lVert l'\rVert _{L_i}<r$ and hence $\lVert \phi^{-1}(l)\rVert _L<r$, as required.
\end{proof}

\subsubsection{The completion along a seminorm}
Throughout this paper ``complete" means what one sometimes calls ``Hausdorff and complete". Similarly, by ``completion" we mean what one sometimes calls ``separated completion".

The completion of $A$ with respect to the semimetric $d(x,y)=\lVert x-y\rVert$ is denoted $\hatA$. Note that $\lVert \ \rVert$ extends to $\hatA$ by continuity and the completion map $\alp\:A\to\hatA$ is an {\em isometry} (i.e. $\lVert x\rVert =\lVert \alp(x)\rVert$) whose kernel is the kernel of $\lVert \ \rVert$. Completion is functorial and, in the case of groups or modules, it takes exact sequences to {\em semiexact} sequences, i.e. sequences in which $\Im(d_{i+1})$ is dense in $\Ker(d_i)$. Moreover, if all morphisms of an exact sequence are strictly admissible then the completion is exact and strictly admissible.

\subsubsection{Banach rings and modules}
A seminormed ring or module is called {\em Banach} if it is complete; in particular, the seminorm is a norm. Note that a Banach $\calA$-module $M$ is automatically a Banach $\hatcalA$-module.

\subsubsection{Unit balls}
The unit ball of a seminormed ring $\calA$ will be denoted $\calA^\di$; it is a subring of $\calA$. The unit ball of a seminormed $\calA$-module $M$ will be denoted $M^\di$; it is an $\calA^\di$-module.

\begin{rem}
(i) Perhaps, $\Mcirc$ would be a better notation, but the sign $\circ$ is traditionally reserved for the unit ball of the spectral seminorm of a Banach algebra. So, we use the diamond sign $\di$ instead.

(ii) If $k$ is a real-valued field, $0\neq\pi\in\kcirccirc$, and $V$ is a seminormed $k$-space then the induced topology on $V^\di$ is the $\pi$-adic one. In this case, $V$ is Banach if and only $V^\di$ is $\pi$-adic.
\end{rem}

\subsubsection{Bounded categories}
For the sake of completeness, we make some remarks on the categories of seminormed abelian groups (resp. rings or $\calA$-modules) with bounded homomorphisms. Usually, seminormed $\calA$-modules in our sense are called non-expansive and a general seminormed $\calA$-module $M$ is defined to be an $\calA$-module provided with a seminorm $\lVert \ \rVert_M$ such that there exists $C=C(M)$ satisfying $\lVert am\rVert_M\le C\lvert a\rvert_\calA\lVert m\rVert_M$ for any $a\in\calA$ and $m\in M$. Some results below can be extended to bounded homomorphisms using the following lemma, but we will not pursue this direction in the sequel.

\begin{lem}\label{shrinkublem}
Assume that $\phi\:S\to R$ is a bounded homomorphism of seminormed abelian groups, rings or modules. Then replacing the norm of $S$ with an equivalent norm one can make $\phi$ non-expansive.
\end{lem}
\begin{proof}
Define a new seminorm $\lVert \ \rVert'_S$ by $\lVert a\rVert'_S=\max(\lVert a\rVert_S,\lVert \phi(a)\rVert_R)$.
\end{proof}

\subsection{$\Kcirc$-modules}
In this section we study $\Kcirc$-modules for a real-valued field $K$. This material will be heavily used later, in particular, because such modules appear as unit balls of seminormed $K$-vector spaces.

\subsubsection{Almost isomorphisms}
Let $K$ be a real-valued field and let $M$ be a $\Kcirc$-module. We say that $M$ is a {\em torsion} module if any its element has a non-zero annihilator. By $M_\tor$ and $M_\tf=M/M_\tor$ we denote the (maximal) torsion submodule and the (maximal) torsion free quotient, respectively. We say that an element $x\in M$ is {\em almost zero} if for any $r<1$ there exists $\pi\in\Kcirc$ such that $r<\lvert \pi\rvert$ and $\pi x=0$, and $x$ is called {\em essential} otherwise. If $\lvert K^\times\rvert$ is discrete then any non-zero element is essential. As in \cite{Gabber-Ramero}, we say that a module is almost zero if all its elements are so, and a homomorphism is an {\em almost isomorphism} if its kernel and cokernel almost vanish.

\subsubsection{Almost isomorphic envelope}\label{almissec}
Let $N$ be a $\Kcirc$-module with submodules $M$ and $M'$. We say that $M$ and $M'$ are {\em almost isomorphic} as submodules if the embeddings $M\into M+M'$ and $M'\into M+M'$ are almost isomorphisms. By the {\em almost isomorphic envelope}, or just {\em envelope}, $M^\env$ of $M$ in $N$ we mean the maximal submodule $M'\subseteq N$ which is almost isomorphic to $M$. Obviously, $M'$ consists of all elements $x\in N$ such that for any $r<1$ there exists $\pi\in\Kcirc$ satisfying $r<\lvert \pi\rvert$ and $\pi x\in M$. Thus, if $\lvert K^\times\rvert$ is discrete then $M^\env=M$ for any $M$ and $N$, and if $\lvert K^\times\rvert $ is dense then $M^\env$ is the maximal submodule $M'\subseteq N$ such that $\Kcirccirc M'\subseteq M$. For example, if $N=K$ and $\lvert K^\times\rvert $ is dense then $(\Kcirccirc)^\env=\Kcirc$.

\subsubsection{Adic seminorm}
Given a $\Kcirc$-module $M$ we define the {\em adic seminorm} as $$\lVert x\rVert_\adic=\inf_{a\in\Kcirc\vert \ x\in aM}\lvert a\rvert .$$

\begin{lem}\label{adiclem}
Let $M$ be a $\Kcirc$-module and $x\in M$ an element.

(i) The adic seminorm is the maximal $\Kcirc$-seminorm on $M$ bounded by 1.

(ii) $\lVert x\rVert_\adic=0$ if and only if $x$ is divisible (see \S\ref{realvalsec}).
\end{lem}
\begin{proof}
The claim is almost obvious and the only case that requires a little care is when the valuation is trivial. In this case, $\lVert \ \rVert_\adic$ is trivial, i.e. $\lVert x\rVert_\adic=1$ whenever $x\neq 0$, and so $\lVert x\rVert_\adic=0$ if and only if $x=0$, i.e. $x$ is divisible.
\end{proof}

\begin{lem}\label{cotorlem}
(i) Any homomorphism $\phi\:M\to N$ between $\Kcirc$-modules is non-expansive with respect to the adic seminorms.

(ii) Assume that $M$ and $N$ are torsion free. Then $\phi$ is an isometry if and only if $\Ker(\phi)$ is divisible and $\Coker(\phi)$ contains no essential torsion elements.
\end{lem}
\begin{proof}
The first claim is obvious, so let us prove (ii). If $Q=\Ker(\phi)$ is not divisible, then it contains an element $x$ which is not infinitely divisible in $Q$ and hence also not infinitely divisible in $M$. Then $\lVert x\rVert_\adic\neq 0$, and hence $\phi$ is not an isometry. So, we can assume that $Q$ is divisible and we should prove that in this case $\phi$ is an isometry if and only if $\Coker(\phi)$ contains no essential torsion elements. Since $Q$ is divisible, $M/Q$ is torsion free and the map $M\to M/Q$ is an isometry. Thus, replacing $M$ with $M/Q$ we can assume that $\phi$ is injective. In this case the assertion is obvious.
\end{proof}

\subsubsection{Finitely presented modules}\label{fpsec}
The following result is proved as its classical analogue over DVR by reducing a matrix to a diagonal one by elementary operations.

\begin{lem}\label{twolattices}
Assume that $L\subseteq M$ are free $\Kcirc$-modules of ranks $l$ and $m$, respectively. Then there exists a basis $e_1\..e_m$ of $M$ and elements $\pi_1\..\pi_l\in\Kcirc$ such that $\pi_1 e_1\..\pi_le_l$ is a basis of $L$.
\end{lem}

\begin{cor}\label{fpcor}
Any finitely presented $\Kcirc$-module $M$ is of the form $\oplus_{i=1}^n M_i$ with each $M_i$ non-zero cyclic, say $M_i=\Kcirc/\pi_i\Kcirc$, and $1>\lvert \pi_1\rvert \ge\dots\ge\lvert \pi_n\rvert \ge 0$. In addition, the sequence $\lvert \pi_1\rvert \..\lvert \pi_n\rvert $ is determined by $M$ uniquely.
\end{cor}

%We say that an element $x$ of a $\Kcirc$-module $M$ is {\em primitive} if $\lVert x\rVert_\adic=1$. If $M$ is free then any element $x\in M$ is of the form $\pi y$ with $y\in M$ primitive and, clearly, $\lvert \pi\rvert =\lVert x \rVert_\adic$. The following result is proved as its classical analogue over a DVR.

%\begin{lem}\label{fglem}
%Assume that $M$ is a finitely generated torsion free $\Kcirc$-module. Then $M$ is free and any primitive element $e_1\in M$ can be completed to a basis $e_1\..e_n$.
%\end{lem}

\subsection{Seminorms on $K$-vectors spaces}\label{semivectsec}
Now, let us study seminorms on vector spaces over a real-valued field $K$.

\subsubsection{Orthogonal bases and cartesian spaces}
We recall some results from \cite[Chapter~2]{bgr}. For simplicity we only consider the finite-dimensional case; generalizations to the case of infinite dimension can be found in loc.cit. So, assume that $(V,\lVert \ \rVert )$ is a finite-dimensional seminormed $K$-vector space. A basis $e_1\..e_n$ of $V$ is called {\em $r$-orthogonal}, where $r\in(0,1]$, if for any $v=\sum_i a_ie_i$ the inequality $\lVert v\rVert \ge r\max_i(\lvert a_i\rvert \cdot\lVert e_i\rVert )$ holds. If $r=1$ then the basis is called orthogonal, and if in addition $\lVert e_i\rVert =1$ for $1\le i\le n$ then the basis is called {\em orthonormal}. One says that the seminormed space $V$ and the seminorm $\lVert \ \rVert $ are {\em weakly cartesian} (resp. {\em cartesian}, resp. {\em strictly cartesian}) if $V$ possesses an $r$-orthogonal basis for some $r$ (resp. an orthogonal basis, resp. an orthonormal basis).

\begin{rem}\label{cartsec}
(i) If $V$ is weakly cartesian then it possesses an $r$-orthogonal basis for any $r<1$, see \cite[Proposition~2.6.2/3]{bgr}.

(ii) If $K$ is complete then $V$ is weakly cartesian if and only if its seminorm is a norm, see \cite[Proposition~2.3.3/4]{bgr}. Conversely, if $K$ is not complete, say $x\in \hatK\setminus K$ then it is easy to see that $K+Kx$ with the norm induced from $\hatK$ is a normed vector space of dimension two which is not weakly cartesian (e.g. its completion is the one-dimensional $\hatK$-vector space $\hatK$).

(iii) If $K$ is spherically complete then any normed vector space is cartesian by \cite[Proposition~2.4.4/2]{bgr}. Conversely, if $K$ is not spherically complete one can easily construct two-dimensional normed vector spaces which are not cartesian.
\end{rem}

\subsubsection{Index of norms}\label{indexnormsec}
If $U$ is one-dimensional with basis $e$ then sending a seminorm to its value on $e$ provides a one-to-one correspondence between seminorms (resp. norms) on $V$ and the half-line $\bfR_{\ge 0}$ (resp. $\bfR_{>0}$). In particular, if $\lVert \ \rVert $ is a norm and $\lVert \ \rVert '$ is a seminorm on $U$ then the index $[\lVert \ \rVert ':\lVert \ \rVert ]=\lVert e\rVert '/\lVert e\rVert $ is a well-defined number independent of  the choice of $e$.

We can extend this construction using the top exterior powers. Assume that $\dim(V)=d$ and set $U=\det(V)=\bigwedge^dV$. Any seminorm $\lVert \ \rVert $ on $V$ induces the {\em determinant seminorm} $\lVert \ \rVert _{\det}=\lVert \ \rVert _{\bigwedge^d}$ on $U$, see \S\ref{powersec}. Moreover, if $\lVert \ \rVert $ is weakly cartesian then it is easy to see that $\lVert \ \rVert _{\det}$ is a norm, hence for any pair of weakly cartesian norms we can define the {\em index} $[\lVert \ \rVert ':\lVert \ \rVert ]$ to be $[\lVert \ \rVert '_{\det}:\lVert \ \rVert _{\det}]$. Obviously, the index is transitive, i.e. $$[\lVert \ \rVert ':\lVert \ \rVert ]\cdot[\lVert \ \rVert '':\lVert \ \rVert ']=[\lVert \ \rVert '':\lVert \ \rVert ].$$

\begin{rem}
The intuitive meaning of the index is that it measures the inverse ratio of volumes of the unit balls of the norms, at least when $|K^\times|$ is dense. This will be made precise in Lemma~\ref{contentindex} below.
\end{rem}

\subsection{Unit balls}\label{unitballsec}
Next, we study seminorms in terms of their unit balls. In this section, the trivially valued case will be uninteresting: we will often have to exclude it and in the remaining cases it will reduce to a triviality.

\subsubsection{Semilattices}
Assume that $V$ is a $K$-vector space and $M\subseteq V$ is a $\Kcirc$-submodule. If $M\otimes_{\Kcirc}K=V$ then we say that $M$ is a {\em semilattice of} $V$. If a semilattice $M$ is a free $\Kcirc$-module then we say that $M$ is a {\em lattice}. Note that $M$ is a semilattice if and only if it contains a basis of $V$ and hence contains a lattice.

\subsubsection{Almost unit balls}
For any $K$-seminorm $\lVert\ \rVert$ on a $K$-vector space $V$ the unit ball $V^\di$ is a $\Kcirc$-module. More generally, we say that a $\Kcirc$-submodule $M\subseteq V$ is an {\em almost unit ball} of ${\lVert\ \rVert}$ if it is almost isomorphic to $V^\di$ in the sense of \S\ref{almissec}. In particular, $V^\di$ itself is an almost unit ball too. If the valuation on $K$ is non-trivial then any almost unit ball $M$ is a semilattice.

\subsubsection{Seminorm determined by a semilattice}\label{semilatsec}
Any semilattice $M\subseteq V$ determines a $K$-seminorm on $V$ as follows: $M$ possesses the canonical adic seminorm and there is a unique way to extend it to a $K$-seminorm $\lVert\ \rVert_M$ on $V=KM$. By Lemma~\ref{adiclem}, $\lVert\ \rVert_M$ is the maximal seminorm whose unit ball contains $M$. The following lemma shows to which extent these correspondences between seminorms and semilattices are inverse one to another. The proof is simple so we omit it.

\begin{lem}\label{semilatticelem}
Let $V$ be a vector space over a real-valued field $K$.

(i) Assume that $M\subseteq V$ is a semilattice and let $\lVert\ \rVert$ be the maximal seminorm such that $\lVert M\rVert\le 1$. Then $M^\env$ is the unit ball of $\lVert\ \rVert$. In particular, $M$ is an almost unit ball of $V$.

(ii) Assume that the valuation on $K$ is non-trivial, $V$ is provided with a $K$-seminorm $\lVert\ \rVert$ and $V^\di$ is the unit ball. Let $\lVert\ \rVert'$ be the maximal seminorm bounded by 1 on $V^\di$. Then $\lVert\ \rVert'$ is the minimal seminorm that dominates $\lVert\ \rVert$ and takes values in the closure of $\lvert K\rvert$ in $\bfR_{\ge 0}$.
\end{lem}

\begin{cor}
Assume that $K$ is a real-valued field whose valuation is non-trivial and $V$ is a $K$-vector space. Then the above constructions establish a one-to-one correspondence between almost isomorphism classes of semilattices $M\subseteq V$ and $K$-seminorms on $V$ taking values in the closure of $\lvert K\rvert$ in $\bfR_{\ge 0}$.
\end{cor}

\subsubsection{Bounded semilattices}
Assume now that $V$ is finite-dimensional and let us interpret the results of Section~\ref{semivectsec} in terms of semilattices. A semilattice $M$ of $V$ is called {\em bounded} if it is contained in a lattice. Equivalently, for some (and then any) sublattice $L\subseteq M$ one has that $\pi M\subseteq L$ for some $0\neq\pi\in\Kcirc$.

\begin{lem}\label{bddsemilem}
Let $M$ be a semilattice in a finite-dimensional $K$-vector space $V$ with the associated seminorm $\lVert \ \rVert _M$.

(i) The following conditions are equivalent: (a) $M$ is almost isomorphic to a lattice of $V$ in the sense of \S\ref{almissec}, (b) $M^\env$ is a lattice, (c) $\lVert \ \rVert _M$ is strictly cartesian.

(ii) The following conditions are equivalent: (d) $M$ is bounded, (e) for any $\pi\in\Kcirccirc$ there exists a lattice $L\subseteq M$ such that $\pi M\subseteq L$, (f) $\lVert \ \rVert _M$ is weakly cartesian.
\end{lem}
\begin{proof}
Since any lattice is an almost isomorphic envelope of itself, the equivalence (a)$\Longleftrightarrow$(b) is clear. By Lemma~\ref{semilatticelem}(i), $M^\env$ is the unit ball of $\lVert \ \rVert _M$, hence a basis of $V$ is orthonormal if and only if it is also a basis of $M^\env$. This shows that (b)$\Longleftrightarrow$(c).

Obviously, (e) implies (d). If $\pi M\subseteq L\subseteq M$ and $\lvert \pi\rvert =r$ then any basis of $L$ is $r$-orthogonal with respect to $\lVert \ \rVert _M$ and we obtain that (d)$\implies$(f). Finally, let us prove that (f) implies (e). Assume that $\lVert \ \rVert _M$ is weakly cartesian and choose an $r$-orthogonal basis $e_1\..e_n$ for some $r\in(0,1]$. If $0<\lvert \pi\rvert <r$ then $M$ is contained in the lattice $\oplus_{i=1}^n\pi^{-1}\Kcirc e_i$. It follows that if $\lvert K^\times\rvert $ is discrete then $M$ is a lattice, so we can assume that $\lvert K^\times\rvert $ is dense. Choose any number $s\in(0,1)$. By Remark~\ref{cartsec}(i), there exists an $s$-orthogonal basis $e_1\..e_n$. Multiplying $e_i$ by elements of $K$ we can also achieve that $e_i\in M$ and $e_i\notin\pi M$ for any $\pi\in K$ with $\lvert \pi\rvert <s$. Then $e_i$ generate a lattice $L\subseteq M$ such that $\pi M\subseteq L$ for any $\pi\in K$ with $\lvert \pi\rvert <s^2$.
\end{proof}

\begin{rem}
(i) By definition, $m\in M$ is divisible if and only if $\lVert m\rVert _M=0$. So, Remark~\ref{cartsec}(ii) implies that $K$ is complete if and only if any semilattice without non-zero divisible elements is bounded.

(ii) One can introduce a class of almost lattices using property (e) of Lemma~\ref{bddsemilem} as the definition. We do not use this terminology since an almost lattice is the same as a bounded semilattice. Using Lemma~\ref{bddsemilem} and Remark~\ref{cartsec}(iii) one can easily check that if $K$ is spherically complete and $\lvert K^\times\rvert =\bfR_{>0}$ then $M$ is an almost lattice if and only if it is almost isomorphic to a lattice, but in general there exist almost lattices which are not almost isomorphic to a lattice.
\end{rem}

\subsubsection{Index of bounded semilattices}
If $M$ is a semilattice in $V$ then $\bigwedge^i M$ is a semilattice in $\bigwedge^i V$. Furthermore, if $M$ is bounded then it is contained in a lattice $L$ and hence $\bigwedge^i M$ is contained in the lattice $\bigwedge^i L$. In particular, $\bigwedge^i M$ is bounded. If $d=\dim(V)$ then we call $\det(M)=\bigwedge^dM$ the {\em determinant} of $M$. As in the case of seminorms, we define the {\em index} of two bounded semilattices $M,N$ of $V$ to be the ratio of their determinants:
$$[M:N]=\rvert \det(M):\det(N)\rvert=\sup\left\{\rvert \pi\rvert \colon\pi\in K,\pi\det(N)\subseteq\det(M)\right\}.$$

\begin{lem}\label{indexlem}
Let $V$ be a finite-dimensional vector space over a real-valued field $K$, let $M,M'$ be two bounded semilattices of $V$ and let $\lVert \ \rVert $ and $\lVert \ \rVert'$ be the associated norms on $V$. Then,

(i) For any $i$, $\lVert \ \rVert _{\bigwedge^i}$ is the norm associated with $\bigwedge^i M$.

(ii) $[\lVert \ \rVert :\lVert \ \rVert ']=[M:M']^{-1}$.
\end{lem}
\begin{proof}
Since $M$ is an almost unit ball of $\lVert \ \rVert $, it is easy to see that $\lVert \ \rVert _{\bigwedge^i}$ is the maximal norm such that $\lVert v_1\wedge\dots\wedge v_i\rVert _{\bigwedge^i}\le 1$ for any $v_1\..v_i\in M$. Thus the module $\bigwedge^i M$, which is generated by the elements $v_1\wedge\dots\wedge v_i$ with $v_1\..v_i\in M$, is an almost unit ball of $\lVert \ \rVert _{\bigwedge^i}$. This implies (i), and taking $i=\dim(V)$ we reduce (ii) to the one-dimensional case, which is clear.
\end{proof}

\subsection{Content of $\Kcirc$-modules}\label{contsec}

\subsubsection{The definition}
Given a finitely presented $\Kcirc$-module $M$ we represent it as in Corollary~\ref{fpcor} and define the {\em content} of $M$ to be $\cont(M)=\prod_{i=1}^n\lvert\pi_i\rvert$. In general, we set $\cont(M)=\inf_\alp\cont(M_\alp)$, where $M_\alp$ run through all finitely presented subquotients of $M$. Obviously, this is compatible with the definition in the finitely presented case. Note that $\cont(M)=0$ if $M$ is not a torsion module.

\begin{rem}
The content invariant adequately measures the ``size" of a torsion module $M$. In particular, $\cont(M)=1$ if and only if $M$ almost vanishes. Also, if the valuation of $K$ is discrete and $\pi_K$ is a uniformizer then $\cont(M)=\lvert \pi_K\rvert^{{\rm length}(M)}$.
\end{rem}

\subsubsection{Relation to the index}
We will study the content by relating it to index of semilattices.

\begin{lem}\label{contentindex}
Assume that $V$ is a finite-dimensional vector space over a real-valued field $K$ and $L\subseteq M$ are two bounded semilattices of $V$. Then $[M:L]=\cont(M/L)^{-1}$.
\end{lem}
\begin{proof}
If $M$ and $L$ are lattices then the assertion follows by use of Lemma~\ref{twolattices}. In particular, this covers the case of DVR's, so in the sequel we assume that $\rvert K^\times\rvert $ is dense. Choose $0\neq\pi\in\Kcirccirc$. Then it is easy to see that $[M:\pi L]=\pi^{-d}[M:L]$ and $\cont(M/\pi L)=\pi^d\cont(M/L)$, where $d=\dim(V)$. Therefore, it suffices to prove the lemma for $M$ and $\pi L$, and we can assume in the sequel that $L\subseteq\pi M$ for some $\pi$ with $|\pi|<1$.

Let $\pi_1,\pi_2,\dots$ be elements of $\Kcirc$ such that the sequence $\rvert \pi_i\rvert $ strictly increases and tends to 1. By Lemma~\ref{bddsemilem}(d)$\Longleftrightarrow$(e), for each $i\in\bfN$ there exist lattices $L_i$ and $M_i$ such that $$\pi_i L_i\subseteq L\subseteq L_i\subseteq M_i\subseteq M\subseteq\pi_i^{-1}M_i.$$ Since the case of lattices was established, it suffices to check the equalities $$\lim_i\cont(M_i/L_i)=\cont(M/L),\ \ \lim_i[M_i:L_i]=[M:L].$$ The latter follow from the observation that $$\cont(M_i/L_i)\ge\cont(M/L)\ge\cont(\pi^{-1}_iM_i/\pi_i L_i)=\pi_i^{2d}\cont(M_i/L_i)$$
and $$[M_i:L_i]\le [M:L]\le[\pi^{-1}_iM_i:\pi_i L_i]=\pi_i^{-2d}[M_i:L_i].$$
\end{proof}

\subsubsection{Properties}
The following continuity result reduces computation of contents to the finitely generated case.

\begin{lem}\label{contentlem}
If a $\Kcirc$-module $M$ is a filtered union of submodules $M_i$ then $$\cont(M)=\inf_i\cont(M_i)=\lim_i\cont(M_i).$$
\end{lem}
\begin{proof}
Any subquotient of $M_i$ is a subquotient of $M$ hence $\cont(M_i)\ge\cont(M)$. Also, it is easy to see that any finitely presented subquotient of $M$ is a subquotient of each $M_i$ with a large enough $i$. Hence the equalities hold.
\end{proof}

Now we can establish the main property of content, the multiplicativity.

\begin{theor}\label{contlem}
If $0\to M'\to M\to M''\to 0$ is a short exact sequence of $\Kcirc$-modules then
$\cont(M)=\cont(M')\cdot\cont(M'').$
\end{theor}
\begin{proof}
If $M$ is not torsion then either $M'$ or $M''$ is not torsion and hence both $\cont(M)$ and $\cont(M')\cont(M'')$ vanish. So assume that $M$ is torsion. First, we consider the case when $M$ is finitely generated. Fix an epimorphism $L=(\Kcirc)^n\onto M$ and denote its kernel $Q$. Then $L$ is a lattice in $V=K^n$ and since $M=L/Q$ is torsion, $Q$ is a semilattice of $V$. Let $Q'$ be the kernel of the composition $L\to M\to M''$ then $M''=L/Q'$ and $M'=Q'/Q$. Hence the claim follows from the multiplicativity of the index and Lemma~\ref{contentindex}.

Assume now that $M$ is a general torsion module and let $\{M_i\}$ be the family of finitely generated submodules of $M$. Then $M'$ is the filtered union of the submodules $M'_i=M'\cap M_i$ and $M''$ is the filtered union of the modules $M''_i=M_i/M'_i$. By the above case, $\cont(M_i)=\cont(M'_i)\cdot\cont(M''_i)$ and it remains to pass to the limit and use Lemma~\ref{contentlem}.
\end{proof}

\section{Metrization of sheaves}\label{metrsheafsec}

\subsection{Seminormed sheaves}\label{sheavessec}
Throughout Section \ref{sheavessec}, $\calC$ is a site, i.e. a category provided with a Grothendieck topology. In our applications, $\calC$ will be the category associated with a $G$-topological space, more concretely, it will be of the form $X_G$ for an analytic space $X$.

\subsubsection{Quasi-norms}\label{semimetric}
The sup seminorm can be infinite on a non-compact set, so it is technically convenient to introduce the following notion. Let $\obfR_{\ge 0}=\bfR_{\ge 0}\cup\{\infty\}$ be the one-pointed compactification of $\bfR_{\ge 0}$ with addition and multiplication satisfying all natural rules and the rule $0\cdot\infty=0$. {\em Quasi-norms} on abelian groups, rings and modules are defined similarly to seminorms but with the target $\obfR_{\ge 0}$. For example, a quasi-norm $\lvert \ \rvert $ on a ring $A$ is a map $\lvert \ \rvert \:A\to\obfR_{\ge 0}$ such that $\lvert 0\rvert =0$, $\lvert 1\rvert=1$, $\lvert a-b\rvert \le\max(\lvert a\rvert ,\lvert b\rvert )$ and $\lvert ab\rvert \le\lvert a\rvert\cdot\lvert b\rvert$. The material of Section~\ref{normringsec}, including the constructions of quotient and tensor product seminorms, extends to quasi-norms straightforwardly.

\subsubsection{Seminorms on sheaves}\label{sheavesseminormedsec}
Let $\calA$ be a sheaf of abelian groups on $\calC$. By a {\em pre-quasi-norm} $\lVert \ \rVert $ on $\calA$ we mean a family of quasi-norms $\lVert \ \rVert _U$ on $\calA(U)$, where $U$ runs through the objects of $\calC$, such that the restriction maps $\calA(U)\to\calA(V)$ are non-expansive. In the same way one defines quasi-norms on sheaves of rings and quasi-norms on sheaves of modules over a sheaf of rings provided with a pre-quasi-norm. The constructions we describe below for pre-quasi-norms on sheaves of abelian groups hold also for pre-quasi-norms on sheaves of rings and modules.

A pre-quasi-norm is called a {\em quasi-norm} if it satisfies the following locality condition: for any covering $\{U_i\to U\}$ in $\calC$ the equality $\lVert s\rVert_U=\sup_i\lVert s_i\rVert_{U_i}$ holds. A pre-quasi-norm is called {\em locally bounded} if for any $U$ in $\calC$ and $s\in\calA(U)$ there exists a covering $\{U_i\to U\}$ such that $\lVert s_i\rVert_{U_i}<\infty$ for any $i$, where $s_i$ denotes $s\vert_{U_i}$. A locally bounded quasi-norm will be called a {\em seminorm}, and once a seminorm $\lVert \ \rVert $ is fixed we call the pair $\calA=(\calA,\lVert \ \rVert )$ a {\em seminormed sheaf of abelian groups}.

\begin{rem}
(i) Our ad hoc definitions have the following categorical interpretation. If $\lVert \ \rVert $ is a pre-quasi-norm on $\calA$ then the pair $(\calA,\lVert \ \rVert )$ can also be viewed as a presheaf of quasi-normed abelian groups on $\calC$, and this presheaf is a sheaf if and only if $\lVert \ \rVert $ is a quasi-norm.

(ii) If an object $U$ is quasi-compact (i.e. any of its coverings possesses a finite refinement) and $\lVert \ \rVert $ is locally bounded then $\lVert \ \rVert _U$ is a seminorm. In particular, if the subcategory of quasi-compact objects $\calC_c$ is cofinal in $\calC$ then seminorms on $\calA$ can be viewed as sheaves of seminormed abelian groups on $\calC_c$ with the underlying sheaf of abelian groups $\calA$.
\end{rem}

\subsubsection{Sheafification}
For any pre-quasi-norm $\lVert \ \rVert '$ on a sheaf of abelian groups $\calA$ we define the {\em sheafification} $\lVert \ \rVert =\alp(\lVert \ \rVert ')$ by the rule $$\lVert s\rVert _U=\inf_{\{U_i\to U\}}\sup_{i\in I}\lVert s_i\rVert '_{U_i},$$ where $U$ is an object of $\calC$, $s\in\calA(U)$ and the infimum is over all coverings of $U$.

\begin{lem}\label{sheafiflem}
Keep the above notation. Then $\lVert \ \rVert $ is a quasi-norm and it is the maximal quasi-norm dominated by $\lVert \ \rVert '$. In addition, if $\lVert \ \rVert '$ is locally bounded then $\lVert \ \rVert $ is a seminorm.
\end{lem}
\begin{proof}
The fact that $\lVert \ \rVert $ is a quasi-norm follows easily from the transitivity of coverings. The other assertions are obvious.
\end{proof}

\begin{rem}
The universal property from the first part of the lemma justifies the notion ``sheafification". In fact, the same argument shows that $(\calA,\lVert \ \rVert ')\to(\calA,\lVert \ \rVert )$ is the universal (non-expansive) map from $(\calA,\lVert \ \rVert ')$ to a sheaf of abelian groups with a quasi-norm. Thus, $(\calA,\lVert \ \rVert )$ is even the sheafification of $(\calA,\lVert \ \rVert ')$ as a sheaf of quasi-normed abelian groups.
\end{rem}

\subsubsection{Operations on seminormed sheaves}\label{opersec}
Let us extend various operations, including quotients, tensor products, symmetric powers and exterior powers, to seminormed sheaves. This is done in two stages. First one works with sections over each $U$ separately. This produces a sheaf with a pre-quasi-norm, which is easily seen to be locally bounded. Then one sheafifies this pre-quasi-norm if needed. As in the case of usual sheaves of rings and modules, the second step is needed in the case of constructions that are not compatible with limits, such as colimits (including quotients), tensor products, etc.

For the sake of illustration, let us work this out in the case of quotients. Assume that $\phi\:\calA\to\calB$ is an epimorphism of sheaves of abelian groups and $\calA$ is provided with a seminorm $\lVert \ \rVert$. First, we endow $\calB$ with the quotient pre-quasi-norm $\lVert\ \rVert''$, i.e. for each $U$ in $\calC$ the quasi-norm $\lVert \ \rVert ''_U$ on $\calB(U)$ is the maximal quasi-norm making the map $(\calA(U),\lVert \ \rVert _U)\to(\calB(U),\lVert \ \rVert ''_U)$ non-expansive. Then the quotient seminorm $\lVert \ \rVert '$ on $\calB$ is defined to be the sheafification of $\lVert \ \rVert ''$. In particular, $\lVert \ \rVert '$ is the maximal seminorm on $\calB$ making the homomorphism $(\calA,\lVert \ \rVert )\stackrel\phi\to(\calB,\lVert \ \rVert ')$ non-expansive.

\begin{rem}
Provide $\calK=\Ker(\phi)$ with any seminorm $\lVert \ \rVert _\calK$ making the embedding $\calK\into\calA$ non-expansive. For example, one can take the restriction of $\lVert \ \rVert $ on $\calK$. Then the universal property characterizing $\phi$ implies that $(\calB,\lVert \ \rVert ')$ is the cokernel of the map of the seminormed sheaves of abelian groups $(\calK,\lVert \ \rVert _\calK)\to(\calA,\lVert \ \rVert )$.
\end{rem}

\subsubsection{Pushforwards}\label{pushsec}
Recall that a morphism $f\:\calC'\to\calC$ of sites is a functor $\calF\:\calC\to\calC'$ that satisfies certain properties, see \cite[Tag:00X0]{stacks}. If $\calA'$ is a sheaf of abelian groups on $\calC'$ with a pre-quasi-norm $\lVert \ \rVert '$ then we endow $\calA=f_*(\calA')$ with the {\em pushforward pre-quasi-norm} $\lVert \ \rVert $ as follows: for any $U$ in $\calC$ and $U'=\calF(U)$ we have that $\calA(U)=\calA'(U')$ and we set $\lVert \ \rVert _U=\lVert \ \rVert _{U'}$. Clearly, if $\lVert \ \rVert '$ is a quasi-norm then $\lVert \ \rVert $ is also a quasi-norm, but the pushforward of a seminorm can be unbounded.

\begin{rem}
An important case when the pushforward of a seminorm is a seminorm is when $f$ corresponds to a proper map of topological spaces. This fact will not be used, so we do not check it here.
\end{rem}

\subsubsection{Pullbacks}
Assume, now, that $\calA$ is a sheaf of abelian groups on $\calC$ with a pre-quasi-norm $\lVert \ \rVert $ and $\calA'=f^{-1}(\calA)$. Then $\calA'(U')=\colim_{U'\to\calF(U)}\calA(U)$ and we define $\lVert \ \rVert '_{U'}$ to be the colimit of the quasi-norms $\lVert \ \rVert _U$. If $\lVert \ \rVert $ is a quasi-norm, it still may happen that the pre-quasi-norm $\lVert \ \rVert '$ is not a quasi-norm, so we define $f^{-1}(\lVert \ \rVert )$ to be the sheafification of $\lVert \ \rVert '$. On the positive side we note that local boundedness is preserved by the pullback, so if $\lVert \ \rVert $ is a seminorm then $f^{-1}(\lVert \ \rVert )$ is a seminorm too.

\subsection{Points and stalks of seminorms}
In this section we study stalks of seminorms at points and describe seminorms that are fully controlled by stalks. For concreteness, we usually consider sheaves of abelian groups but the cases of rings and modules are similar.

\subsubsection{Points of sites}
We use the terminology of \cite[Tag:00Y3]{stacks} when working with points of sites. Recall that a point $x$ of $\calC$ is a functor $x\:\calC\to\mathrm{Sets}$ that satisfies certain properties and given an object $U$ of $\calC$ an elements $f\in x(U)$ is interpreted as a morphism $f\:x\to U$. The set of isomorphism classes of points of $\calC$ will be denoted $\lvert \calC\rvert $.

\subsubsection{Stalks}
Assume that $\lVert \ \rVert $ is a pre-quasi-norm on a sheaf of abelian groups $\calA$. Given a point $x\in\lvert \calC\rvert $, we endow the stalk $\calA_x=\colim_{x\to U}\calA(U)$ at $x$ with the {\em stalk quasi-norm} $\lVert \ \rVert _x$ defined as follows: if $x\to V$ and $s\in\calA(V)$ then $\lVert s\rVert _x=\inf_{x\to U\to V}\lVert s\rVert _U$. The following result follows by unveiling the definitions.

\begin{lem}
Assume that $\lVert \ \rVert $ is a pre-quasi-norm on a sheaf of abelian groups $\calA$. Then,

(i) The stalks of $\lVert \ \rVert $ and of the sheafification $\alp(\lVert \ \rVert )$ coincide.

(ii) If $\lVert \ \rVert $ is locally bounded then the stalks $\lVert \ \rVert _x$ are seminorms.
\end{lem}

Also, Lemma~\ref{colimlem} implies the following result.

\begin{lem}\label{stalkcompat}
Stalks of seminormed sheaves of rings or modules are compatible with quotients, tensor products, symmetric and exterior powers. For example, given a seminormed sheaf of rings $\calA$, seminormed $\calA$-modules $\calM$, $\calN$ and $\calL=\calM\otimes_\calA\calN$, and a point $x\in\lvert \calC\rvert $, there is a natural isomorphism of seminormed $\calA_x$-modules $\calM_x\otimes_{\calA_x}\calN_x=\calL_x$.
\end{lem}

\subsubsection{Semicontinuity}\label{semicontsec}
Let $\calP\subseteq\lvert \calC\rvert $ be a subset. A seminorm $\lVert \ \rVert $ on $\calA$ induces the set of stalk seminorms $\{\lVert \ \rVert _x\}_{x\in\calP}$ which satisfies the following semicontinuity condition: if $x\in\calP$ and $s_x\in\calA_x$ then for any $\veps>0$ there exists a morphism $x\to U$ such that $s_x$ is induced from $s\in\calA(U)$ and $\lVert s\rVert _U<\lVert s\rVert _x+\veps$, in particular, $\lVert s\rVert _y<\lVert s\rVert _x+\veps$ for any $y\to U$. Any family of seminorms satisfying this condition will be called {\em upper semicontinuous}.

\begin{exam}\label{semicontexam}
If $\calC$ is the site of open subsets of a topological space $X$ then a family of seminorms $\{\lVert \ \rVert_x\}_{x\in X}$ is upper semicontinuous if and only if for any open $U$ and a section $s\in\calA(U)$ the function $\lVert s\rVert \:U\to\bfR$ sending $x$ to $\lVert s\rVert _x$ is upper semicontinuous.
\end{exam}

\subsubsection{$\calP$-seminorms}\label{Psemisec}
Let $\calP\subseteq\lvert \calC\rvert $. Given a seminorm $\lVert \ \rVert $ on $\calA$ set $\lVert s\rVert _{\calP,U}=\sup_{x\in\calP,x\to U}\lVert s_x\rVert _x$. Clearly, $\lVert \ \rVert _\calP$ is a seminorm on $\calA$, and we say that $\lVert \ \rVert $ is a {\em $\calP$-seminorm} if $\lVert \ \rVert =\lVert \ \rVert _\calP$.

\begin{lem}\label{Psemilem}
Keep the above notation then,

(i) $\lVert \ \rVert _\calP$ is the maximal $\calP$-seminorm on $\calA$ that is dominated by $\lVert \ \rVert $. The stalks of $\lVert \ \rVert _\calP$ and $\lVert \ \rVert $ at any point of $\calP$ coincide.

(ii) There is a natural bijection between $\calP$-seminorms on $\calA$ and upper semicontinuous families of seminorms on the stalks of $\calA$ at the points of $\calP$.
\end{lem}
\begin{proof}
The first claim is clear. We noticed earlier that any seminorm gives rise to an upper semicontinuous family. Conversely, to any family of seminorms $\{\lVert \ \rVert _x\}$ on the stalks at the points $x\in\calP$ we assign the sup seminorm $\lVert s\rVert _{\calP,U}=\sup_{x\in\calP,x\to U}\lVert s_x\rVert _x$. To prove (ii) we should prove that if the family is upper semicontinuous then the stalk $\lVert \ \rVert _{\calP,x}$ coincides with the original seminorm $\lVert \ \rVert _x$. Clearly, $\lVert \ \rVert _x\le\lVert \ \rVert _{\calP,x}$. In addition, if $\lVert s\rVert _x<r$ then there exists a morphism $x\to U$ such that $s$ is induced from an element $s_U\in\calA(U)$ and $\lVert s_U\rVert _y<r$ for any $y\to U$. Then $\lVert s\rVert _{\calP,x}\le\lVert s_U\rVert _{\calP,U}<r$, thus proving (ii).
\end{proof}

\begin{rem}\label{prem}
The class of $\calP$-seminorms is not preserved under quotients and other operations. One can easily construct such examples already when $\calC$ is associated to the topological space $X=\{\eta,s\}$ consisting of an open point $\eta$ and a closed point $s$, and $\calP=\{\eta\}$.
\end{rem}

\subsubsection{Conservative families}
As one might expect, if $\calC$ possesses enough points then the theory of seminorms can be developed in terms of stalks. More concretely, assume that $\calC$ possesses a conservative family of points $\calP$. Then seminorms and operations on them are completely controlled by the $\calP$-stalks. In particular, one can define all operations on seminormed sheaves stalkwise. This follows from the following result.

\begin{theor}\label{conservlem}
Assume that $\calP$ is a conservative family of points of a site $\calC$ and $\calA$ is a sheaf of abelian groups (resp. rings, resp. modules over a seminormed sheaf of rings) on $\calC$. Then any seminorm $\lVert \ \rVert $ on $\calA$ is a $\calP$-seminorm. In particular, there is a natural bijective correspondence between seminorms on $\calA$ and upper semicontinuous families of seminorms on the stalks of $\calA$ at the points of $\calP$.
\end{theor}
\begin{proof}
All cases are proved similarly, so assume that $\calA$ is a sheaf of abelian groups. Fix $r>0$ and let $\calA^\di_r\subseteq\calA$ denote the ball of radius $r$; it is the presheaf such that $\calA^\di_r(U)$ is the set of all elements $s\in\calA(U)$ with $\lVert s\rVert _U\le r$. The locality condition satisfied by $\lVert \ \rVert $ implies that $\calA^\di_r$ is, in fact, a sheaf. Note that the stalk of $\calA^\di_r$ at a point $x$ coincides with the ball $(\calA_x)^\di_r$.

Define a subsheaf $\calA^\di_{\calP,r}\stackrel {i_r}\into\calA^\di_r$ by the condition that $s\in\calA^\di_{\calP,r}(U)$ if for any $x\to U$ with $x\in\calP$ the inequality $\lVert s\rVert _x\le r$ holds. Note that $\calA^\di_{\calP,r}$ is the $r$-ball of the seminorm $\lVert \ \rVert _\calP$ and the stalk of $\calA^\di_{\calP,r}$ at a point $x\in\calP$ equals to $(\calA_x)^\di_r$. So, the embedding of sheaves $i_r$ induces isomorphisms of stalks at the points of $\calP$, and using that $\calP$ is conservative we obtain that $i_r$ is an equality. Thus, the balls of the seminorms $\lVert \ \rVert $ and $\lVert \ \rVert _\calP$ coincide and hence the seminorms coincide. We proved that $\lVert \ \rVert $ is a $\calP$-seminorm, and the second claim follows from Lemma~\ref{Psemilem}(ii).
\end{proof}

\subsection{Sheaves on analytic spaces}\label{sheavesansec}
In this section we study seminorms on $\calO_{X_G}$-modules, where $X$ is an $H$-strict $k$-analytic space and $X_G$ denotes the site of $H$-strict analytic domains in $X$.

\subsubsection{Analytic points}
Note that any point $x\in X$ defines a point of $X_G$, so we can view $\lvert X\rvert $ as a subset of $\lvert X_G\rvert $. A point of $X_G$ is called {\em analytic} if its isomorphism class lies in $\lvert X\rvert $. A seminorm on a sheaf of abelian group on $X_G$ is called {\em $G$-analytic} if it is an $\lvert X\rvert $-seminorm in the sense of \S\ref{Psemisec}. For example, the spectral seminorm $\lvert \ \rvert $ on $\calO_{X_G}$ satisfies $\lvert f\rvert _U=\sup_{x\in U}\lvert f(x)\rvert $, i.e. it is $G$-analytic.

\begin{exam}\label{nonanalyticexam}
(i) A typical example of a non-analytic point $z$ on a unit disc $E=\calM(k\{t\})$ is as follows: the family of neighborhoods of $z$ is the set of all domains $U\subseteq E$ that contain an open annulus $r<\lvert t\rvert <1$ for some $r$. One can view $z$ as the maximal point with $\lvert t\rvert _z<1$. In Huber adic geometry (see \cite{hub}) it corresponds to a valuation of height two such that $\lvert t\rvert _z<1$ and $r<\lvert t\rvert _z$ for any real $r<1$.

(ii) Using $z$ one easily constructs a seminorm $\lVert \ \rVert $ on $\calO_{X_G}$ which is not $G$-analytic. For example, set $\lVert f\rVert _U=0$ if $z\notin U$ and $\lVert f\rVert _U=\lVert f(p)\rVert $ otherwise, where $p$ is the maximal point of $E$. In fact, all stalks of $\lVert \ \rVert $, excluding the stalk at $z$, are zero.
\end{exam}

For the sake of simplicity, we only use analytic points in Section~\ref{sheavesansec} (and in the main part of this paper). Arbitrary points of $X_G$ will be described in Section~\ref{Gsec}, and then we will extend to them some results of this section.

\subsubsection{Analytic seminorms}
Let $\calA$ be a sheaf of abelian groups on $X_G$ with a seminorm $\lVert \ \rVert $. For any domain $U$ and a section $s\in\calA(U)$ consider the function $\lVert s\rVert \:U\to\bfR_{\ge 0}$ that sends $x\in U$ to $\lVert s\rVert _x$. As we saw in \S\ref{semicontsec}, this function is upper $G$-semicontinuous, i.e. if $x\in X$ satisfies $\lVert s\rVert _x<r$ then there exists an analytic domain $V$ such that $x\in V\subseteq U$ and $\lVert s\rVert _y<r$ for any $y\in V$. We say that the seminorm $\lVert \ \rVert $ is {\em analytic} if it is $G$-analytic and all functions $\lVert s\rVert $ are upper semicontinuous with respect to the usual topology of $X$, i.e. in the above situation $V$ can be chosen to be a neighborhood of $x$.

\begin{rem}\label{ansemrem}
(i) Both conditions are essential in the definition of analytic seminorms. For example, the seminorm $\lVert \ \rVert $ in Example~\ref{nonanalyticexam} has zero stalks at all analytic points, hence all functions $\lVert s\rVert $ vanish though the seminorm is not analytic.

(ii) Assume that $X=\cup_i X_i$ is an admissible covering. Then a function $X\to\bfR$ is upper semicontinuous if and only if its restrictions to $X_i$ are upper semicontinuous. Therefore, a seminorm on $\calA$ is analytic if and only if its restrictions onto $\calA\vert_{X_i}$ are analytic, i.e. analyticity is a $G$-local condition.

(iii) The following simple observation will not be used, so we omit a justification. A seminorm $\lVert \ \rVert $ is analytic if and only if for any affinoid domain $U\subseteq X$, a section $s\in\calA(U)$, a point $x\in U$ and $r>\lVert s\rVert _x$ there exists a neighborhood $V$ of $x$ in $U$ such that $\lVert s\rVert _V<r$. This can also be reformulated using stalks of $\calA$ in the usual topology: $\lVert \ \rVert $ is analytic if and only if for any affinoid domain $U\subseteq X$ and a point $x\in U$ the map $\calA_{\lvert U\rvert ,x}\to\calA_x$ is an isometry with respect to the stalk seminorms of $\lVert \ \rVert $, where $\calA_{\lvert U\rvert }$ denotes the restriction of $\calA$ onto the topological space $\lvert U\rvert $.
\end{rem}

\subsubsection{Local rings of $\calO_{X_G}$}\label{localringsec}
In Berkovich geometry, one usually works with local rings $\calO_{X,x}$ of good spaces and their residue fields $\kappa(x)=\calO_{X,x}/m_x$. We will also need the local rings $\calO_{X_G,x}$ and their residue fields $\kappa_G(x)=\calO_{X_G,x}/m_{G,x}$.

\begin{lem}\label{stalklem}
Let $X$ be an analytic space with a point $x\in X$. Then

(i) $(\calO_{X_G,x},\lvert \ \rvert _x)$ is a seminormed local ring whose maximal ideal $m_{G,x}$ is the kernel of $\lvert \ \rvert _x$.

(ii) The residue norm on $\kappa_G(x)$ is a valuation and the completion coincides with $\calH(x)$. In particular, $\calH(x)$ is the completion of $(\calO_{X_G,x},\lvert \ \rvert _x)$.
\end{lem}
\begin{proof}
Note that $\calO_{X_G,x}$ is the filtered colimit of the rings $\calO_{V,x}$, where $V$ is an affinoid domain containing $x$, and the transition maps are local. Hence (i) follows from the observation that the maximal ideal of $\calO_{V,x}$ is the kernel of the restriction of $\lvert \ \rvert _x$. Furthermore, this implies that the residue field $\kappa_G(x)$ of $\calO_{X_G,x}$ is the filtered union of the residue fields of $\calO_{V,x}$. The latter are dense subfields of $\calH(x)$ hence the same is true for $\kappa_G(x)$. It remains to note that the completion of $(\calO_{X_G,x},\lvert \ \rvert _x)$ is isomorphic to the completion of its quotient by the kernel of the seminorm.
\end{proof}

\subsubsection{Seminormed $\calO_{X_G}$-modules}
Assume that $(\calF,\lVert \ \rVert )$ is a seminormed $\calO_{X_G}$-module. For any analytic point $x$ the kernel of $\lVert \ \rVert_x$ contains $m_{G,x}\calF_x$ and hence $\lVert \ \rVert_x$ is induced from the residue seminorm on the fiber $\calF(x)=\calF_x/m_{G,x}\calF_x$. We call the latter the {\em fiber seminorm} and denote it $\lVert \ \rVert_{(x)}$. The completion of $\calF(x)$ with respect to $\lVert \ \rVert_{(x)}$ is a Banach $\calH(x)$-space that will be called the {\em completed fiber} and denoted $\wh{\calF(x)}$.

\subsubsection{Pullbacks}\label{pullsec}
If $f\:Y\to X$ is morphism of Berkovich spaces and $\calF$ is a seminormed $\calO_{X_G}$-module then we define the pullback as $f^*\calF=f^{-1}\calF\otimes_{f^{-1}\calO_{X_G}}\calO_{Y_G}$, where both $f^{-1}$ and the tensor product are taken in the sense of seminormed sheaves.

\subsubsection{The case of invertible sheaves}
For illustration, let us describe $G$-analytic $\calO_{X_G}$-seminorms on an invertible module $\calF$. Such a seminorm $\lVert \ \rVert $ is determined by a $G$-semicontinuous family of seminorms $\lVert \ \rVert _x$ for $x\in X$. Sending a seminorm to its fiber establishes a bijection between $\calO_{X_G,x}$-seminorms on $\calF_x$ and $\kappa_G(x)$-seminorms on the one-dimensional vector space $\calF(x)$. Finally, if $s_x$ is a basis of $\calF(x)$ then sending a seminorm to its value on $s_x$ provides a parametrization of $\kappa_G(x)$-seminorms on $\calF(x)$ by numbers $r\in\bfR_{\ge 0}$.

\begin{lem}\label{rankonelem}
Assume that $X$ is an analytic space and $\calF$ is a free $\calO_{X_G}$-module of rank one with basis $s$. Then the correspondence $\lVert \ \rVert \mapsto \lVert s\rVert$ establishes a bijection between $G$-analytic $\calO_{X_G}$-seminorms on $\calF$ and upper $G$-semicontinuous functions $r\:X\to\bfR_{\ge 0}$. Furthermore, a seminorm is analytic if and only if the function $\lVert s\rVert$ is upper semicontinuous.
\end{lem}
\begin{proof}
For any seminorm $\lVert\ \rVert$ the function $\lVert s\rVert$ is upper $G$-semicontinuous and the argument above the lemma shows that it determines a $G$-analytic seminorm uniquely. Conversely, given a $G$-semicontinuous function $r$ consider the seminorm $\lVert\ \rVert_x$ on $\calF_x$ such that $\lVert s\rVert_x=r(x)$. It is easy to see that the family $\{\lVert \ \rVert _x\}_{x\in X}$ is $G$-semicontinuous and hence gives rise to a $G$-analytic seminorm by Lemma~\ref{Psemilem}(ii). It remains to show that if $\lVert s\rVert $ is
upper semicontinuous then the seminorm is analytic. Indeed, a section $t\in\calF(U)$ is of the form $fs$ with $f\in\calO_{X_G}(U)$, and using that the function $\lvert f\rvert $ is continuous we obtain that $\lVert t\rVert =\lvert f\rvert \cdot\lVert s\rVert $ is upper semicontinuous.
\end{proof}

\subsubsection{Analytic $\calO_{X_G}$-seminorms on coherent sheaves}
For coherent sheaves one can describe analyticity in terms of fiber seminorms of $X$ and $X_G$.

\begin{lem}\label{cohlem}
Assume that $X$ is a good analytic space with a coherent $\calO_X$-module $\calF$ and let $\calF_G$ denote the associated coherent $\calO_{X_G}$-module on $X_G$. Let $\lVert \ \rVert $ be an $\calO_{X_G}$-seminorm on $\calF_G$ and for any $x\in X$ endow stalks $\calF_x$ and $\calF_{G,x}$ and the fibers $\calF(x)$ and $\calF_G(x)$ with the stalk and the fiber seminorms of $\lVert \ \rVert $. Then the following conditions are equivalent:

(i) $\lVert \ \rVert $ is analytic.

(ii) The map $\calF_x\to\calF_{G,x}$ is an isometry for any $x\in X$.

(iii) The map $\calF(x)\to\calF_G(x)$ is an isometry for any $x\in X$

(iv) The map $\calF(x)\otimes_{\kappa(x)}\kappa_G(x)\to\calF_G(x)$ is an isometric isomorphism for any $x\in X$.
\end{lem}
\begin{proof}
(iii)$\Longleftrightarrow$(iv) Set $K=\kappa(x)$, $K'=\kappa_G(x)$, $V=\calF(x)$ and endow $V'=V\otimes_KK'$ with the tensor seminorm. The map $h\:V'\to\calF_G(x)$ is an isomorphism since $\calF$ is coherent, so (iv) is satisfied if and only if $h$ is an isometry. Note also that the inclusion $K\into K'$ is an isometry, hence $V\into V'$ is an isometry by \cite[Lemma~3.1]{Poineau@angelique}. In particular, if $h$ is an isometry then $g\:V\to\calF_G(x)$ is an isometry, i.e. (iv)$\implies$(iii). Conversely, if $g$ is an isometry then $h$ is an isometry because $h$ is non-expansive and $V$ is dense in $V'$.

(ii)$\Longleftrightarrow$(iii) This follows from the fact that the seminorm of $\calF_x$ is induced from the seminorm of $\calF(x)$, and similarly for $\calF_{G,x}$ and $\calF_G(x)$.

(i)$\implies$(ii) We should prove that if $x\in X$ and $s\in\calF_x$ then the image  $s_G\in\calF_{G,x}$ of $s$ satisfies $\lVert s_G\rVert _x\ge\lVert s\rVert $. Choose a neighborhood $U$ of $x$ such that $s$ is defined on $U$. Since the seminorm is analytic, for any $r>\lVert s_G\rVert _x$ there exists a neighborhood $V\subseteq U$ of $x$ such that $\lVert s\rVert _y<r$ for any $y\in V$. Since $\lVert \ \rVert $ is $G$-analytic, this implies that $\lVert s\rVert _V<r$ and hence $\lVert s\rVert _x<r$.

(ii)$\implies$(i) First, we claim that (ii) holds for the restriction of $\calF_G$ onto any good domain $U\subseteq X$. Let $\calG$ denote the coherent $\calO_U$-module $\calF_G\vert_U$. Choose any $x\in U$ and let $\kappa_U(x)$ be the residue field of $x$ in $U$, in particular, $\kappa(x)\subseteq\kappa_U(x)\subseteq\kappa_G(x)$. Since (ii) and (iii) are equivalent, we should check that the  map $h\:\calG(x)\to\calG_G(x)=\calF_G(x)$ is an isometry. This follows easily from the fact that $\calG(x)=\calF(x)\otimes_{\kappa(x)}\kappa_U(x)$ as vector spaces, $h$ is non-expansive and $\calF(x)\to\calF_G(x)$ is an isometry.

Now, let us prove that $\lVert \ \rVert $ is analytic. Let $U\subseteq X$ be an affinoid domain and $s\in\calF_G(U)$. We claim that the function $\lVert s\rVert $ is upper semicontinous. It suffices to check this for affinoid domains in $U$ hence we can assume that $U$ is affinoid. Fix $x\in U$ and let $r>\lVert s\rVert _x$. We showed that $\calG_x\to\calF_{G,x}$ is an isometry, where $\calG=\calF_G\vert_U$, hence there exists a neighborhood $V\subseteq U$ of $x$ with $\lVert s\rVert _V<r$. In particular, $\lVert s\rVert _y<r$ for any $y\in V$ and hence $\lVert s\rVert$ is upper semicontinous at $x$.

It remains to show that $\lVert \ \rVert $ is $G$-analytic. If this is not so then $\lVert s\rVert _U>r>\sup_{x\in U}\lVert s\rVert _x$ for some choice of $U$, $s\in\calA(U)$ and $r$. If $U=\cup_i U_i$ is an admissible affinoid covering then $\lVert s\rVert _{U_i}>r$ for some $i$, and replacing $U$ with $U_i$ we can assume that $U$ is affinoid. By the same argument as above, for any point $x\in U$ there exists an open neighborhood $V_x\subseteq U$ of $x$ with $\lVert s\rVert _{V_x}<r$. Since $U=\cup_x V_x$ is an admissible covering, we obtain that $\lVert s\rVert _U\le\sup_x\lVert s\rVert _{V_x}<r$. The contradiction concludes the proof.
\end{proof}

\subsubsection{Operations}
Finally, let us study when the property of a seminorm to be analytic is preserved by various operations. Certainly, some restrictions should be imposed (see Remark~\ref{prem}), so we only consider the coherent case, which will be used later.

\begin{lem}\label{ancohlem}
Analyticity of $\calO_{X_G}$-seminorms on coherent sheaves is preserved under the following operations: quotients, tensor products, symmetric and exterior powers.
\end{lem}
\begin{proof}
We consider the case of quotients. The tensor products are dealt with similarly and the other cases follow. So, let $\calM\onto\calN$ be a surjection of coherent $\calO_{X_G}$-modules and let $\lVert \ \rVert _\calM$ be an analytic $\calO_{X_G}$-seminorm on $\calM$. We should prove that the quotient seminorm $\lVert \ \rVert _\calN$ is analytic. By Remark~\ref{ansemrem}(ii), the question is $G$-local on $X$ hence we can assume that $X$ is good.

In the sequel, $\calM$ and $\calN$ denote the $\calO_X$-sheaves, while $\calM_G$ and $\calN_G$ denote the associated $\calO_{X_G}$-sheaves. Set $K=\kappa(x)$ and $K'=\kappa_G(x)$, and endow $U=\calM(x)$, $U'=\calM_G(x)$, $V=\calN(x)$ and $V'=\calN_G(x)$ with the fiber seminorms of $\lVert \ \rVert _\calM$ and $\lVert \ \rVert _\calN$. In particular, we have surjections $f\:U\onto V$ and $f'=f\otimes_KK'\:U'\onto V'$ such that the seminorms on the targets are the quotient seminorms.

By the equivalence (i)$\Longleftrightarrow$(iv) in Lemma~\ref{cohlem}, the isomorphism $U\otimes_KK'\toisom U'$ is an isometry and it suffices to prove that the isomorphism $h\:V\otimes_KK'\toisom V'$ of seminormed vector spaces is an isometry. By definition, $\lVert \ \rVert _V$ is the maximal seminorm such that $\lVert f(u)\rVert _V\le\lVert u\rVert _U$ for any $u\in U$, hence the seminorm on $V\otimes_KK'$ is the maximal one such that $\lVert f(u)\otimes a\rVert \le\lvert a\rvert \cdot\lVert u\rVert _U$ for any $u\in U$ and $a\in K'$. In the same way, the seminorm on $V'$ is the maximal one such that $\lVert f'(u\otimes a)\rVert \le\lvert a\rvert \cdot\lVert u\rVert _U$. Since $h(f(u)\otimes a)=f'(u\otimes a)$, the seminorms match and $h$ is an isometry.
\end{proof}

\section{Differentials of seminormed rings}\label{difsec}

\subsection{K\"ahler seminorms}\label{banachdifsec}

\subsubsection{The definition}\label{omeganormsec}
Given a seminormed ring $\calB$ and a homomorphism of rings $\phi\:A\to\calB$ we equip $\Omega_{\calB/A}$ with the {\em K\"ahler seminorm} $\lVert \ \rVert_\Omega$ given by the formula $$\lVert x\rVert_\Omega=\inf_{x=\sum c_idb_i}\max_i\lvert c_i\rvert_\calB\lvert b_i\rvert_\calB,$$ where $x\in\Omega_{\calB/A}$ and the infimum is over all representations of $x$ as $\sum c_id_{\calB/A}(b_i)$ with $c_i,b_i\in\calB$. In fact, we will always work with $\|\ \|_\Omega$ when $\phi$ is a non-expansive homomorphism of seminormed rings, but its independence of the seminorm of $A$ will be used.

\subsubsection{Universal properties}
Both the K\"ahler seminorm and the seminormed module $(\Omega_{\calB/\calA},\lVert \ \rVert_\Omega)$ can be characterized by appropriate universal properties.

\begin{lem}\label{universemilem}
Assume that $\calA\to\calB$ is a non-expansive homomorphism of seminormed rings. Then,

(i) $\lVert \ \rVert_\Omega$ is the maximal $\calB$-seminorm that makes $d_{\calB/\calA}\:\calB\to\Omega_{\calB/\calA}$ a non-expansive $\calA$-homomorphism.

(ii) $d_{\calB/\calA}$ is the universal non-expansive $\calA$-derivation of $\calB$ with values in a seminormed $\calB$-module:
$$\Hom_{\calB,\nonexp}(\Omega_{\calB/\calA},M)\toisom \Der_{\calA,\nonexp}(\calB,M)$$
for any seminormed $\calB$-module $M$.
\end{lem}
\begin{proof}
The first claim is obvious. In (ii), we should only prove that any non-expansive $\calA$-derivation $d\:\calB\to M$ with values in a seminormed $\calB$-module $M$ factors into a composition of $d_{\calB/\calA}$ and a non-expansive homomorphism $h\:\Omega_{\calB/\calA}\to M$. By the usual universal property, we have a unique such factoring with $h$ being a homomorphism of modules, and it remains to show that $h$ is non-expansive. Let $\lVert \ \rVert '_\Omega$ be defined by $\lVert x\rVert '_\Omega=\max(\lVert x\rVert_\Omega,\lVert h(x)\rVert_M)$. Then it immediately follows that $(\Omega_{\calB/\calA},\lVert \ \rVert '_\Omega)$ is a seminormed $\calB$-module and $d_{\calB/\calA}$ is non-expansive with respect to $\lVert \ \rVert '_\Omega$. So, $\lVert \ \rVert_\Omega=\lVert \ \rVert '_\Omega$ by (i), and hence $h$ is non-expansive.
\end{proof}

\subsubsection{An alternative definition}
Similarly to the case of usual rings, one can define the seminormed module $\Omega_{B/A}$ in terms of the kernel of $B\otimes_AB\to B$. The proof reduces to repeating the classical argument (see \cite[Section 26.C]{Matsumura}) and checking that all relevant maps are non-expansive.

\begin{lem}\label{omegaalter}
Assume that $f\:\calA\to\calB$ is a non-expansive homomorphism of seminormed rings and $I$ is the kernel of the induced homomorphism of seminormed rings $\calB\otimes_\calA\calB\to\calB$. Then the classical isomorphism $\phi\:\Omega_{\calB/\calA}\toisom I/I^2$ is an isometry.
\end{lem}
\begin{proof}
Recall that $\phi$ is induced by the derivation $d\:\calB\to I/I^2$ given by $db=b\otimes 1-1\otimes b$. By Lemma~\ref{universemilem}(ii), it suffices to show that any non-expansive $\calA$-derivation $\partial\:\calB\to M$ factors uniquely into a composition of $d$ and a non-expansive homomorphism of $\calB$-modules $h\:I/I^2\to M$. In fact, a homomorphism $h$ exists and is unique by the classical theory, so we should only check that it is non-expansive.

Let $\calC=\calB*M$ denote the $\calB$-module $\calB\oplus M$ provided with the multiplication $(b,m)(b',m')=(bb',bm'+b'm)$ and the seminorm $|(b,m)|_\calC=\max(|b|_\calB,\|m\|_M)$. The $\calA$-bilinear map $\calB\times \calB\to\calB*M$ sending $(b,b')$ to $(bb',b\partial(b'))$ is non-expansive, hence it induces a non-expansive homomorphism $\lam\:\calB\otimes_\calA\calB\to\calB*M$. By the classical argument, $\lam$ vanishes on $I^2$ and takes $I$ to $M$, in particular, it induces a non-expansive homomorphism $h\:I/I^2\to M$. It remains to notice that $h\circ d=\partial$.
\end{proof}

\subsection{Basic properties of K\"ahler seminorms}

\subsubsection{Fundamental sequences}
First and second fundamental sequences extend to the context of seminormed rings.

\begin{lem}\label{fundseclem}
Assume that $\calA\to\calB\to\calC$ are non-expansive homomorphisms of seminormed rings. Then,

(i) The maps of the first fundamental sequence $$\Omega_{\calB/\calA}\otimes_\calB\calC\stackrel g\to\Omega_{\calC/\calA}\stackrel f\to\Omega_{\calC/\calB}\to 0$$ are non-expansive and $f$ is strictly admissible.

(ii) If the homomorphism $\phi\:\calB\to\calC$ is onto and $J$ is its kernel then the maps of the second fundamental sequence
$$J/J^2\to\Omega_{\calB/\calA}\otimes_\calB\calC\stackrel g\to\Omega_{\calC/\calA}\to 0$$
are non-expansive. Furthermore, if $\calB\to\calC$ is strictly admissible then $g$ is strictly admissible.
\end{lem}
\begin{proof}
This directly follows from the definition of K\"ahler seminorms. For example, let us check the second assertion in (ii). We should prove that the quotient seminorm does not exceed the K\"ahler seminorm of $\Omega_{\calC/\calA}$. The latter is the maximal seminorm satisfying the inequalities $\lVert dc\rVert \le\lvert c\rvert _\calC$. For any $r>\lvert c\rvert _\calC$ we can find $b\in\phi^{-1}(c)$ such that $\lvert b\rvert _\calB<r$. Therefore $\lVert db\rVert _\Omega<r$ and we obtain that the quotient seminorm of $dc$ does not exceed $r$.
\end{proof}

\begin{rem}
Even if the map $g\:\Omega_{\calB/\calA}\otimes_\calB\calC\to\Omega_{\calC/\calA}$ is an isomorphism, it does not have to be an isometry. Simple examples of this type are obtained when $\calB\to\calC$ is an isomorphism but not an isometry.
\end{rem}

\begin{cor}\label{fundcor}
Let $\calA\to\calB$ be a local homomorphism of seminormed local rings and let $k_\calA$ and $k_\calB$ be the residue fields provided with the quotient seminorms. If $\lvert m_\calB\rvert _\calB=0$ then the natural map $\Omega_{\calB/\calA}\to\Omega_{k_\calB/k_\calA}$ is an isometry.
\end{cor}
\begin{proof}
By our assumption, the map $\Omega_{\calB/\calA}\to \Omega_{\calB/\calA}/m_\calB\Omega_{\calB/\calA}$ is an isometry. Hence the map $\Omega_{\calB/\calA}\to\Omega_{k_\calB/\calA}$ is an isometry by Lemma~\ref{fundseclem}(ii). Since the map $\calA\to k_\calB$ factors through $k_\calA$ there is also an isometry $\Omega_{k_\calB/\calA}\to\Omega_{k_\calB/k_\calA}$.
\end{proof}

\subsubsection{Base change}
Similarly to the classical case, K\"ahler differentials of seminormed modules are compatible with base changes.

\begin{lem}\label{basechangelem}
Assume that $\calA\to\calB$ and $\calA\to\calA'$ are non-expansive homomorphism of seminormed rings and set $\calB'=\calB\otimes_\calA\calA'$. Then $\phi\:\Omega_{\calB/\calA}\otimes_{\calA'}\calB'\to\Omega_{\calB'/\calA'}$ is an isomorphism of seminormed modules.
\end{lem}
\begin{proof}
Since $\phi$ is an isomorphism of modules we should prove that it is an isometry. Let $\lVert \ \rVert_s$ and $\lVert \ \rVert_t$ denote the seminorms on the source and on the target, respectively. Recall that $\lVert\ \rVert_s$ is the maximal seminorm making the bilinear map $\Omega_{\calB/\calA}\times\calB'\to\Omega_{\calB/\calA}\otimes_{\calA'}\calB'$ non-expansive (see \S\ref{tensorsec}). This fact and Lemma~\ref{universemilem}(i) imply that $\lVert \ \rVert_s$ is the maximal seminorm for which $\lVert ad(b)\otimes c\rVert_s\le \lvert a\rvert_\calB\lvert b\rvert_\calB\lvert c\rvert_{\calB'}$ for any $a,b\in\calB$ and $c\in\calB'$. In addition, Lemma~\ref{universemilem}(i) implies that $\lVert \ \rVert_t$ is the maximal seminorm such that $\lVert xd(y)\rVert_t\le \lvert x\rvert_{\calB'}\lvert y\rvert_{\calB'}$ for any choice of $x,y\in\calB'$. Since $ad(b)\otimes c$ goes to $acd(b)$, it follows that any inequality defining $\lVert \ \rVert_s$ holds also for $\lVert \ \rVert_t$, and so $\phi$ is non-expansive.

It remains to check that $\phi^{-1}$ is non-expansive and for this we will check that any inequality defining $\lVert \ \rVert_t$ holds for $\lVert \ \rVert_s$ too. As we noted above, $\lVert \ \rVert_t$ is defined by the inequalities $\|z\|_t\le\lvert x\rvert_{\calB'}\lvert y\rvert_{\calB'}$, where $x,y\in\calB'$ and $z=xdy$. Fix $r>\lvert y\rvert_{\calB'}$ and find a representation $y=\sum_{i=1}^n b_i\otimes a'_i$ with $b_i\in\calB$, $a'_i\in\calA'$, and $\lvert b_i\rvert_\calB\lvert a'_i\rvert_{\calA'}\le r$ for any $1\le i\le n$. Then $\phi^{-1}(z)=\sum_{i=1}^n d(b_i)\otimes a'_ix$ and hence $$\lVert \phi^{-1}(z)\rVert_s\le\max_i\lvert b_i\rvert_\calB\lvert a'_i\rvert_{\calA'}\lvert x\rvert_{\calB'}\le r\lvert x\rvert_{\calB'}.$$ Thus $\lVert \phi^{-1}(xdy)\rVert_s\le\lvert y\rvert_{\calB'}\lvert x\rvert_{\calB'}$ for any $x,y\in\calB'$, and we are done.
\end{proof}

\subsubsection{Density}

\begin{lem}\label{denselem}
If $\phi\:\calB_0\to\calB$ is a non-expansive homomorphism of seminormed rings with a dense image then the K\"ahler seminorm of $\Omega_{\calB/\calB_0}$ vanishes.
\end{lem}
\begin{proof}
For any $b\in\calB$ and $\veps>0$ there exists $b_0\in\calB_0$ such that $\lvert b-b_0\rvert <\veps$. Hence $\lVert db\rVert =\lVert d(b-b_0)\rVert \le\veps$.
\end{proof}

\begin{cor}\label{densecor}
If $\calA\to\calB_0\to\calB$ are homomorphisms of seminormed rings and the image of $\calB_0\to\calB$ is dense then the homomorphism $\phi\:\Omega_{\calB_0/\calA}\otimes_{\calB_0}\calB\to\Omega_{\calB/\calA}$ has a dense image.
\end{cor}
\begin{proof}
By Lemma~\ref{fundseclem}(ii), $K=\Im(\phi)$ is the kernel of the admissible surjection $\Omega_{\calB/\calA}\to\Omega_{\calB/\calB_0}$. Since the seminorm on $\Omega_{\calB/\calB_0}$ is trivial by Lemma~\ref{denselem}, it follows that any element of $\Omega_{\calB/\calA}$ can be approximated by an element of $K$ with any precision, i.e. $K$ is dense.
\end{proof}

\subsubsection{Filtered colimits}
We will also need that K\"ahler seminorms are compatible with filtered colimits.

\begin{lem}\label{normcolimlem}
Assume that $\calA$ is a seminormed ring and $\{\calB_\lam\}_{\lam\in\Lam}$ is a filtered family of seminormed $\calA$-algebras with non-expansive transition homomorphisms. Let $\calB$ be the colimit seminormed algebra, see \S\ref{normcolimsec}. Then $\Omega_{\calB/\calA}=\colim_\lam\Omega_{\calB_\lam/\calA}$ as seminormed $\calA$-modules.
\end{lem}
\begin{proof}
It is a classical result that the modules are isomorphic, so we should compare the seminorms. Let $\lVert \ \rVert$ be the colimit seminorm on $\Omega_{\calB/\calA}$. By Lemma \ref{universemilem}(i) and \S\ref{normcolimsec}, $\lVert \ \rVert$ is the maximal seminorm such that the composed $\calA$-homomorphisms $\calB_\lam\stackrel{d_\lam}{\to}\Omega_{\calB_\lam/\calA}\to\Omega_{\calB/\calA}$ are non-expansive. The latter decompose as $\calB_\lam\stackrel{\phi_\lam}{\to}\calB\stackrel{d}{\to}\Omega_{\calB/\calA}$ and it follows from the definition of $\lvert \ \rvert_\calB$ that $d\:\calB\to\Omega_{\calB/\calA}$ is non-expansive and $\lVert \ \rVert$ is the maximal seminorm for which this happens. Thus, $\lVert \ \rVert$ coincides with the K\"ahler seminorm.
\end{proof}

\subsection{Completed differentials}

\subsubsection{The module $\whOmega_{\calB/\calA}$}\label{modomegasec}
By $\whOmega_{\calB/\calA}$ we denote the completion of $\Omega_{\calB/\calA}$. It is a Banach $\calB$-module and we call its norm the {\em K\"ahler norm}. The corresponding $\calA$-derivation will be denoted $\whd_{\calB/\calA}\:\calB\to\whOmega_{\calB/\calA}$.

\subsubsection{The universal property}
Lemma~\ref{universemilem} and the universal property of the completions imply the following result.

\begin{lem}\label{banachuniversallem}
If $\phi\:\calA\to\calB$ is a non-expansive homomorphism of seminormed rings then $\whd_{\calB/\calA}$ is the universal non-expansive $\calA$-derivation of $\calB$ with values in Banach $\calB$-modules. Namely,
$$\Hom_{\calB,\nonexp}(\whOmega_{\calB/\calA},M)\toisom \Der_{\calA,\nonexp}(\calB,M)$$
for any Banach $\calB$-module $M$.
\end{lem}

\begin{rem}\label{Omegarem}
(i) In the case of $k$-affinoid algebras, Berkovich defines $\whOmega_{\calB/\calA}$ in \cite[\S3.3]{berihes} as $J/J^2$, where $J=\Ker(\calB\to\calB\wtimes_\calA\calB)$. Note that the notation in \cite{berihes} does not use hat because uncompleted modules of K\"ahler differentials are never considered there, but we have to distinguish them in our paper. It follows from \cite[Prop. 3.3.1(ii)]{berihes} and Lemma~\ref{banachuniversallem} that Berkovich's definition is equivalent to ours. In particular, if $X=\calM(\calB)$ and $S=\calM(\calA)$ then $\Gamma(\Omega_{X/S})=\whOmega_{\calB/\calA}$.

(ii) Alternatively, one could deduce the equivalence of the two definitions from Lemma~\ref{omegaalter}. Even more generally, Lemma~\ref{omegaalter} implies that if $\calA\to\calB$ is a homomorphism of Banach rings, $J=\Ker(\calB\to\calB\wtimes_\calA\calB)$ and $\ol{J^2}$ is the closure of $J^2$ in $J$, then $\whOmega_{\calB/\calA}=\wh{J/J^2}=J/\ol{J^2}$. In the affinoid case, all ideals are automatically closed, so $\ol{J^2}=J^2$.
\end{rem}

\section{K\"ahler seminorms of real-valued fields}\label{Kahlersec}
Our next aim is to study K\"ahler seminorms on the vector spaces $\Omega_{K/A}$, where $K$ is a real-valued field and $\phi\:A\to K$ is a homomorphism of rings. In this case we will use the notation $\Acirc=\phi^{-1}(\Kcirc)$. In fact, we are mainly interested in the cases when $A=\bfZ$ or $A$ is a field, but we will consider an arbitrary $A$ when possible.

\subsection{Log differentials}
The aim of this section is to express the K\"ahler seminorm on $\Omega_{K/A}$ in terms of modules of log differentials, so we start with recalling some basic facts about the latter.

\subsubsection{Log rings}
A {\em log structure} on a ring $A$ is a homomorphism of monoids $\alp_A\:M_A\to(A,\cdot)$ inducing an isomorphism $M_A^\times\toisom A^\times$. The triple $(A,M_A,\alp_A)$ is called a {\em log ring}. Usually we will denote it as $(A,M_A)$ and, when this cannot cause to a confusion, denote elements $\alp_A(m)$ simply by $m$. Homomorphisms of log rings are defined in the natural way.

\subsubsection{Log differentials}
If $(A,M_A)\to(B,M_B)$ is a homomorphism of log rings then we denote by $\Omega_{(B,M_B)/(A,M_A)}$ the module of log differentials. Recall that the latter is defined in \cite[\S1.7]{Kato-log} as the quotient of $\Omega_{B/A}\oplus (B\otimes M_B^\gp)$ by the relations of the form $(0,1\otimes a)$, where $a\in M_A$, and $(d_{B/A}(b),-b\otimes b)$, where $b\in M_B$. The full form of the latter relation is $(d_{B/A}(\alp_B(b)),-\alp_B(b)\otimes b)$, and we used our convention to present it in a more compact form. For any $b\in M_B$ we denote by $\delta_{B/A}(b)$ the image of $1\otimes b\in B\otimes M_B^\gp$ in $\Omega_{(B,M_B)/(A,M_A)}$. It satisfies the equality $b\delta_{B/A}(b)=d_{B/A}(b)$. Intuitively, one may view $\delta b$ as the log differential $d\log b$ and $\Omega_{(B,M_B)/(A,M_A)}$ is the universal module obtained from $\Omega_{B/A}$ by adjoining the log differentials of the elements of $M_B$.

\subsubsection{The universal log derivative}
A {\em log $(A,M_A)$-derivation} of $(B,M_B)$ with values in a $B$-module $N$ consists of an $A$-derivation $d\:B\to N$ and a homomorphism $\delta\:M_B\to N$ such that $dm=m\delta m$ for any $m\in M_B$ and $\delta m=0$ for any $m$ coming from $M_A$. The $B$-module of all log derivations is denoted $\Der_{(A,M_A)}((B,M_B),N)$.

\begin{lem}\label{logderlem}
If $(A,M_A)\to(B,M_B)$ is a homomorphism of log rings then $$\left(d_{B/A}\:B\to\Omega_{(B,M_B)/(A,M_A)},\ \delta_{B/A}\:M_B\to\Omega_{(B,M_B)/(A,M_A)}\right)$$ is a universal log $(A,M_A)$-derivation of $(B,M_B)$. Namely, for any $B$-module $N$ the induced homomorphism
$$\Hom_B(\Omega_{(B,M_B)/(A,M_A)},N)\toisom\Der_{(A,M_A)}((B,M_B),N)$$
is an isomorphism.
\end{lem}

This fact is proved similarly to its classical analogue for K\"ahler differentials, so we skip the proof (see also \cite[Proposition~IV.1.1.6]{Ogus@book}). As an immediate corollary one obtains the first fundamental sequence (see also \cite[Proposition~IV.2.3.1]{Ogus@book}).

\begin{cor}\label{logfundseqlem}
Assume that $(A,M_A)\to(B,M_B)\to(C,M_C)$ are homomorphisms of log rings. Then the sequence
$$\Omega_{(B,M_B)/(A,M_A)}\otimes_BC\to\Omega_{(C,M_C)/(A,M_A)}\to\Omega_{(C,M_C)/(B,M_B)}\to 0$$
is exact.
\end{cor}

\subsubsection{The integral log structure}
Given a valued field $K$ and a homomorphism $\phi\:B\to\Kcirc$, we will use the notation $$\Omega^\rmlog_{\Kcirc/B}=\Omega_{(\Kcirc,\Kcirc\setminus\{0\})/(B,B\setminus\Ker(\phi))}$$ and $\Omega^\rmlog_{\Kcirc}=\Omega^\rmlog_{\Kcirc/\bfZ}$.

\subsubsection{The K\"ahler seminorm on $\Omega_{K/A}$}
The following result expresses the K\"ahler seminorm in terms of a module of log differentials.

\begin{theor}\label{unitballth}
Assume that $K$ is a real-valued field with a subring $A$ and let $\lVert \ \rVert_\Omega$ denote the K\"ahler seminorm of $\Omega_{K/A}$. Then,

(i) Localization induces an isomorphism $\Omega^\rmlog_{\Kcirc/\Acirc}\otimes_{\Kcirc}K=\Omega_{K/A}$. In particular, one can naturally view $(\Omega^\rmlog_{\Kcirc/\Acirc})_\tf$ as a semilattice in $\Omega_{K/A}$.

(ii) $\lVert \ \rVert_\Omega$ is the maximal $K$-seminorm whose unit ball contains $(\Omega^\rmlog_{\Kcirc/\Acirc})_\tf$. In particular, $(\Omega^\rmlog_{\Kcirc/\Acirc})_\tf$ is an almost unit ball of $\lVert \ \rVert_\Omega$.
\end{theor}
\begin{proof}
Part (i) follows from the fact that the modules of log differentials are compatible with localizations and the log structure at the generic point of $\Kcirc$ is trivial: $$\Omega^\rmlog_{\Kcirc/\Acirc}\otimes_{\Kcirc}K=\Omega_{(K,K^\times)/(\Acirc,\Acirc\setminus\{0\})}=\Omega_{K/\Acirc}=\Omega_{K/A}.$$ Note that $\lVert \ \rVert_\Omega$ is determined by the inequalities $\lVert dx\rVert_\Omega\le\lvert x\rvert$ with $x\in K$. For any seminorm, $\lVert dx\rVert \le \lvert x\rvert$ if and only $\lVert d(x^{-1})\rVert =\lVert -x^{-2}dx\rVert \le\lvert x^{-1}\rvert$, hence already the inequalities $\lVert \delta x\rVert_\Omega\le 1$ with $0\neq x\in\Kcirc$ determine $\lVert \ \rVert_\Omega$. This proves (ii).
\end{proof}

In view of \S\ref{semilatsec} we obtain the following formula for $\lVert \ \rVert_\Omega$.

\begin{cor}\label{cotorcor}
Keep the assumptions of Theorem~\ref{unitballth}. Then the restriction of the K\"ahler seminorm of $\Omega_{K/A}$ to $(\Omega^\rmlog_{\Kcirc/\Acirc})_\tf$ coincides with the adic seminorm of the latter: $\lVert x\rVert_\Omega=\lVert x\rVert_\adic$ for any $x\in(\Omega^\rmlog_{\Kcirc/\Acirc})_\tf$.
\end{cor}

\begin{rem}
(i) We have just seen that the K\"ahler seminorm on $\Omega_{K/A}$ is tightly related to the torsion free quotient of $\Omega^\rmlog_{\Kcirc/\Acirc}$. Nevertheless, we will later see (e.g. in Theorem~\ref{kahlerextth}) that the torsion submodules of the modules $\Omega^\rmlog_{\Kcirc/\Acirc}$ do affect subtle issues related to K\"ahler seminorms.

(ii) It follows from Corollary~\ref{relcor} below that if $A$ itself is a field and $\lvert A^\times\rvert$ is dense then one can replace $\Omega^\rmlog_{\Kcirc/\Acirc}$ by $\Omega_{\Kcirc/\Acirc}$ in Theorem~\ref{unitballth} and its corollary. However, in the discretely valued case the discrepancy between the two modules is essential, and one has to stick to logarithmic differentials (e.g, see Section~\ref{comparsec}).

(iii) Although any serious investigation of ramification theory of valued fields should study the modules $\Omega_{\Lcirc/\Kcirc}$ and $\Omega^\rmlog_{\Lcirc/\Kcirc}$ and their torsion, it seems that \cite[Chapter~6]{Gabber-Ramero} is the only source where this topic was systematically explored.
\end{rem}

\subsection{Main results on $\Omega_{\Lcirc/\Kcirc}$}\label{omegasec}
In Sections \ref{omegasec}--\ref{almtamesec} we will study (logarithmic) differentials of real-valued fields. In fact, all results hold for general valued fields whose height does not exceed one, but we prefer to fix the group of values to keep uniformness of the paper: we still work within the "seminormed framework".

We start with studying modules $\Omega_{\Lcirc/\Kcirc}$, and we are especially interested in a control on their torsion submodules. To large extent this is based on results of Gabber and Ramero from \cite[Chapter 6]{Gabber-Ramero}, where arbitrary valued fields are studied.

\subsubsection{Basic ramification theory notation}
We say that a finite extension $L/K$ is {\em unramified} if $\Lcirc/\Kcirc$ is \'etale. An algebraic extension $L/K$ is called {\em unramified} if its finite subextensions are so. As in \cite[Sections~6.2]{Gabber-Ramero}, we denote by $K^\sh$ the strict henselization of $K$ (in the sense that $(K^\sh)^\circ$ is the strict henselization of $\Kcirc$); it is the maximal unramified extension of $K$. By $K^t$ we denote the maximal tame extension of $K$; it is the union of all extensions $L/K^\sh$ whose degree is invertible in $\tilK$. For example, if $K$ is trivially valued then $K^t=K^\sh$ is the separable closure of $K$.

\subsubsection{The cotangent complex $\bfL_{\Lcirc/\Kcirc}$}
The module $\Omega_{\Lcirc/\Kcirc}$ is the zeroth homology of the cotangent complex $\bfL_{\Lcirc/\Kcirc}$. Studying the latter is the central topic of \cite[Sections~6.3,6.5]{Gabber-Ramero} and we formulate the main result below. We say that a field extension $L/K$ is {\em separable} if $L$ is geometrically reduced over $K$.

\begin{theor}\label{cotanth}
Let $L/K$ be an extension of real-valued fields. Then,

(i) $H_i(\bfL_{\Lcirc/\Kcirc})=0$ for $i>1$.

(ii) The module $H_1(\bfL_{\Lcirc/\Kcirc})$ is torsion free and it vanishes if and only if $L/K$ is separable.

(iii) The module $\Omega_{\Lcirc/\Kcirc}$ is torsion whenever $L/K$ is algebraic and separable.
\end{theor}
\begin{proof}
(i) and the first assertion of (ii) are proved in \cite[Theorem~6.5.12(i)]{Gabber-Ramero}. Since $H_1(\bfL_{\Lcirc/\Kcirc})\otimes_{\Lcirc}L=H_1(\bfL_{L/K})$ and the latter module vanishes if and only if $L/K$ is separable, we obtain the second assertion of (ii). Finally, $\Omega_{\Lcirc/\Kcirc}\otimes_{\Lcirc}L=\Omega_{L/K}$ and the latter module vanishes whenever $L/K$ is algebraic and separable.
\end{proof}

\begin{rem}\label{cotanrem}
The proving method of \cite[Theorem~6.5.12(i)]{Gabber-Ramero} is as follows: first one explicitly computes $\bfL_{\Lcirc/\Kcirc}$ for certain elementary extensions $L/K$ using that in this case $\Lcirc$ is a filtered colimit of subrings $\Kcirc[x_i]$ (see \cite[Propositions~6.3.13 and 6.5.9]{Gabber-Ramero}), then the general case is deduced via transitivity triangles.
\end{rem}

\subsubsection{The six-term exact sequence}
In the sequel, we will use the notation $\Upsilon_{\Lcirc/\Kcirc}=H_1(\bfL_{\Lcirc/\Kcirc})$. By \cite[Proposition~II.2.1.2]{Illusie@cotangent} any tower of real-valued extensions $F/L/K$ gives rise to a distinguished transitivity triangle of cotangent complexes
$$\bfL_{\Lcirc/\Kcirc}\otimes_{\Lcirc}\Fcirc\to\bfL_{\Fcirc/\Kcirc}\to\bfL_{\Fcirc/\Lcirc}\stackrel{+1}\to$$
and by Theorem~\ref{cotanth} we obtain a six-term exact sequence of homologies $$0\to\Upsilon_{\Lcirc/\Kcirc}\otimes_{\Lcirc}\Fcirc\to\Upsilon_{\Fcirc/\Kcirc}\to\Upsilon_{\Fcirc/\Lcirc}\to\Omega_{\Lcirc/\Kcirc}
\otimes_{\Lcirc}\Fcirc\to\Omega_{\Fcirc/\Kcirc}\to\Omega_{\Fcirc/\Lcirc}\to 0.$$

\subsubsection{Tame extensions}
Cotangent complexes of tame extensions are especially simple.

\begin{lem}\label{tamecomplex}
If $L/K$ is a tame algebraic extension then the $\Lcirc$-module $\Omega_{\Lcirc/\Kcirc}$ is isomorphic to $\Lcirccirc/\Kcirccirc\Lcirc$ and hence $\bfL_{\Lcirc/\Kcirc}$ is quasi-isomorphic to the module $\Lcirccirc/\Kcirccirc\Lcirc$ placed in the degree 0.
\end{lem}
\begin{proof}
By Theorem~\ref{cotanth} it suffices to prove the first isomorphism. The claim is obvious for unramified extensions since both sides vanish. Using the transitivity triangles we can now replace $L$ and $K$ with $L^\sh$ and $K^\sh$ and assume in the sequel that $K=K^\sh$. Both sides of the asserted equality commute with filtered unions of extensions hence we can assume that $L/K$ is finite. Then the lemma reduces to two cases: (a) if the valuation of $K$ is discrete $K$ then $\Omega_{\Lcirc/\Kcirc}$ is cyclic of length $e_{L/K}-1$, (b) if $\lvert K^\times\rvert $ is dense then $\Omega_{\Lcirc/\Kcirc}=0$. The first claim is classical, so we only check the second one.

Since $K=K^\sh$ it follows that $L/K$ breaks into a tower of elementary extensions of the form $F(\pi^{1/l})/F$ with $l$ invertible in $\tilK$ and $|\pi|\notin|K^\times|^l$. Using the transitivity triangles it suffices to consider the case when $L=K(\pi^{1/l})$. Note that $\Lcirc$ is the filtered union of the subalgebras $A_i=\Kcirc[x_i^{1/l}]$ where $x_i\in\pi K^l$ satisfy $\lvert x_i\rvert <1$. Since $l\in(\Lcirc)^\times$ one easily sees that $\Omega_{A_i/\Kcirc}=A_idx_i/x_i^{l-1}A_idx_i$, and hence $\Omega_{\Lcirc/\Kcirc}$ is the filtered colimit of the modules $M_j=\Lcirc dx_i/x_i^{l-1}\Lcirc dx_i$ with the following transition maps: if $\lvert x_i\rvert \le \lvert x_j\rvert $ then $A_i\subseteq A_j$ and the map $M_i\to M_j$ is given by $dx_i\mapsto\frac{x_i}{x_j}dx_j$. The image of $dx_i$ in $M_j$ vanishes whenever $\frac{\lvert x_i\rvert }{\lvert x_j\rvert }\le \lvert x_j\rvert^{l-1}$, that is $\lvert x_i\rvert \le\lvert x_j\rvert^l$. Since $\lvert K^\times\rvert $ is dense, $\lvert x_j\rvert $ can be arbitrarily close to 1 and taking $\lvert x_j\rvert >\lvert x_i\rvert^{1/l}$ we kill the image of $dx_i$. Thus $\colim_i M_i=0$ and we are done.
\end{proof}

\subsubsection{Dense extensions}
Using the method outlined in Remark~\ref{cotanrem} we will also study $\bfL_{\Lcirc/\Kcirc}$ when $K$ is dense in $L$. It is easy to see that $\hatOmega_{\Lcirc/\Kcirc}=0$ and hence $\Omega_{\Lcirc/\Kcirc}=0$ is divisible, but we will need a more precise statement. We say that an $\Lcirc$-module $M$ is a {\em vector space} if it is an $L$-module. Equivalently, $M$ is divisible and torsion free.

\begin{lem}\label{vectorlem}
Assume that $L/K$ is an extension of real-valued fields and $K$ is dense in $L$. Then both $\Omega_{\Lcirc/\Kcirc}$ and $\Upsilon_{\Lcirc/\Kcirc}$ are $L$-vector spaces.
\end{lem}
\begin{proof}
We start with three classes of elementary extensions, and then the general case will be deduced in three more steps.

Case 1. {\em Assume that $L/K$ is separable algebraic.} Since $K$ is dense in $L$, it follows that $L$ lies in the henselization of $K$, i.e. $\Lcirc/\Kcirc$ is \'etale. In this case, the cotangent complex vanishes.

Case 2. {\em Assume that $L/K$ is purely inseparable of degree $p$.} In this case, $L=K(x)$ with $a=x^p\in K\setminus K^p$ and there exist $c_i\in K$ such that $\lim_i\lvert x-c_i\rvert =0$. Clearly, $f(t)=t^p-a$ is the minimal polynomial of $x$. It is easy to see that $\Lcirc$ is the filtered colimit of its subrings $\Kcirc[x_i]$, with $x_i=\frac{x-c_i}{\pi_i}$ where $\pi_i\in K$ are such that $\lim_i\lvert \pi_i\rvert =0$ (same argument as in the proof of \cite[6.3.13(i)]{Gabber-Ramero}). It follows that $\bfL_{\Lcirc/\Kcirc}=\colim_i\bfL_{\Kcirc[x_i]/\Kcirc}$. Since $\Kcirc[x_i]=\Kcirc[t]/I_i$ where $I_i=(f_i)$ and $f_i(t)=t^p-\frac{a-c_i^p}{\pi_i^p}$ is the minimal polynomial of $x_i$, the homologies of $\bfL_{\Kcirc[x_i]/\Kcirc}$ are easily computable (cf. the proof of \cite[6.3.13(iv)]{Gabber-Ramero}): $\Omega_{\Kcirc[x_i]/\Kcirc}$ is the invertible module with basis $dx_i$ and $\Upsilon_{\Kcirc[x_i]/\Kcirc}=I_i/I_i^2$ is the invertible module with basis $f_i$. Since $dx_i=\frac{dx}{\pi_i}$, we obtain that $\Omega_{\Lcirc/\Kcirc}=Ldx$.

To describe the map $\Upsilon_{\Kcirc[x]/\Kcirc}\to\Upsilon_{\Kcirc[x_i]/\Kcirc}$ we consider compatible presentations $\Kcirc[t]\onto\Kcirc[x]$ and $\Kcirc[t_i]\onto\Kcirc[x_i]$, where the connecting maps take $t$ and $x$ to $\pi_it_i+c_i$ and $\pi_ix_i+c_i$, respectively. The connecting maps induce the map $I/I^2\to I_i/I_i^2$ that sends $f=t^p-a$ to $\pi_i^pt_i^p+c_i^p-a=\pi_i^pf_i(t_i)$. This completely determines the filtered family $\Upsilon_{\Kcirc[x_i]/\Kcirc}$, and since $\lim_i\lvert \pi^p_i\rvert =0$, its colimit is the one-dimensional $L$-vector space with basis $f$.

Case 3. {\em Assume that $L=K(x)$ is purely transcendental.} In this case we have that $\Upsilon_{\Lcirc/\Kcirc}\otimes_{\Lcirc}L=\Upsilon_{L/K}=0$, and since $\Upsilon_{\Lcirc/\Kcirc}$ is torsion free, it actually vanishes. The module $\Omega_{\Lcirc/\Kcirc}$ is computed as in Case 2. First, one shows that $\Lcirc$ is the filtered colimit of localizations of the $\Kcirc$-algebras $\Kcirc[x_i]$ where $x_i=\frac{x-c_i}{\pi_i}$ and $\lim_i\lvert \pi_i\rvert =0$ (same argument as in \cite[6.3.13(i)]{Gabber-Ramero} or \cite[6.5.9]{Gabber-Ramero}). It then follows that $\Omega_{\Lcirc/\Kcirc}$ is the colimit of invertible modules $\Lcirc dx_i$ and $dx_i=\frac{dx}{\pi_i}$. Thus, $\Omega_{\Lcirc/\Kcirc}=Ldx$.

Case 4. {\em Assume that $L/K$ is finite.} We induct on $[L:K]$. If $L/K$ is as in Cases 1 or 2 then we are done. Otherwise, it can be split into a tower $K\subsetneq F\subsetneq L$, and the claim holds true for $F/K$ and $L/F$ by the induction. Therefore, in the six-term exact sequence the terms corresponding to $F/K$ and $L/F$ are vector spaces. It follows easily that the remaining terms are vector spaces too.

Case 5. {\em Assume that $L/K$ is finitely generated.} This time we induct on $d=\trdeg(L/K)$. If $d=0$ then we are in Case 4. Otherwise choose a purely transcendental subextension $F=K(x)\subseteq L$ and note that the assertion holds for $F/K$ by Case 3 and for $L/F$ by the induction. The same argument with the six-term sequence completes Case 5.

Case 6. {\em The general case.} Obviously, $L$ is the filtered colimit of its finitely generated $K$-subfields $L_i$ and $\Lcirc=\colim_i L_i^\circ$. It remains to use that the cotangent complex and the homology are compatible with filtered colimits, and filtered colimits of vector spaces are vector spaces.
\end{proof}

\subsubsection{The different}
We have defined the content of $\Kcirc$-modules in \S\ref{contsec}. For any separable extension of real-valued fields $L/K$ we define the {\em different} to be $\delta_{L/K}=\cont((\Omega_{\Lcirc/\Kcirc})_\tor)$. In particular, $\delta_{L/K}=1$ if $L$ is trivially valued. For inseparable extensions we set $\delta_{L/K}=0$.

\begin{theor}\label{difth}
Let $F/L/K$ be a tower of algebraic extensions of real-valued fields, then

(i) $\delta_{F/K}=\delta_{F/L}\delta_{L/K}$.

(ii) If $L$ is not trivially valued and $L/K$ is tame then $\delta_{L/K}=\lvert \Kcirccirc\rvert /\lvert \Lcirccirc\rvert $. In particular, if $K$ is not trivially valued then $\delta_{K^t/K}=\lvert \Kcirccirc\rvert $.

(iii) If $L/K$ is finite and separable then $\delta_{L/K}>0$.
\end{theor}
\begin{proof}
If $F/K$ is inseparable then either $F/L$ or $L/K$ is inseparable and both sides of the equality in (i) vanish. So, assume that $F/K$ is separable. By Theorem~\ref{cotanth} we have a short exact sequence of torsion modules
$$0\to\Omega_{\Lcirc/\Kcirc}\otimes_{\Lcirc}\Fcirc\to\Omega_{\Fcirc/\Kcirc}\to\Omega_{\Fcirc/\Lcirc}\to 0.$$
Hence (i) follows from Theorem~\ref{contlem} and the fact that for any $\Lcirc$-module $M$ and the $\Fcirc$-module $M'=M\otimes_{\Lcirc}\Fcirc$ the equality $\cont(M)=\cont(M')$ holds.

If $L$ is not trivially valued then $\cont(\Lcirccirc/\Kcirccirc\Lcirc)=\lvert \Kcirccirc\rvert /\lvert \Lcirccirc\rvert $ and hence (ii) follows from Lemma~\ref{tamecomplex}.

Let us prove (iii). Using (i) it suffices to prove that the different of the larger extension $K^tL/L$ does not vanish. Furthermore, $\delta_{K^t/K}>0$ by (ii), hence it suffices to consider the finite extension $K^tL/K^t$. Thus, we can assume that $K=K^t$ and then $L/K$ splits to a tower of extensions of degree $p=\cha(\tilK)$. Using (i) again it suffices to consider the case when $K=K^t$ and $[L:K]=p$. By \cite[Proposition~6.3.13]{Gabber-Ramero} we have that $\Lcirc$ is a filtered colimit of subalgebras $\Kcirc[a_ix+b_i]$, where $x\in L$ and $a_i,b_i\in K$. We can also assume that $x\in\Lcirc$. Note that the numbers $\lvert a_i\rvert $ are bounded by a finite number $C$ because otherwise $x\in\hatK$ and hence $\hatK$ contains the separable extension $L/K$, which is impossible since $K$ is henselian.

Set $x_i=a_ix+b_i$, then $\Omega_{\Lcirc/\Kcirc}$ is generated by the elements $dx_i$. Let $f(x)$ be the minimal polynomial of $x$ over $K$. Then $dx$ is annihilated by $f'(x)$ and $f'(x)\neq 0$ since $L/K$ is separable. It follows that $dx_i$ is annihilated by any $\pi$ with $\lvert \pi\rvert <\lvert f'(x)\rvert C^{-1}$ and since $\Omega_{\Lcirc/\Kcirc}$ is the filtered union of the submodules generated by $dx_i$, we obtain by Lemma~\ref{contentlem} that $\cont(\Omega_{\Lcirc/\Kcirc})\ge\lvert f'(x)\rvert C^{-1}>0$.
\end{proof}

\begin{rem}
(i) We will not need this, but our definition of the different is compatible with the definition of Gabber and Ramero in the following sense. They introduced in \cite[Section~2.3]{Gabber-Ramero} a class $\calC$ of uniformly almost finitely generated $\Kcirc$-modules $M$ and defined for them the formation of Fitting's ideals $F_i(M)$. It is not difficult to show that for any uniformly almost finitely generated torsion module $M$ the equality $\lvert F_0(M)\rvert =\cont(M)$ holds. In addition, it is proved in \cite[Proposition~6.3.8]{Gabber-Ramero} that if $L/K$ is a finite separable extension of real-valued fields then $\Omega_{\Lcirc/\Kcirc}$ is uniformly finitely generated, and then the different of $L/K$ is defined to be $F_0(\Omega_{\Lcirc/\Kcirc})$. Thus, in this case $\delta_{L/K}$ equals to the absolute value of the different from \cite{Gabber-Ramero}.

(ii) Classically, one defines the different only for algebraic extensions, but it is a reasonable invariant in the transcendental case too. For example, one can show that if $k$ is not trivially valued and $x\in\bfA^1_k$ is the maximal point of a disc $E$ then (the non-classical) $\delta_{\calH(x)/k}$ is the maximum of $\delta_{\calH(z)/k}$ where $z$ runs over rigid points in $E$. In the algebraic case, the different measures how wild the extension $K/k$ is. This also provides a good intuition for the meaning of the different when $K/k$ is not algebraic.
\end{rem}

\subsection{Relations between $\Omega^\rmlog_{\Lcirc/\Kcirc}$ and $\Omega_{\Lcirc/\Kcirc}$}
Our next aim is to compare the modules $\Omega^\rmlog_{\Lcirc/\Kcirc}$ and $\Omega_{\Lcirc/\Kcirc}$.

\subsubsection{Comparison homomorphism}
Let $\lam_{K/A}$ denote the natural map $\Omega_{\Kcirc/\Acirc}\to\Omega^\rmlog_{\Kcirc/\Acirc}$, and let $\lam_K=\lam_{K/\bfZ}$.

\begin{lem}\label{rellem}
For any real-valued field $K$ there is an exact sequence of $\Kcirc$-modules
$$0\to\Omega_{\Kcirc}\stackrel{\lam_K}\longrightarrow\Omega_{\Kcirc}^\rmlog\stackrel{\rho_K}\longrightarrow \lvert K^\times\rvert \otimes\tilK\to 0,$$
where $\rho_K(x\delta_K y)=|y|\otimes\tilx$ for $x,y\in\Kcirc$ and $\Kcirc$ acts on $\lvert K^\times\rvert \otimes\tilK$ through $\tilK$. In particular, if $K$ is not discretely valued (but $K$ may be trivially valued) then $\lam_K$ is an almost isomorphism.
\end{lem}
\begin{proof}
The map $\lam_L$ is injective for any valued field $L$ by \cite[Corollary~6.4.18(i)]{Gabber-Ramero}. Furthermore, if $L$ is of a finite height $h$ then \cite[Corollary~6.4.18(ii)]{Gabber-Ramero} constructs a filtration of $\Coker(\lam_L)$ of length $h$ with explicitly described quotients. Since the height of $K$ is one, this filtration degenerates and the cited result yields an isomorphism $\Coker(\lam_K)\toisom\lvert K^\times\rvert \otimes(\Kcirc/\Kcirccirc)$. The explicit description of $\rho_K$ is due to the description of a map $\rho$ in the proof of \cite[Proposition~6.4.15]{Gabber-Ramero} (see the formula for $\tilrho$ above \cite[Claim~6.4.16]{Gabber-Ramero}).
\end{proof}

\begin{cor}\label{relcor}
If $L/K$ is an extension of real-valued fields then $\Ker(\lam_{L/K})$ is annihilated by any element of $\Kcirccirc$ and $\Coker(\lam_{L/K})$ is annihilated by any element of $\Lcirccirc$. In particular, if $\lvert K^\times\rvert$ is dense then $\lam_{L/K}$ is an almost isomorphism.
\end{cor}
\begin{proof}
Let $N$ be the image of the map $\Omega^\rmlog_{\Kcirc}\otimes_{\Kcirc}\Lcirc\to\Omega^\rmlog_{\Lcirc}$ and let $\mu$ be the composition of $\lam_K\otimes_{\Kcirc}\Lcirc$ and the map $\Omega^\rmlog_{\Kcirc}\otimes_{\Kcirc}\Lcirc\onto N$. Applying the snake lemma to the commutative diagram
$$
\xymatrix{
& \Omega_{\Kcirc}\otimes_{\Kcirc}\Lcirc\ar[d]^\mu\ar[r] & \Omega_{\Lcirc}\ar@{^(_->}[d]^{\lam_L}\ar[r] & \Omega_{\Lcirc/\Kcirc} \ar[r]\ar[d]^{\lam_{L/K}}& 0\\
0\ar[r] & N \ar[r]& \Omega^\rmlog_{\Lcirc}\ar[r] & \Omega^\rmlog_{\Lcirc/\Kcirc} \ar[r]& 0\\
}
$$
we obtain an exact sequence
\begin{equation}\label{snakeeq}
0\to\Ker(\lam_{L/K})\to\Coker(\mu)\to\Coker(\lam_L)\to\Coker(\lam_{L/K})\to0.
\end{equation}
By Lemma~\ref{rellem}, $\Coker(\lam_L)$ is annihilated by $\Lcirccirc$, hence the same is true for $\Coker(\lam_{L/K})$. Similarly, $\Coker(\lam_K)$ is annihilated by $\Kcirccirc$, hence the same is true for $\Coker(\lam_K)\otimes_{\Kcirc}\Lcirc=\Coker(\lam_K\otimes_{\Kcirc}\Lcirc)$, for its quotient $\Coker(\mu)$, and for the submodule $\Ker(\lam_{L/K})$ of $\Coker(\mu)$.
\end{proof}

Using Lemma \ref{rellem} and its corollary one can easily verify the following examples.

\begin{exam}
(i) If $K$ is discretely valued with uniformizer $\pi$ then $\Coker(\lam_K)=\tilK$ with generator $\delta_K(\pi)$. In particular, $\lam_K$ is not an almost isomorphism in this case.

(ii) If $\cha(\tilK)=p>0$ and $\lvert K^\times\rvert$ is $p$-divisible then $\lam_K$ is an isomorphism.

(iii) Assume that $L/K$ is a tamely ramified finite extension of discretely valued fields with uniformizers $\pi_K$ and $\pi_L$. Note that $\Omega^\rmlog_{\Lcirc/\Kcirc}=0$ (in fact, $\Lcirc/\Kcirc$ is log \'etale), but $\Omega_{\Lcirc/\Kcirc}\toisom \pi_K\Lcirc/\pi_L\Lcirc$ is a cyclic module of length $e_{L/K}-1$. In particular, $\lam_{L/K}$ is not injective if the tamely ramified extension $L/K$ is not unramified.
\end{exam}

\subsubsection{The discrete valuation case}
If the valuation of $K$ is discrete then the discrepancy between $\Omega^\rmlog_{\Lcirc/\Kcirc}$ and $\Omega_{\Lcirc/\Kcirc}$ can be sensitive. However, the following trick  reduces the study of modules $\Omega^\rmlog_{\Lcirc/\Kcirc}$ to the non-discrete case.

\begin{lem}\label{discrlem}
Assume that $K$ is a discrete-valued field, $A\to K$ a homomorphism, and $L/K$ a tamely ramified algebraic extension. Then $\Omega^\rmlog_{\Kcirc/\Acirc}\otimes_{\Kcirc}\Lcirc\toisom\Omega^\rmlog_{\Lcirc/\Acirc}$.
\end{lem}
\begin{proof}
Since the formation of modules of log differentials is compatible with filtered colimits it suffices to consider the case when $L/K$ is finite. Then the situation is well-known: the log structures are fine and the inclusion $\Kcirc\into\Lcirc$ is log \'etale. By \cite[1.1(iii)]{Olsson} Olsson's log cotangent complex $\bfL^\rmlog_{\Lcirc/\Kcirc}$ is quasi-isomorphic to zero, and the lemma follows by considering the transitivity triangle $$\bfL^\rmlog_{\Kcirc/\Acirc}\otimes_{\Kcirc}\Lcirc\to\bfL^\rmlog_{\Lcirc/\Acirc}\to\bfL^\rmlog_{\Lcirc/\Kcirc}\stackrel{+1}\to$$
of $\Lcirc/\Kcirc/\Acirc$ (see \cite[1.1(v)]{Olsson}) and the associated sequence of homologies.
\end{proof}

\subsection{Main results on $\Omega^\rmlog_{\Lcirc/\Kcirc}$}
One can define the whole log cotangent complex $\bfL^\rmlog_{\Lcirc/\Kcirc}$ using an approach of Gabber that was elaborated by Olsson in \cite[Section~8]{Olsson}. The advantage of Gabber's theory is that it deals with not necessarily fine log rings, as our case is. It seems very probable that analogues of all main results of Section~\ref{omegasec} hold also in the logarithmic setting. We do not explore this here and only study the maps $$\psi^\rmlog_{\Lcirc/\Kcirc/\Acirc}\:\Omega^\rmlog_{\Kcirc/\Acirc}\otimes_{\Kcirc}\Lcirc\to\Omega^\rmlog_{\Lcirc/\Acirc}$$
by comparing them with the non-logarithmic analogues
$$\psi_{\Lcirc/\Kcirc/\Acirc}\:\Omega_{\Kcirc/\Acirc}\otimes_{\Kcirc}\Lcirc\to\Omega_{\Lcirc/\Acirc}.$$
Note that $\psi^\rmlog_{\Lcirc/\Kcirc/\Acirc}$ is obtained from $\psi^\rmlog_{\Lcirc/\Kcirc/\bfZ}$ by dividing both the source and the target by the image of $\Omega^\rmlog_{\Acirc}\otimes_{\Acirc}\Lcirc$, in particular, the following result holds.

\begin{lem}\label{kercokerlem}
Keep the above notation then $\Coker(\psi^\rmlog_{\Lcirc/\Kcirc/\bfZ})=\Coker(\psi^\rmlog_{\Lcirc/\Kcirc/\Acirc})$ and $\Ker(\psi^\rmlog_{\Lcirc/\Kcirc/\Acirc})$ is a quotient of $\Ker(\psi^\rmlog_{\Lcirc/\Kcirc/\bfZ})$.
\end{lem}

\subsubsection{Tame extensions}
The assertion of Lemma~\ref{discrlem} holds for non-discretely valued fields, but we will only need the following slightly weaker version.

\begin{lem}\label{tameextlem}
Assume that $L/K$ is a tamely ramified algebraic extension of real-valued fields and $A\to K$ is a homomorphism. Then the homomorphism $\psi^\rmlog_{\Lcirc/\Kcirc/\Acirc}$ is an almost isomorphism.
\end{lem}
\begin{proof}
Lemma \ref{discrlem} covers the case of a discrete-valued $K$. The trivially graded case is obvious since the logarithmic structures are trivial. It remains to consider the case when $\lvert K^\times\rvert $ is dense. By Lemma~\ref{rellem}, the maps $\lam_K$ and $\lam_L$ are almost isomorphisms. Since $\psi_{\Lcirc/\Kcirc/\bfZ}$ is an isomorphism by Lemma~\ref{tamecomplex}, we obtain that $\psi^\rmlog_{\Lcirc/\Kcirc/\bfZ}$ is an almost isomorphism. The assertion for an arbitrary $A$ follows by applying Lemma~\ref{kercokerlem}.
\end{proof}

\subsubsection{Separable extensions}
The logarithmic version of Theorem~\ref{cotanth} would imply that $\psi^\rmlog_{\Lcirc/\Kcirc/\Acirc}$ is injective whenever $L/K$ is separable. Comparison with the non-logarithmic case yields a slightly weaker result:

\begin{lem}\label{philem}
Assume that $L/K$ is a separable algebraic extension of real-valued fields and $A\to K$ is a homomorphism. Then the kernel of the map $\psi^\rmlog_{\Lcirc/\Kcirc/\Acirc}$ almost vanishes.
\end{lem}
\begin{proof}
Again, Lemma~\ref{kercokerlem} reduces the general claim to the case of the maps $\psi^\rmlog_{\Lcirc/\Kcirc/\bfZ}$. Furthermore, using Lemma~\ref{tameextlem}, it suffices to prove that the kernel of $\psi^\rmlog_{(L^t)^\circ/(K^t)^\circ/\bfZ}$ almost vanishes. Both $K^t$ and $L^t$ are not discrete-valued, hence $\psi^\rmlog_{(L^t)^\circ/(K^t)^\circ/\bfZ}$ is almost isomorphic to $\psi_{(L^t)^\circ/(K^t)^\circ/\bfZ}$ by Lemma~\ref{rellem}. The latter map has trivial kernel by Theorem~\ref{cotanth}(ii).
\end{proof}

\subsubsection{Dense extensions}
Next we consider the case when $K$ is dense in $L$.

\begin{lem}\label{denseloglem}
Assume that $L/K$ is an extension of real-valued fields such that $K$ is dense in $L$ and $A\to K$ is a homomorphism. Then the $\Lcirc$-modules $\Ker(\psi^\rmlog_{\Lcirc/\Kcirc/\bfZ})$ and $\Coker(\psi^\rmlog_{\Lcirc/\Kcirc/\Acirc})$ are vector spaces and the $\Lcirc$-module $\Ker(\psi^\rmlog_{\Lcirc/\Kcirc/\Acirc})$ is divisible.
\end{lem}
\begin{proof}
Consider the following commutative diagram where $\lam=\lam_K\otimes_{\Kcirc}\Lcirc$ and $\rho=\rho_K\otimes_{\Kcirc}\Lcirc$:
$$
\xymatrix{
0\ar[r] & \Omega_{\Kcirc}\otimes_{\Kcirc}\Lcirc\ar[d]^{\psi_{\Lcirc/\Kcirc/\bfZ}}\ar[r]^{\lam} & \Omega^\rmlog_{\Kcirc}\otimes_{\Kcirc}\Lcirc\ar[d]^{\psi^\rmlog_{\Lcirc/\Kcirc/\bfZ}}\ar[r]^(.4)\rho & (\lvert K^\times\rvert \otimes\tilK)\otimes_{\Kcirc}\Lcirc \ar[r]\ar[d]^\gamma& 0\\
0\ar[r] & \Omega_{\Lcirc} \ar[r]^{\lam_L}& \Omega^\rmlog_{\Lcirc}\ar[r]^{\rho_L} & \lvert L^\times\rvert \otimes\tilL \ar[r]& 0\\
}
$$
The rows are exact by Lemma~\ref{rellem}. Moreover, the formulas for $\rho_K$ and $\rho_L$ provided by Lemma~\ref{rellem} imply that $\gamma$ is the natural map. Since $\lvert K^\times\rvert =\lvert L^\times\rvert$ and $\tilK=\tilL$ we obtain that $\gamma$ is an isomorphism, and hence ${\psi_{\Lcirc/\Kcirc/\bfZ}}$ and ${\psi^\rmlog_{\Lcirc/\Kcirc/\bfZ}}$ have the same kernel and cokernel. By Lemma~\ref{vectorlem} these $\Lcirc$-modules are vector spaces. The general case follows by applying Lemma~\ref{kercokerlem}.
\end{proof}

\subsubsection{Log different}
Given a separable extension $L/K$ of real-valued fields we define the {\em log different} of $L/K$ to be $\delta_{L/K}^\rmlog=\cont((\Omega^\rmlog_{\Lcirc/\Kcirc})_\tor)$. For inseparable extensions we set $\delta_{L/K}^\rmlog=0$. The log different is related to the usual different in a very simple way and this allows to easily establish its basic properties:

\begin{theor}\label{logdifth}
Let $L/K$ be an algebraic extension of real-valued fields, then

(i) If $K$ is not trivially valued then $\delta^\rmlog_{L/K}=\delta_{L/K}\lvert \Lcirccirc\rvert/\lvert \Kcirccirc\rvert$.

(ii) If $F/L$ is another algebraic extension then $\delta^\rmlog_{F/K}=\delta^\rmlog_{F/L}\delta^\rmlog_{L/K}$.

(iii) If $L/K$ is tame then $\delta^\rmlog_{L/K}=1$.

(iv) If $L/K$ is finite and separable then $\delta^\rmlog_{L/K}>0$.
\end{theor}
\begin{proof}
The map $\Omega^\rmlog_{\Kcirc}\otimes_{\Kcirc}(K^t)^\circ\to\Omega^\rmlog_{(K^t)^\circ}$ is an almost isomorphism by Lemma~\ref{tameextlem}. Using this and the similar fact for $L$, one obtains an almost isomorphism $$\Omega^\rmlog_{\Lcirc/\Kcirc}\otimes_{\Lcirc}(L^t)^\circ\to\Omega^\rmlog_{(L^t)^\circ/(K^t)^\circ},$$ in particular, $\delta^\rmlog_{L/K}=\delta^\rmlog_{L^t/K^t}$. By Lemma~\ref{rellem}, $\Omega^\rmlog_{(L^t)^\circ/(K^t)^\circ}$ is almost isomorphic to $\Omega_{(L^t)^\circ/(K^t)^\circ}$, hence $\delta^\rmlog_{L^t/K^t}=\delta_{L^t/K^t}$. Finally, $$\delta_{L^t/K^t}=\delta_{L/K}\delta_{L^t/L}/\delta_{K^t/K}=\delta_{L/K}\lvert \Lcirccirc\rvert /\lvert \Kcirccirc\rvert$$ by Theorem~\ref{difth}, and we obtain (i). Using (i), one immediately deduces the other assertions from their non-logarithmic analogues proved in Theorem~\ref{difth}.
\end{proof}

\subsection{Almost tame extensions and fields}\label{almtamesec}

\subsubsection{Almost unramified extensions}
Let $L/K$ be a separable algebraic extension of real-valued fields. Following Faltings we say that $L/K$ is {\em almost unramified} if $\Omega_{\Lcirc/\Kcirc}$ almost vanishes.

\begin{lem}\label{almostunram}
Assume that $F/L/K$ is a tower of separable algebraic extensions of real-valued fields, then

(i) $L/K$ is almost unramified if and only if $\delta_{L/K}=1$.

(ii) $F/K$ is almost unramified if and only if both $F/L$ and $L/K$ are so.
\end{lem}
\begin{proof}
Recall that the module $\Omega_{\Lcirc/\Kcirc}$ is torsion by Theorem~\ref{cotanth}(iii). So, $\Omega_{\Lcirc/\Kcirc}$ almost vanishes if and only if its content equals to 1, and we obtain (i). The second claim follows from (i) because the different is multiplicative by Theorem~\ref{difth}(i).
\end{proof}

\subsubsection{Almost tame extensions}
In the discrete valuation case, almost unramified is the same as unramified, but in general there even are wildly ramified almost unramified extensions. Thus, being almost unramified is not a good measure of ``wildness" of extensions: there are tamely ramified extensions which are not almost unramified, and there are wildly ramified extensions which are almost unramified. Another weird property is that a tame extension $L/K$ is almost unramified if and only if $K$ is not discretely valued.

It is natural to seek for a natural enlargement of the class of almost unramified extensions that includes all tame ones, and this can be achieved very simply: one should pass to logarithmic differentials. So, we say that a separable algebraic extension $L/K$ is {\em almost tame} if $\Omega^\rmlog_{\Lcirc/\Kcirc}$ almost vanishes. Note that any tame extension is almost tame by Theorem~\ref{logdifth}(iii). As earlier, we immediately obtain the following result:

\begin{lem}\label{almosttamelem}
Assume that $F/L/K$ is a tower of separable algebraic extensions of real-valued fields, then

(i) $L/K$ is almost tame if and only if $\delta^\rmlog_{L/K}=1$.

(ii) $F/K$ is almost tame if and only if both $F/L$ and $L/K$ are so.
\end{lem}

\subsubsection{The non-discrete case}\label{nondiscrsec}
In the non-discrete case, there is no difference between almost tame and almost unramified because the kernel and cokernel of $\Omega_{\Lcirc/\Kcirc}\to\Omega^\rmlog_{\Lcirc/\Kcirc}$ almost vanish by Corollary~\ref{relcor}.

\subsubsection{The discrete case}
As one should expect, in the discrete-valued case the ``almost" version of tameness does not provide anything new.

\begin{lem}\label{discrcaselem}
Assume that $L/K$ is a separable algebraic extension of real-valued fields and the valuation on $K$ is discrete. Then $L/K$ is almost tame if and only if it is tame.
\end{lem}
\begin{proof}
Assume first that $L/K$ is finite. Then it is a classical result that $\delta_{L/K}\le\lvert \Kcirccirc\rvert /\lvert \Lcirccirc\rvert $ and the equality holds if and only if the extension is tame. Using Theorem~\ref{logdifth}(i) we obtain that $\delta_{L/K}^\rmlog=1$ if and only if $L/K$ is tame. The general case follows due to the following two observations: (1) $L/K$ is tame if and only if all its finite subextensions $L_i/K$ are tame, (2) $\delta_{L/K}^\rmlog$ is the limit of $\delta_{L_i/K}^\rmlog$ because $\Omega^\rmlog_{\Lcirc/\Kcirc}$ is the filtered colimit of $\Omega^\rmlog_{L_i^\circ/\Kcirc}$.
\end{proof}

\subsubsection{The defectless case}
In fact, one can describe a more general situation where almost tameness reduces to tameness. Recall that a finite extension of real-valued fields $L/K$ is called {\em defectless} if $e_{L/K}f_{L/K}=[L:K]$, and an algebraic extension is defectless if all its finite subextensions are so.

\begin{lem}\label{defectlesslem}
Assume that $L/K$ is a defectless separable extension of real-valued fields. Then $L/K$ is almost tame if and only if it is tame.
\end{lem}
\begin{proof}
The case of a trivially valued $K$ is obvious, and the case of a discretely valued $K$ was established in Lemma~\ref{discrcaselem}, so assume that $\lvert K^\times\rvert$ is dense. In this case, $L/K$ is almost tame if and only if it is almost unramified by \S\ref{nondiscrsec}, hence the lemma reduces to the following claim: {\em if $\lvert K^\times\rvert$ is dense and $L/K$ is defectless and not tame then $L/K$ is not almost unramified.} We start with establishing special cases of the claim. Lemma~\ref{almosttamelem}(ii) will be used all the time so we will not mention it.

Case 1. {\em $L/K$ is wildly ramified of degree $p$.} If $f_{L/K}=p$ choose $x\in\Lcirc$ such that $\tilx\notin\tilK$, and if $e_{L/K}=p$ choose $x\in L$ such that $\lvert x\rvert\notin\lvert K\rvert$. In either case, the elements $1,x\.. x^{p-1}$ form an orthogonal $K$-basis of $L$. We claim that $\delta_{L/K}=r^{1-p}\lvert h'(x)\rvert$, where $h(t)$ is the minimal polynomial of $x$ and $r=|x|$.

If $f_{L/K}=p$ then $\Lcirc=\Kcirc[x]$ and hence $\Omega_{\Lcirc/\Kcirc}=\Lcirc dx/\Lcirc h'(x)dx$. In particular, $\delta_{L/K}=\lvert h'(x)\rvert$, as claimed. If $e_{L/K}=p$ then $\Lcirc$ is the filtered union of subrings $A_j=\Kcirc[x_j]$ with $j\in\bfN$, where $x_j=\frac{x}{\pi_j}$ and $\pi_j\in K$ are such that $r_j=\lvert\pi_j\rvert$ decrease and tend to $r$ from above. It follows that $$\Omega_{\Lcirc/\Kcirc}=\colim_j\Omega_{A_j/\Kcirc}=\colim_j\Lcirc dx_j/\Lcirc h'_j(x_j)dx_j,$$ where $h_j(t)$ is the minimal polynomial of $x_j$ over $K$. Note that $h_j(t)=\pi_j^{-p}h(\pi_j t)$. Hence $h'_j(x_j)=\pi^{1-p}_jh'(\pi_jx_j)=\pi^{1-p}_jh'(x)$ and using that $dx_j=\pi_j^{-1}dx$ we obtain $$\Omega_{\Lcirc/\Kcirc}=\bigcup_j \pi_j^{-1}\Kcirc dx\slash\bigcup_j\pi_j^{-p}h'(x)\Kcirc=K^{\circcirc}_{r^{-1}} dx/K^{\circcirc}_{r^{-p}|h'(x)|},$$ where we set $K_s^{\circcirc}=\{a\in K\vert\ |a|>s\}$. Therefore, $\delta_{L/K}$ is as claimed.

It remains to estimate $\lvert h'(x)\rvert $, so let $h(t)=t^p+\sum_{i=0}^{p-1}a_it^i$. We claim that $\lvert a_i\rvert <r^{p-i}$ for $i>0$. First, $\lvert a_i\rvert\le r^{p-i}$ by \cite[Proposition~3.2.4/3]{bgr}. If $e_{L/K}=p$ then $r^{p-i}\notin|K^\times|$ for $0<i<p$, hence $a_i<r^{p-i}$. If $f_{L/K}=p$ then $\tilh(t)$ is the minimal polynomial of $\tilx$. Since $\tilL/\tilK$ is inseparable, $\tilh(t)$ is of the form $t^p-c$ and hence $|a_i|<1=r$ for $0<i<p$. Thus, $$\left\lvert h'(x)\right\rvert =\left\lvert px^{p-1}+\sum_{i=1}^{p-1}ia_ix^{i-1}\right\rvert \le\max\left(\lvert pr^{p-1}\rvert ,\max_{0<i<p}(\lvert a_ir^{i-1}\rvert )\right)<r^{p-1}$$ and we obtain that $\delta_{L/K}<1$, proving that $L/K$ is not almost tame.

Case 2. {\em $L/K$ is finite, Galois and totally wildly ramified.} In this case, $L/K$ splits into a tower of defectless wildly ramified extensions of degree $p$. Hence $L/K$ is not almost tame by Case 1.

Case 3. {\em $L/K$ is a composition of a tame extension $F/K$ and a Galois totally wildly ramified extension $L/E$ (i.e. $[L:E]=p^n$).} Since $L/K$ is not tame, we have that $n>0$ and hence $L/K$ is not almost tame by Case 2.

Case 4. {\em $L/K$ is finite.} By the standard ramification theory there exists a finite tame extension $F/L$ such that $F/K$ splits into a composition of a tame extension and a Galois totally wildly ramified extension. Since $F/L$ is tame it is defectless and hence $F/K$ is defectless. Thus, $F/K$ is not almost tame by Case 3, and since $F/L$ is almost tame we necessarily have that $L/K$ is not almost tame.

Case 5. {\em The general case.} By definition, $L/K$ contains a finite non-tame defectless subextension $F/K$, which is not almost tame by Case 4. Hence $L/K$ is not almost tame.
\end{proof}

\subsubsection{Almost tame extensions: the summary}
We can summarize a few above results as follows.

\begin{theor}\label{altameextth}
Assume that $L/K$ is a separable algebraic extension of real-valued fields, then

(i) $L/K$ is almost tame if and only if it is either tame or almost unramified.

(ii) Assume that $L/K$ is defectless; in particular, this is the case when $L/K$ is tame or $\lvert K^\times\rvert$ is discrete. Then $L/K$ is almost tame if and only if it is tame.

(iii) If $\lvert K^\times\rvert$ is dense then $L/K$ is almost tame if and only if $L/K$ is almost unramified.
\end{theor}

\subsubsection{Deeply ramified fields}
Assume that the valuation on $K$ is non-trivial. We refer the reader to \cite[Section~6.6]{Gabber-Ramero} for the definition and basic properties of deeply ramified fields. Recall that a real-valued field $K$ is called {\em deeply ramified} if $\Omega_{(K^s)^\circ/\Kcirc}=0$, and by \cite[Proposition~6.6.2]{Gabber-Ramero} this condition can be weakened by requiring that $K^s/K$ is almost unramified (i.e. $\Omega_{(K^s)^\circ/\Kcirc}$ almost vanishes). Furthermore, any deeply ramified field is perfect by \cite[Proposition~6.6.6(i)$\Longleftrightarrow$(ii)]{Gabber-Ramero}, hence $K$ is deeply ramified if and only if the algebraic closure $K^a$ is almost unramified over $K$.

\subsubsection{Almost tame fields}
We extend the notion of deeply ramified fields by replacing almost unramified extensions with almost tame ones. Assume that $K$ is a real-valued field, whose valuation can be trivial. We say that $K$ is {\em almost tame} if the extension $K^a/K$ is almost tame. For example, a trivially valued field is almost tame if and only if it is perfect.

\begin{rem}
(i) Giving a name to a special class of valued fields $K$ one often refers either to the property of $K$ over a ground valued field (e.g. over $\bfQ_p$) or to the properties all extensions of $K$ satisfy. It is slightly confusing but both approaches are used in the theory of valued fields. On the one hand, deeply ramified extensions are ``deeply ramified" over a discretely valued subfield. On the other hand, in the model theory of valued fields, $K$ is called {\em tame} if the extension $K^a/K$ is tame. Our notion of almost tame fields is an analogue (and generalization) of this classical notion.

(ii) Scholze defines in \cite[Section 3]{Scholze} perfectoid fields to be complete deeply ramified fields of height one and positive residue characteristic. In a sense, the condition of being perfectoid is a valuation-theoretic version of perfectness. So, it seems natural to extend the class of perfectoid fields by allowing non-complete fields and including perfect trivially valued fields and fields of residue characteristic zero. As we are going to prove, this larger class coincides with the class of almost tame fields. So, the notion ``perfectoid" is a reasonable alternative to ``almost tame".
\end{rem}

\begin{theor}\label{almosttameth}
A real-valued field $K$ is almost tame if and only if at least one of the following assertions is true: (i) $K$ is a perfect trivially valued field, (ii) $\cha(\tilK)=0$, (iii) $K$ is deeply ramified.
\end{theor}
\begin{proof}
We can assume that the valuation is non-trivial, as the other case is obvious. If $p=\cha(\tilK)=0$ then any algebraic extension of $K$ is tamely ramified, so $K$ is almost tame. In the sequel we assume that $p>0$. If $K$ is discretely valued then it possesses wildly ramified extensions, so $K$ is not almost tame by Lemma~\ref{discrcaselem}. If $\lvert K^\times\rvert$ is dense then there is no difference between almost tame and almost unramified extensions. In particular, in this case $K$ is almost tame if and only if it is deeply ramified.
\end{proof}

\begin{cor}\label{almosttamecor}
Assume that $K$ is a real-valued field.

(i) If $\cha(K)>0$ then $K$ is almost tame if and only if it is perfect.

(ii) Assume that the valuation is non-trivial. Then $K$ is almost tame if and only if $K^s/K$ is almost tame.
\end{cor}

\subsubsection{Separable extensions of almost tame fields}
One can also characterize almost tame fields in terms of separable extensions, at cost of considering transcendental ones.

\begin{theor}\label{torth}
Let $K$ be a real-valued field. Then $K$ is almost tame if and only if for any separable extension of real-valued fields $L/K$ the module $\Omega^\rmlog_{\Lcirc/\Kcirc}$ is almost torsion free.
\end{theor}
\begin{proof}
We start with the direct implication, so assume that $K$ is almost tame. If the valuation of $K$ is trivial then the module $\Omega^\rmlog_{\Lcirc/\Kcirc}=\Coker(\psi^\rmlog_{\Lcirc/\Kcirc/\bfZ})$ is torsion free by \cite[Corollary~6.5.21]{Gabber-Ramero}, so assume that the valuation is non-trivial. If $p=\cha(\tilK)=0$ then $\Omega^\rmlog_{\Lcirc/\Kcirc}$ is torsion free by \cite[Lemma~6.5.16]{Gabber-Ramero}, so assume that $p>0$. Then $\lvert K^\times\rvert$ is dense and hence $\lam_{L/K}$ is an almost isomorphism by Corollary~\ref{relcor}. So, it suffices to show that $\Omega_{\Lcirc/\Kcirc}$ is torsion free. Recall that $K$ is deeply ramified by Theorem~\ref{almosttameth} and hence $\Omega_{\Kcirc}$ is divisible.

Case 1: $\cha(K)=p$. Then $\Omega_{\Lcirc}=\Omega_{\Lcirc/\bfF_p}$ is torsion free by \cite[Claim~6.5.21]{Gabber-Ramero}. Thus $\Omega_{\Lcirc/\Kcirc}$ is the cokernel of the map $\Omega_{\Kcirc}\otimes_{\Kcirc}\Lcirc\to\Omega_{\Lcirc}$ whose source is divisible and target is torsion free. Hence $\Omega_{\Lcirc/\Kcirc}$ is torsion free.

Case 2: $\cha(K)=0$. Consider the algebraic closure $F=K^a$ provided with an extension of the valuation. Let $E$ be the composite valued field $FL$. Then $\Omega_{\Lcirc/\Kcirc}\otimes_{\Lcirc}\Ecirc\into\Omega_{\Ecirc/\Kcirc}$ by \cite[Lemma~6.3.32(ii)]{Gabber-Ramero}, hence it suffices to show that $\Omega_{\Ecirc/\Kcirc}$ is torsion free. The latter follows from the facts that $\Omega_{\Fcirc/\Kcirc}=0$ since $K$ is deeply ramified and $\Omega_{\Ecirc/\Fcirc}$ is torsion free by \cite[Lemma~6.5.20(i)]{Gabber-Ramero}.

Now, let us prove the inverse implication. The case when the valuation of $K$ is non-trivial follows from Corollary~\ref{almosttamecor}: if the torsion module $\Omega^\rmlog_{(K^s)^\circ/\Kcirc}$ is almost torsion free then it almost vanishes and hence $K^s/K$ is almost tame. Assume, now, that the valuation is trivial.
It suffices to prove that if $K$ is not perfect, say $a\in K\setminus K^p$, then there exists a separable extension $L/K$ such that $\Omega^\rmlog_{\Lcirc/\Kcirc}$ contains an essential torsion element. Let $R$ be the localization of $K[t]$ at the maximal ideal generated by $\pi=t^p-a$; it corresponds to a point $x\in\bfA^1_K$ with $k(x)=K(a^{1/p})$. Then $R$ is a discrete valuation ring of $K(t)$ with uniformizer $\pi$ and, since $\bfA^1_K$ is a smooth $K$-curve, $\Omega_{R/K}=Rdt$ is a free module with basis $dt$. The $R$-module $\Omega^\rmlog_{R/K}$ is generated over $\Omega_{R/K}$ by $\delta\pi$ subject to the relation $\pi\delta\pi-d\pi=0$. Since $d\pi=0$, we obtain that $\Omega^\rmlog_{R/K}=R\oplus R/\pi R$ and $\delta\pi$ is a non-trivial torsion element.
\end{proof}

\begin{rem}
It might look surprising that $R$ with the log structure generated by $\pi$ is not log smooth over $K$. The reason for this is that the log structure is geometrically ``non-reduced" over $K$ because $\pi=(t-a^{1/p})^p$ in $R\otimes_KK(a^{1/p})$.
\end{rem}

\subsection{K\"ahler seminorms and field extensions}
Assume that $L/K$ is an extension of real-valued fields and $A\to K$ is a homomorphism of rings. In this section we will apply the theory of real-valued field extensions to compare the K\"ahler seminorms $\lvert \ \rvert_{\Omega,K/A}$ and $\lvert \ \rvert_{\Omega,L/A}$ on $\Omega_{L/A}$ and $\Omega_{K/A}$, respectively.

\subsubsection{The map $\psi_{L/K/A}$}
The two seminorms are related by the non-expansive map $$\psi_{L/K/A}\:\Omega_{K/A}\otimes_KL\to\Omega_{L/A},$$ where the seminorm on the source is the base change of $\lVert \ \rVert_{\Omega,K/A}$. Naturally, we say that $\lVert \ \rVert_{\Omega,K/A}$ and $\lVert \ \rVert_{\Omega,L/A}$ {\em agree} if $\psi_{L/K/A}$ is an isometry. For shortness, the seminorms on both sides of $\psi_{L/K/A}$ will be denoted $\lVert \ \rVert$. This is safe since we consider only one seminorm on each vector space.

To study $\psi_{L/K/A}$ it is useful to consider the following commutative diagram
\begin{equation}\label{eq1}
\xymatrix{
\Omega^\rmlog_{\Kcirc/\Acirc}\otimes_{\Kcirc}\Lcirc\ar@{->>}[d]^{\zeta}\ar[r]^(.6){\psi^\rmlog_{\Lcirc/\Kcirc/\Acirc}} & \Omega^\rmlog_{\Lcirc/\Acirc}\ar@{->>}[d]^\veps\ar[r] & \Omega^\rmlog_{\Lcirc/\Kcirc} \ar[r]\ar@{->>}[d]^\lam& 0\\
(\Omega^\rmlog_{\Kcirc/\Acirc}\otimes_{\Kcirc}\Lcirc)_\tf\ar@{^(_->}[d]^\alp\ar[r]^(.6)\phi & (\Omega^\rmlog_{\Lcirc/\Acirc})_\tf\ar@{^(_->}[d]^\beta\ar[r] & \Coker(\phi)\ar[d]^\gamma\ar[r] & 0\\
\Omega_{K/A}\otimes_KL\ar[r]^(.55){\psi_{L/K/A}} & \Omega_{L/A}\ar[r] & \Omega_{L/K}\ar[r] & 0,
}
\end{equation}
where the top and the bottom rows are the first fundamental sequences. Note that the source and the target of $\phi$ are the almost unit balls of the two seminorms by Theorem~\ref{unitballth}.

\begin{lem}\label{psilem}
Keep the above notation. Then $\psi_{L/K/A}$ is an isometry if and only if $\Ker(\phi)$ is divisible and $\Coker(\phi)$ contains no essential torsion elements.
\end{lem}
\begin{proof}
By Corollary~\ref{cotorcor}, $\psi_{L/K/A}$ is an isometry if and only if $\phi$ is an isometry with respect to the adic seminorm. It remains to use Lemma~\ref{cotorlem}(ii).
\end{proof}

\subsubsection{Separable algebraic extensions}
Note that $\psi_{L/K/A}$ is an isomorphism whenever $L/K$ is separable and algebraic.

\begin{theor}\label{kahlerextth}
Assume that $L/K$ is a separable algebraic extension of real-valued fields and $A\to K$ is a homomorphism.

(i) If $L/K$ is almost tame then the isomorphism $\psi_{L/K/A}$ is an isometry.

(ii) Assume that $\Omega^\rmlog_{\Lcirc/\Acirc}$ is torsion free. Then $\psi_{L/K/A}$ is an isometry if and only if $L/K$ is almost tame. Moreover, the content of the quotient of the unit balls of the two seminorms equals to $\delta^\rmlog_{L/K}$.
\end{theor}
\begin{proof}
Note that $\Omega^\rmlog_{\Lcirc/\Kcirc}$ is a torsion module because $\Omega^\rmlog_{\Lcirc/\Kcirc}\otimes_{\Lcirc}L=\Omega_{L/K}=0$.

(i) Being a quotient of $\Omega^\rmlog_{\Lcirc/\Kcirc}$, the module $\Coker(\phi)$ in diagram (\ref{eq1}) is almost zero. In addition, $\Ker(\phi)\subseteq\Ker(\psi_{L/K/A})=0$. Thus, (i) follows from Lemma~\ref{psilem}.

(ii) The kernel of $\psi^\rmlog_{\Lcirc/\Kcirc/\Acirc}$ almost vanishes by Lemma~\ref{philem}, hence its source is almost torsion free. Thus, $\veps$ is an isomorphism and $\zeta$ is an almost isomorphism, in particular, $\psi^\rmlog_{\Lcirc/\Kcirc/\Acirc}$ is almost isomorphic to $\phi$. By Theorem~\ref{unitballth}, $\psi^\rmlog_{\Lcirc/\Kcirc/\Acirc}$ is almost isomorphic to the embedding of the unit balls. Therefore, the quotient of the unit balls is almost isomorphic to $\Omega^\rmlog_{\Lcirc/\Kcirc}$ giving rise to equality of the contents.
\end{proof}

\subsubsection{Dense extensions}
Another case that will be very important in the sequel is when $K$ is dense in $L$, for example, $K=\kappa(x)$ and $L=\calH(x)$.

\begin{theor}\label{denseth}
Assume that $L/K$ is an extension of real-valued fields such that $K$ is dense in $L$, and let $A\to K$ be a homomorphism of rings. Then the map $\psi_{L/K/A}$ is an isometry with a dense image.
\end{theor}
\begin{proof}
Density of the image follows from Corollary~\ref{densecor}. Set $\chi=\psi^\rmlog_{\Lcirc/\Kcirc/\Acirc}$ for shortness. To prove that $\psi_{L/K/A}$ is an isometry we recall that $\Coker(\chi)$ is a vector space and $\Ker(\chi)$ is divisible by Lemma~\ref{denseloglem}. Let $\chi_\tor$ and $\chi_\tf$ be the maps $\chi$ induces between the torsion submodules and the torsion free quotients of its arguments. In particular, $\chi_\tf$ is the map $\phi$ from diagram (\ref{eq1}). The snake lemma yields an exact sequence $$\Ker(\chi)\stackrel{\beta}\to\Ker(\phi)\to\Coker(\chi_\tor)\stackrel{\alp}\to\Coker(\chi)\to\Coker(\phi)\to 0.$$
Since $\Coker(\chi)$ is torsion free and $\Coker(\chi_\tor)$ is torsion, $\alp=0$ and so $\Coker(\phi)=\Coker(\chi)$ is a vector space. In addition, the torsion free group $\Ker(\phi)$ is an extension of the divisible group $\Im(\beta)$ by the torsion group $\Coker(\chi_\tor)$. It follows easily that, in fact, $\Ker(\phi)=\Im(\beta)$. Thus, $\Ker(\phi)$ is divisible and hence $\psi_{L/K/A}$ is an isometry by Lemma~\ref{psilem}.
\end{proof}

\begin{cor}\label{densecor2}
In the situation of Theorem~\ref{denseth}, the completion of $\hatpsi_{L/K/A}$ is an isometric isomorphism. In particular, $\hatOmega_{K/A}\toisom\hatOmega_{L/A}$.
\end{cor}

\subsection{K\"ahler seminorms and monomial valuations}\label{orthosec}
In this section we study finitely generated extensions $L/K$ such that $\Lcirc/\Kcirc$ behaves similarly to log smooth extensions, though it does not have to be finitely presented.

\subsubsection{Orthonormal bases of $\Omega_{L/K}$}
Note that if $t_1,\dots,t_n$ is a separable transcendence basis of a field extension $L/K$ then $dt_1\..dt_n$ is a basis of $\Omega_{L/K}$. The following result is an immediate corollary of Theorem~\ref{unitballth}.

\begin{lem}\label{orthlem}
Given an extension of real-valued fields $L/K$ with a separable transcendence basis $t_1{,\dots ,}t_n$ consider the following conditions:

(i) $\Omega^\rmlog_{\Lcirc/\Kcirc}$ is a free $\Lcirc$-module with basis $\delta t_1{,\dots ,}\delta t_n$.

(ii) $\frac{dt_1}{t_1}{,\dots ,}\frac{dt_n}{t_n}$ is an orthonormal basis of $\Omega_{L/K}$.

Then (i)$\implies$(ii) and the conditions are equivalent whenever $\Omega^\rmlog_{\Lcirc/\Kcirc}$ is torsion free.
\end{lem}

\subsubsection{Generalized Gauss valuations}\label{genGausssec}
We will use the underline to denote tuples, e.g. $\ur=(r_1\..r_n)$. Assume that $K$ is a real-valued field and $A=K[\ut]$ for $\ut=(t_1{,\dots ,}t_n)$. For any tuple $\ur=(r_1{,\dots ,}r_n)\in\bfR_{>0}^n$ by $\lvert \ \rvert_\ur$ we denote the generalized Gauss valuation on $A$ defined by the formula $$\left\lvert \sum_{i\in\bfN^n} a_i\ut^i\right\rvert_\ur=\max_i\left(\lvert a_i\rvert \ur^i\right)=\max_i\left(\lvert a_i\rvert\prod_{j=1}^nr_j^{i_j}\right).$$ The normed ring $(A,\lvert \ \rvert_\ur)$ will be denoted $K[\ut]_\ur$. Since $\lvert \ \rvert_\ur$ is multiplicative it extends to a norm on $K(\ut)$ that will be denoted by the same letter, and we use the notation $(K(\ut),\lvert \ \rvert_\ur)=K(\ut)_\ur$. The following lemma indicates that $(K(\ut)_\ur)^\circ/\Kcirc$ behaves as a log smooth extension.

\begin{lem}\label{gausslem}
Assume that $K$ is a real-valued field, $r_1{,\dots ,}r_n>0$ and $L=K(t_1{,\dots ,}t_n)_\ur$. Then $\Omega^\rmlog_{\Lcirc/\Kcirc}$ is a free $\Lcirc$-module with basis $\delta t_1{,\dots ,}\delta t_n$.
\end{lem}
\begin{proof}
By Lemma~\ref{logderlem} it suffices to show that for any $\Lcirc$-module $M$ and elements $m_1{,\dots ,}m_n$ there exists a unique log $\Kcirc$-derivation $(d,\delta)\:(\Lcirc,\Lcirc\setminus\{0\})\to M$ such that $\delta t_i=m_i$. If $u\in K[\ut]$ satisfies $\lvert u\rvert_\ur=1$ then $u\in(\Lcirc)^\times$ and $u=\sum_{l\in\bfN^n} a_l\ut^l$, where $\lvert a_l\ut^l\rvert \le 1$, so we set $$\delta u=u^{-1}\left(\sum_{l\in\bfN^n}\sum_{i=1}^n l_ia_l\ut^lm_i\right).$$ Since $\lvert L^\times\rvert=\lvert K^\times\rvert r_1^\bfZ\dots r_n^\bfZ$, an arbitrary element $z\in\Lcirc$ is of the form $a\ut^lu/v$ for $l\in\bfZ^n$, $a\in K$ and $u,v\in K[\ut]$ with $\lvert u\rvert_\ur=\lvert v\rvert_\ur=1$, so we set $\delta z=l\delta t+\delta u-\delta v$ and $dz=z\delta z$. It is a direct check that the so defined $(d,\delta)$ is a log $\Kcirc$-derivation. Any other log $\Kcirc$-derivation should satisfy the same formulas for $\delta u$, $\delta z$ and $dz$, so uniqueness is clear.
\end{proof}

\subsubsection{A characterization of Gauss valuations}
Under mild technical assumptions, one can also characterize generalized Gauss valuations in terms of $\Omega^\rmlog_{\Lcirc/\Kcirc}$.

\begin{lem}\label{inversegausslem}
Let $L=K(t_1{,\dots ,}t_n)/K$ be a purely transcendental extension of real-valued fields, let $r_i=\lvert t_i\rvert$, let $p=\expchar(\tilK)$, and assume that $\tilK$ is perfect and $\lvert K^\times\rvert$ is $p$-divisible. If $\frac{dt_1}{t_1}{,\dots ,}\frac{dt_n}{t_n}$ is an orthonormal basis of $\Omega_{L/K}$ then $L=K(t_1{,\dots ,}t_n)_\ur$ as a valued field.
\end{lem}
\begin{proof}
Assume that, conversely, the valuation $\lvert \ \rvert$ on $L$ is not generalized Gauss with respect to $\ut$. Then the valuation on $K[\ut]$ is strictly dominated by $\lvert \ \rvert_\ur$, in particular, there exists a polynomial $f(\ut)=\sum_{l\in\bfN^n} a_l\ut^l$ in $\Lcirc$ such that $\lvert f\rvert <\lvert f\rvert_\ur=\max_l\lvert a_l\rvert \ur^l$. Removing all monomials of absolute value strictly smaller than $\lvert f\rvert_\ur$ we can achieve that the following condition holds: (*) $\lvert f\rvert <\lvert f\rvert_\ur$ and $\lvert a_l\ut^l\rvert =\lvert f\rvert_\ur$ for each $a_l\neq 0$.

Choose $f$ satisfying (*) and of minimal possible degree. We claim that if $p>1$ then not all monomials are $p$-th powers. Indeed, assume that the claim fails, say $f=\sum_{l\in p\bfN^n} a_l\ut^l$. By our assumption on $K$, for any $a\in K$ there exists $b\in K$ with $|a-b^p|<|a|$. Indeed, since $|K^\times|$ is $p$-divisible we have that $a=x^py$ with $|y|=1$, and using that $\tilK$ is perfect we can find $z\in K$ with $\tily=\tilz^p$. Then $b=xz$ is as required. Now, for any $a_l$ fix $b_l$ such that $\lvert b_l^p-a_l\rvert <\lvert a_l\rvert$. Clearly, $\sum_{l\in p\bfN^n} b_l^p\ut^l$ satisfies (*) and hence $\sum_{l\in\bfN^n} b_{pl}\ut^l$ also satisfies (*), but its degree is smaller than that of $f$. It follows that there exists $l\in\bfN^n$ and $1\le i\le n$ such that $a_l\neq 0$ and $\lvert l_i\rvert =1$.

Since the basis $\frac{dt_1}{t_1}{,\dots ,}\frac{dt_n}{t_n}$ of $\Omega_{L/K}$ is orthonormal, $\lvert f\rvert \ge\lVert df\rVert_\Omega$ and
$$df=\sum_{l\in\bfN^n}\sum_{i=1}^n l_ia_l\ut^l\frac{dt_i}{t_i},$$ the inequality $\lvert l_ia_lt^l\rvert \le\lvert f\rvert$ holds for any choice of $l$ and $i$. However, $\lvert l_ia_l\ut^l\rvert =\lvert a_l\ut^l\rvert =\lvert f\rvert_\ur>\lvert f\rvert$ for the choice with $a_l\neq 0$ and $\lvert l_i\rvert =1$, a contradiction.
\end{proof}

\subsubsection{$\ut$-monomial valuations}
Let $L/K$ be an extension of real-valued fields and $\ut=(t_1{,\dots ,}t_n)$ a tuple of elements of $L$. We say the valuation on $L$ is {\em $\ut$-monomial} with respect to $K$ if the induced valuation on $k[\ut]$ is a generalized Gauss valuation. This happens if and only if $L$ contains the valued subfield $K(\ut)_\ur$ where $r_i=\lvert t_i\rvert$. All results proved in Section~\ref{orthosec} can be summarized as follows.

\begin{theor}\label{tmonth}
Assume that $L/K$ is an extension of real-valued fields with a separable transcendence basis $t_1{,\dots ,}t_n$. Consider the following conditions:

(i) The valuation on $L$ is $\ut$-monomial and $\Omega^\rmlog_{\Lcirc/K(\ut)^\circ}=0$.

(ii) $\Omega^\rmlog_{\Lcirc/\Kcirc}$ is a free $\Lcirc$-module with basis $\delta t_1{,\dots ,}\delta t_n$.

(iii) $\frac{dt_1}{t_1}{,\dots ,}\frac{dt_n}{t_n}$ is an orthonormal basis of $\Omega_{L/K}$.

(iv) $\lVert \frac{dt_1}{t_1}\wedge\dots\wedge\frac{dt_n}{t_n}\rVert_{\Omega^n}=1$ in $\Omega^n_{L/K}$.

Then (i)$\implies$(ii)$\implies$(iii)$\Longleftrightarrow$(iv). Furthermore, if $K$ is almost tame then all four conditions are equivalent.
\end{theor}
\begin{proof}
(i)$\implies$(ii) Set $F=K(\ut)$. By the first fundamental sequence, we have a surjective map $\psicirc\:\Omega^\rmlog_{\Fcirc/\Kcirc}\otimes_{\Fcirc}\Lcirc\to\Omega^\rmlog_{\Lcirc/\Kcirc}$, which becomes the isomorphism $\psi\:\Omega_{F/K}\otimes_FL\toisom\Omega_{L/K}$ after tensoring with $L$. In particular, the kernel of $\psicirc$ is torsion. By Lemma~\ref{gausslem}, the source of $\psicirc$ is a free module with basis $\delta t_1{,\dots ,}\delta t_n$. So, the kernel of $\psicirc$ is trivial, and hence $\psicirc$ is an isomorphism.

(ii)$\implies$(iii) This is covered by Lemma~\ref{orthlem}.

(iii)$\Longleftrightarrow$(iv) The direct implication is obvious. Conversely, assume that (iii) fails. Then the lattice $M$ generated by the elements $e_i=\frac{dt_i}{t_i}$ is strictly smaller than $\Omega_{L/K}^\di$ and hence there exists a lattice $M'$ such that $M\subsetneq M'\subseteq\Omega_{L/K}^\di$. Choose a basis $e'_1\..e'_n$ of $M'$ then $$1\ge\lVert e'_1\wedge\dots\wedge e'_n\rVert_{\Omega^n}=[M':M]\cdot\lVert e_1\wedge\dots\wedge e_n\rVert_{\Omega^n},$$ and since $[M':M]>1$ we obtain that (iv) fails.

Finally, assume that $K$ is almost tame, in particular, $\Omega^\rmlog_{\Lcirc/\Kcirc}$ is torsion free. Then the implication (iii)$\implies$(ii) follows from Lemma~\ref{orthlem}. Assume, now, that (ii) holds; in particular, $\psicirc$ is surjective and hence $\Omega^\rmlog_{\Lcirc/\Fcirc}=0$. Furthermore, the isomorphism $\psi$ is non-expansive and $\lVert \frac{dt_i}{t_i}\rVert_\Omega\le 1$ in its source. Therefore, $\psi$ is an isometry and $\frac{dt_1}{t_1}{,\dots ,}\frac{dt_n}{t_n}$ is an orthonormal basis of $\Omega_{F/K}$. By Lemma~\ref{inversegausslem}, the valuation on $F$ is generalized Gauss, i.e. $L$ is $\ut$-monomial.
\end{proof}

\begin{rem}
The assumption that $K$ is almost tame in Theorem~\ref{tmonth} is needed for the implication (iii)$\implies$(ii) even when $n=1$ and $K$ is trivially valued. For example, assume that $p=\cha(K)>0$ and $a\in\Kcirc$ is such that $\tila\notin\tilK^p$ and define the norm on $L=K(t)$ so that $\lvert t^p-a\rvert =r$ for some $r\in(0,1)$ and the norm is $(t^p-a)$-monomial. Then one can check that $\Omega^\rmlog_{\Lcirc/\Kcirc}$ is a direct sum of $\Lcirc\delta t$ and an essential torsion submodule generated by $\delta(t^p-a)$, and hence $\frac{dt}t$ is of norm one. Note that $\Omega^\rmlog_{\Lcirc/\Kcirc}$ is not free in this case, although its torsion free quotient is free with basis $\delta t$. In the particular case of the trivial valuation, the same example in a non-complete setting was already used in the end of proof of Theorem~\ref{torth}.
\end{rem}

\section{Metrization of $\Omega_{X/S}$}\label{metrsec}
Throughout Section \ref{metrsec}, $f\:X\to S$ denotes a morphism of $k$-analytic spaces. In the case of sheaves, $\lvert \ \rvert$ will always denote spectral seminorms on the structure sheaves and their stalks, and $\lVert \ \rVert$ will always denote K\"ahler seminorms (defined below) on sheaves of pluriforms and their stalks.

\subsection{K\"ahler seminorm on $\Omega_{X/S}$}
In this section we introduce a seminorm $\lVert \ \rVert =\lVert \ \rVert_{\Omega,X/S}$ on $\Omega_{X/S}$. We will mention $\Omega$ and $X/S$ in the notation only when a confusion is possible.

%For this we should define seminorms $\lVert \ \rVert_{x}$ on the stalks and establish their semicontinuity. The basic idea is very simple: given $x\in X_G$ and $s=f(x)$ we simply induce the stalk seminorm at $x\in X$ from the K\"ahler norm on $\hatOmega_{\calH(x)/\calH(s)}$.

\subsubsection{The definition}\label{kahlerdefsec}
The construction is straightforward: we simply sheafify the presheaf of K\"ahler seminorms on affinoid algebras. Let $\calC$ be the full subctegory of $X_G$ whose objects are affinoid domains $V=\calM(\calB)$ in $X$ such that $f(V)$ is contained in an affinoid domain $U=\calM(\calA)\subseteq S$. Note that $\Omega_{X/S}(V)=\hatOmega_{\calB/\calA}$, in particular, $\hatOmega_{\calB/\calA}$ is independent of the choice of $U$. Furthermore, the K\"ahler seminorm on $\hatOmega_{\calB/\calA}$ is the quotient of the K\"ahler seminorm on $\hatOmega_{\calA/k}$ by Lemma~\ref{fundseclem}(i), hence it is independent of $U$ too and we can denote it $\lVert \ \rVert '_{\calB/S}$. This construction produces a locally bounded pre-quasi-norm $\lVert \ \rVert '$ on the restriction of $\Omega_{X/S}$ to $\calC$. Its sheafification is a seminorm on $\Omega_{X/S}\vert_\calC$, and we can extend this seminorm to a seminorm $\lVert \ \rVert =\lVert \ \rVert _{\Omega,X/S}$ on the whole $\Omega_{X/S}$ since $\calC$ is cofinal in $X_G$. We call $\lVert \ \rVert _{\Omega,X/S}$ the {\em K\"ahler seminorm} of $X/S$.

\begin{rem}
(i) By definition, if $V$ is a domain in $X$ and $\phi\in\Omega_{X/S}(V)$ then $\lVert \phi\rVert_V =\inf\max_i\lVert \phi\rVert'_{\calA_i/S}$, where the infimum is over admissible coverings $V=\cup_i V_i$ with $V_i=\calM(\calA_i)$ in $\calC$.

(ii) We do not study the question whether $\lVert \ \rVert _V=\lVert \ \rVert '_V$ for any $V$ in $\calC$ (i.e., whether $\lVert \ \rVert '$ is already a seminorm on $\Omega_{X/S}\vert_\calC$). This is not essential for our needs because all results about $\lVert \ \rVert $ will be proved using stalks.
\end{rem}

\subsubsection{The universal property}
The universal property satisfied by K\"ahler seminorms of seminormed rings, see Lemma~\ref{universemilem}(i), and the universal property of sheafification, see Lemma~\ref{sheafiflem}, imply the following characterization of $\lVert \ \rVert $.

\begin{lem}\label{univkahlem}
The K\"ahler seminorm $\lVert \ \rVert _{\Omega,X/S}$ is the maximal seminorm on $\Omega_{X/S}$ that makes the differential $d\:\calO_X\to\Omega_{X/S}$ a non-expansive map.
\end{lem}

\subsubsection{Unit balls}
There is an alternative approach to K\"ahler seminorms via unit balls. It is more elementary but less robust and we will not use it. For the sake of completeness we outline it in the following remark leaving some simple verifications to the interested reader.

\begin{rem}\label{unitballrem}
(i) If $\lvert k^\times\rvert $ is dense then Lemma \ref{univkahlem} provides a simple way to describe the unit ball $\Omega^\di_{X/S}$ of $\lVert \ \rVert $, and we obtain another way to define $\lVert \ \rVert $. This involves the sheaf-theoretic extension of the terminology of Section~\ref{semivectsec}. Let $\calB_{X/S}=\calO^\circ_{X_G}d\calO^\circ_{X_G}$ denote the subsheaf of $\calO^\circ_{X_G}$-modules of $\Omega_{X/S}$ generated by $d\calO^\circ_{X_G}$. Lemma \ref{univkahlem} implies that $\lVert \ \rVert $ is the maximal seminorm whose unit ball contains $d\calO^\circ_{X/G}$, hence $\calB_{X/S}$ is an almost unit ball of $\lVert \ \rVert $ and the unit ball $\Omega^\di_{X/S}$ is the almost isomorphic envelope of $\calB_{X/S}$.

(ii) Let $k$ be arbitrary. Already for the unit disc $E=\calM(k\{t\})$, the inclusion $i\:\calB_{E/k}\into\Omega^\di_{E/k}$ is usually not an isomorphism. For example, if $x$ is the maximal point of a disc around zero of radius $r\notin\lvert k^\times\rvert ^\bfQ$ then $\frac{dt}{t}$ is contained in $\Omega^\di_{E/k,x}$ but not in $\calB_{E/k,x}$. Moreover, if $\lvert k^\times\rvert $ is discrete then $i$ is not even an almost isomorphism. This example suggests that one can improve the situation by adding logarithmic differentials of units, so we set $$\calB_{X/S}^\rmlog=\calB_{X/S}+O^\circ_{X_G}\delta\calO^\times_{X/S},$$ where $\delta\:\calO^\times_{X/S}\to\Omega_{X/S}$ is the logarithmic differential. It is easy to see that $\calB_{X/S}^\rmlog\subseteq\Omega^\di_{X/S}$ and the inclusion is an equality for $E/k$. It is an interesting question whether $\calB_{X/S}^\rmlog=\Omega^\di_{X/S}$ in general.
\end{rem}

\subsubsection{The stalks and the fibers}\label{metrstalksec}
Our next aim is to study local behavior of the K\"ahler seminorm $\lVert \ \rVert$ at points of $X$. Given $x\in X$ with $s=f(x)$, fix an affinoid domain $U=\calM(\calA)$ containing $s$ and let $\{V_\lam=\calM(\calB_\lam)\}_\lam$ be the family of affinoid domains in $X$ such that $x\in V_\lam$ and $f(V_\lam)\subseteq U$. Provide $\Omega_{X/S,x}$ with the stalk seminorm $\lVert \ \rVert_x$, then we saw in \S\ref{kahlerdefsec} that $\Omega_{X/S,x}$ is the filtered colimit of the seminormed $\calA$-modules $\hatOmega_{\calB_\lam/\calA}$ (see \S\ref{normcolimsec}). In fact, $\Omega_{X/S,x}$ is an (uncompleted) filtered colimit of completed modules of differentials, so it can be informally thought of as a partial completion of $\Omega_{\calO_{x}/\calO_{s}}$, where we set $\calO_x=\calO_{X_G,x}$ and $\calO_s=\calO_{S_G,s}$ for shortness. Consider the following commutative diagram of seminormed modules
\begin{equation}\label{eq2}
\xymatrix{
\Omega_{\calB_\lam/\calA}\ar@{->>}[r]^{\alp_\lam}\ar[d]^{\beta_\lam}& \hatOmega_{\calB_\lam/\calA}\ar[d]^{\gamma_\lam}\ar[rd]^{\psi_\lam} & \\
\Omega_{\calO_{x}/\calO_{s}}\ar[r]^{\alp_x}& \Omega_{X/S,x}\ar[r]^(0.4){\psi_x}& \hatOmega_{\calH(x)/\calH(s)},
}
\end{equation}
where $\psi_x$ is the non-expansive $\calA$-homomorphism induced by the $\calA$-homomorphisms $\psi_\lam$ via the universal property of colimits.

\begin{theor}\label{isometryth}
Keep the above notation. Then $\alp_x$ and $\psi_x$ are isometries with dense images. In particular, $\hatOmega_{\calH(x)/\calH(s)}$ is the completion of both $\Omega_{X/S,x}$ and $\Omega_{\calO_{x}/\calO_{s}}$.
\end{theor}
\begin{proof}
By Lemma~\ref{normcolimlem}, $\Omega_{\calO_{x}/\calO_{s}}=\colim_\lam\Omega_{\calB_\lam/\calA}$ as seminormed modules. In particular, $\alp_x$ is the colimit of isometries with dense images $\alp_\lam$, and hence $\alp_x$ itself is an isometry with a dense image. Therefore, it suffices to show that the map $\Omega_{\calO_{x}/\calO_{s}}\to\hatOmega_{\calH(x)/\calH(s)}$ is the completion homomorphism. Indeed, the surjection $\Omega_{\calO_{x}/\calO_{s}}\to\Omega_{\kappa_G(x)/\kappa_G(s)}$ is an isometry by Corollary~\ref{fundcor}
and $\hatOmega_{\kappa_G(x)/\kappa_G(s)}=\hatOmega_{\calH(x)/\calH(s)}$ by Corollary~\ref{densecor2}.
\end{proof}

\begin{cor}\label{inducecor}
In the situation of Theorem~\ref{isometryth}, $\psi_x$ identifies the completed fiber $\wh{\Omega_{X/S}(x)}$ with $\hatOmega_{\calH(x)/\calH(s)}$.
\end{cor}

\begin{rem}
(i) Corollary~\ref{inducecor} provides an alternative way to define $\lVert \ \rVert_{x}$. Instead of the colimit definition in \S\ref{metrstalksec}, one can simply induce $\lVert \ \rVert_x$ from $\hatOmega_{\calH(x)/\calH(s)}$ by the rule $\lVert \phi\rVert_{x}=\lVert \psi_x(\phi)\rVert_{\hatOmega,\calH(x)/\calH(s)}$ for $\phi\in\Omega_{X/S,x}$.

(ii) Even more importantly, the corollary provides a convenient way to compute the values of $\lVert \ \rVert_x$ since the finite-dimensional normed vector spaces $\hatOmega_{\calH(x)/\calH(s)}$ are often pretty explicit. For comparison, we note that it is not clear how to describe the fiber $\Omega_{X/S}(x)$ in terms of $\kappa_G(x)/\kappa_G(s)$. One can only say that $\Omega_{X/S}(x)$ is a finite-dimensional quotient, in fact, a partial completion of the huge vector space $\Omega_{\kappa_G(x)/\kappa_G(s)}$. However, the seminorm of $\Omega_{X/S}(x)$ can still have a non-trivial kernel.
\end{rem}

\subsubsection{K\"ahler seminorms on pluriforms}
By the constructions of \S\ref{opersec}, $\lVert \ \rVert_\Omega$ induces seminorms on the sheaves obtained from $\Omega_{X/S}$ by tensor products, symmetric powers and exterior powers. In particular, it induces a canonical seminorm $\lVert \ \rVert_{(\Omega_X^l)^{\otimes m}}$ on the sheaf of pluriforms $(\Omega_X^l)^{\otimes m}$, that will be called K\"ahler seminorm too.

Of particular interest will be the situation when $X\to S$ is quasi-smooth of relative dimension $n$. Then the relative canonical sheaf $\omega_{X/S}=\bigwedge^n\Omega_{X/S}$ is invertible, as well as the relative pluricanonical sheaves $\omega_{X/S}^{\otimes m}$. The corresponding seminorms will be denoted $\lVert \ \rVert_\omega$ and $\lVert \ \rVert_{\omega^{\otimes m}}$.

\subsubsection{Analyticity of the seminorms}
We have already used Corollary~\ref{densecor2} when studying the stalk seminorms of $\lVert \ \rVert $. As another application, let us show that all seminorms we have constructed are analytic.

\begin{theor}\label{analyticomega}
The seminorm $\lVert \ \rVert_\Omega$ and the induced seminorms on the sheaves obtained from $\Omega_{X/S}$ by tensor products, symmetric powers and exterior powers are analytic.
\end{theor}
\begin{proof}
By Remark~\ref{ansemrem}(ii), analyticity is $G$-local, hence it suffices to consider the case when $X$ is affinoid. In the sequel $\Omega_{X/S}$ denotes the $\calO_X$-module and the $\calO_{X_G}$-module will be denoted $\Omega_{X_G/S_G}$. By Lemma~\ref{cohlem}(i)$\Longleftrightarrow$(iii) we should prove that for each $x\in X$ the map $h\:\Omega_{X/S,x}\to\Omega_{X_G/S_G,x}$ is an isometry with respect to the stalks of $\lVert \ \rVert_\Omega$. It suffices to check that the completion of $h$ is an isometry and we know by Theorem~\ref{isometryth} that the completion of $\Omega_{X_G/S_G,z}$ is $\hatOmega_{\calH(x)/\calH(s)}$. So, it suffices to check that $\hatOmega_{\calH(x)/\calH(s)}$ is also the completion of $\Omega_{X/S,x}$, and for this we will copy the argument from the proof of Theorem~\ref{isometryth} but with sheaves in the usual topology.

Set $s=f(x)$, $\calO_x=\calO_{X,x}$ and $\calO_s=\calO_{S,s}$, and provide $\Omega_{\kappa(x)/\kappa(s)}$ with the K\"ahler seminorm. We claim that $\hatOmega_{\calH(x)/\calH(s)}$ is the completion of $\Omega_{\calO_x/\calO_s}$. Indeed, the surjection $\Omega_{\calO_x/\calO_s}\to\Omega_{\kappa(x)/\kappa(s)}$ is an isometry by Corollary~\ref{fundcor} and $\hatOmega_{\kappa(x)/\kappa(s)}=\hatOmega_{\calH(x)/\calH(s)}$ by Corollary~\ref{densecor2}.
\end{proof}

\subsection{Examples}\label{examsec}
In this section, we compute $\lVert \ \rVert_\Omega$ and its completed fibers in a few basic cases. We try to choose simple examples that illustrate the general situation. In particular, we will see that $\lVert \ \rVert_\Omega$ discovers a rather subtle behavior even in the one-dimensional case.

\subsubsection{The case of a disc}\label{discsec}
Assume that $k=k^a$ and $X=\calM(k\{T\})$ is the unit disc. Then $\Omega_X$ is a free sheaf with basis $dT$, so we can identify it with $\calO_X$ by sending $dT$ to $1$. Let $r(x)$ be the radius function, i.e. $r(x)$ is the infimum of radii of subdiscs of $X$ containing $x$. We claim that $\lVert dT\rVert_x=r(x)$ for any $x\in X$. If $x$ is contained in a disc of radius $s$ with center at $a$ then $\lvert T-a\rvert_x\le s$ and hence $\lVert dT\rVert_x=\lVert d(T-a)\rVert_x\le s$. The function $\lVert dT\rVert \:X\to\bfR_{\ge 0}$ is semicontinuous by Theorem~\ref{analyticomega}, therefore it suffices to check that $\lVert dT\rVert_x\ge r(x)$ at a type 2 or 3 point $x$. In such case, replacing $T$ by a suitable $T-a$ with $a\in k$ we can achieve that $\lvert T\rvert_x=r(x)$ and the valuation on $\calH(x)$ is $T$-monomial. But then $\lVert \frac{dT}{T}\rVert_{\Omega,\calH(x)/k}=1$ by Theorem~\ref{tmonth}(i)$\implies$(iii), and so $\lVert dT\rVert_x=\lvert T\rvert_x=r(x)$.

The formula for $\lVert dT\rVert$ implies that its maximality locus consists of a single point, the maximal point of $X$. Another consequence is that $\lVert \ \rVert$ is the seminorm corresponding to $r(x)$ in the sense of Lemma~\ref{rankonelem}.

\begin{rem}
The radius function $r(x)$ is upper semicontinuous but not continuous, and this is a typical behavior for functions of the form $\lVert \phi\rVert$. This indicates that the K\"ahler seminorm on $\Omega_{X/S}$ is very different from the spectral seminorm on $\calO_{X_G}$ even when $\Omega_{X/S}$ is invertible. For example, if $X$ is a curve and $f\in\Gamma(\calO_{X_G})$ is a global function then $\lvert f\rvert$ is locally constant outside of a finite graph. On the other hand, if $\phi\in\Gamma(\Omega_X)$ then for any type 2 point $x$ the value of $\lVert \phi\rVert$ decreases in almost all directions leading from $x$. This property is tightly related to the fact that the maximality locus of such $\phi$ is a finite graph. Also, this indicates that the unit balls $\Omega^\di_{X/S}$ are usually huge $\calO^\circ_{X_G}$-modules.
\end{rem}

\subsubsection{Rigid points: perfect ground field}\label{rigsec}
Assume that $k$ is perfect. Let $x\in X$ be any point with $\calH(x)\subseteq\whka$, for example, a rigid point (i.e. a point with $[\calH(x):k]<\infty$), or a type 1 point on a curve. Note that $\hatOmega_{\calH(x)/k}=0$, since $k^a\cap\calH(x)$ is dense in $\calH(x)$ by Ax-Sen-Tate theorem, see \cite{Ax}. Therefore, any differential form $\phi$ satisfies $\lVert \phi\rVert_{\Omega,x}=0$ by Corollary~\ref{inducecor}.

\subsubsection{Rigid points: non-perfect ground field}\label{rignonperf}
If $k$ is not perfect then the situation is different. For example, assume that $x$ is a rigid point with $l=\calH(x)$ inseparable over $k$; a disc contains plenty of such points, e.g. the points given by $T^p-a=0$ for $a\in k\setminus k^p$. One can easily give examples when K\"ahler seminorm on the finite dimensional vector space $\Omega_{l/k}$ is a norm and hence $\hatOmega_{l/k}=\Omega_{l/k}\neq 0$. (Probably, this is always the case since $k$ is complete.) On the other hand, $\hatOmega_{l/k}$ is the completion of $\Omega_{X/k,x}$, so the seminorm on the latter does not vanish.

Let us outline a concrete particular case. Assume that $l=k(\alp)$ is a purely inseparable extension of $k$ of degree $p$; in particular, $\Omega_{l/k}=ld\alp$. Then $r=\inf_{c\in k}\lvert \alp-c\rvert$ is positive since $k$ is complete. Since $d\alp=d(\alp-c)$, we obviously have that $\lVert d\alp\rVert_{\Omega,l/k}\le r$. One can check straightforwardly that the choice of $\lVert d\alp\rVert =r$ makes the map $d_{l/k}$ non-expansive, so, in fact, $\lVert d\alp\rVert_{\Omega,l/k}=r$. In particular, if $T$ is the coordinate on $X=\bfA^1_k$ and $x\in X$ is the type 1 point given by $T^p=a$, then $\lVert dT\rVert_{x}=s^{1/p}$, where $s=\inf_{c\in k}\lvert a-c^p\rvert$.

\subsubsection{Disc over non-perfect field}\label{discnonperf}
Assume that $\cha(k)=p>0$ and $\tilk$ is not perfect and let $X=\calM(k\{T\})$. Choose any $a\in\kcirc$ such that $\tila\notin\tilk^p$ and let $x$ be the rigid point given by $T^p=a$. Then $\lVert dT\rVert_x=1$ by \S\ref{rignonperf}, and we claim that, more generally, $\lVert dT\rVert =1$ on the whole line connecting $x$ with the maximal point of $X$. In particular, the maximum locus of $\lVert dT\rVert$ contains a huge subgraph, whose combinatorial cardinality (i.e. the cardinality of the set $V\cup E$ of vertices and edges) is easily seen to be equal to the cardinality of $k$ (it is infinite since $\tilk\neq\tilk^p$).

To verify our claim, let $q$ be the maximal point of the disc around $x$ given by $\lvert T^p-a\rvert \le r$. Set $S=T^p-a$ and $L=\wh{k(S)}\subseteq\calH(q)$, then the valuation on $L$ is $S$-monomial and hence $\inf_{c\in L}\lvert S+a-c^p\rvert =1$. Since $T=(a+S)^{1/p}$ we obtain by \S\ref{rignonperf} that $\lVert dT\rVert_{\Omega,\calH(q)/L}=1$, and hence $\lVert dT\rVert_{\Omega,\calH(q)/k}\ge 1$. The opposite inequality is obvious, so $\lVert dT\rVert_{q}=\lVert dT\rVert_{\Omega,\calH(q)/k}=1$.

\begin{rem}\label{mixedrem}
In the mixed characteristic case the situation is even weirder. On the one hand, the K\"ahler seminorm vanishes at rigid points, but on the other hand, if $\tilk$ is not perfect, say $\tila\notin\tilk^p$, and $I$ is the interval connecting $x=(T^p-a)$ with the maximal point $q$ then a similar argument shows that $\lVert dT\rVert =1$ on some neighborhood of $q$ in $I$. In fact, the maximality locus of $\lVert dT\rVert$ is a huge tree with root $q$ but its leaves are points of type 2.
\end{rem}

\subsection{K\"ahler seminorms and base changes}\label{changesec}

\subsubsection{Domination}
For general base changes, K\"ahler seminorms are related as follows.

\begin{lem}\label{domlem}
Assume that $X\to S$ and $S'\to S$ are morphisms of $k$-analytic spaces and $X'=X\times_SS'$. Let $\phi\in\Gamma(\Omega_{X/S})$ and let $\phi'\in\Gamma(\Omega_{X'/S'})$ be the pullback of $\phi$. Then for any point $x'\in X'$ with image $x\in X$ one has that $\lVert \phi'(x')\rVert \le\lVert \phi(x)\rVert$. In other words, the pullback of $\lVert \ \rVert_{\Omega,X/S}$ (see \S\ref{pullsec}) dominates $\lVert \ \rVert_{\Omega,X'/S'}$.
\end{lem}
\begin{proof}
Let $s\in S$ and $s'\in S'$ be the images of $x'$. Unrolling the definitions of the K\"ahler seminorms and the pullback operation we see that the assertion of the lemma reduces to the claim that the map $\whOmega_{\calH(x)/\calH(s)}\otimes_{\calH(x)}\calH(x')\to\whOmega_{\calH(x')/\calH(x)}$ is non-expansive. The latter follows straightforwardly from the definition of the seminorm on the source: similarly to the argument in Lemma~\ref{basechangelem}, all inequalities defining the seminorm of the source hold in the target.
\end{proof}

\begin{rem}\label{domrem}
(i) The domination of seminorms from Lemma~\ref{domlem} is not an equality in general. We will later see that this is often the case when $X$ is a disc, $S=\calM(k)$ and $S'=\calM(l)$ for a finite wildly ramified extension $l/k$. One can also provide a characteristic-free example. Assume that $k=k^a$, $S=\calM(k)$ and $X=S'$ is the unit disc over $k$, and let $x=s'$ be a point. It is easy to see that the fiber over $(x,s')$ in $X'$ contains a type 1 point $x'$. So, $\lVert \ \rVert_{x'}=0$ by \S\ref{rigsec}. On the other hand, if $r(x)>0$ then $\lVert \ \rVert_x\neq 0$.

(ii) At first glance, the above examples are surprising because in the context of seminormed rings, K\"ahler seminorms are compatible with base changes by Lemma~\ref{basechangelem}. However, we use structure sheafs provided with the spectral seminorms in the definition of K\"ahler seminorms, while the tensor seminorms do not have to be spectral. For example, if $E/k$ and $F/k$ are finite extensions of analytic fields such that $K=E\otimes_k F$ is a field, the tensor norm on $K$ can be strictly larger than the valuation. As we will prove below, this phenomenon may only happen when the extensions are wild.
\end{rem}

\subsubsection{Universally spectral norms}
Remark \ref{domrem}(ii) motivates the following definition. We say that a Banach $k$-algebra $\calA$ is {\em spectral} if its norm is power-multiplicative. In other words, $\lvert \ \rvert_\calA$ coincides with the spectral seminorm. We say that a Banach $k$-algebra is {\em universally spectral} if $\calA\wtimes_kl$ provided with the tensor seminorm is spectral for any extension of analytic fields $l/k$.

\begin{rem}\label{unispecrem}
(i) The condition that a Banach algebra is spectral is an analogue of reducedness in the usual ring theory. Thus, universal spectrality can be viewed as an analogue of geometric reducedness over a field.

(ii) For comparison, we note that a point $x\in X$ was called universal by Poineau, see \cite[Definition~3.2]{Poineau@angelique}, if the tensor norm on $\calH(x)\wtimes_kl$ is multiplicative for any $l/k$ (originally, universal norms were called peaked, see \cite[Section~5.2]{berbook}). The algebraic analogue of the property that a Banach field is universal over $k$ is geometrical integrality.
\end{rem}

\subsubsection{Defectless case}
One can show that an algebraic extension is universally spectral if and only if it is almost tame, but this will be worked out elsewhere. Here we only check this for defectless extensions.

\begin{lem}\label{univspeclem}
Assume that $K/k$ is a finite defectless extension of analytic fields. Then $K$ is universally spectral over $k$ if and only if $K/k$ is tame.
\end{lem}
\begin{proof}
Throughout the proof, given an analytic $k$-field $l$ we set $L=K\otimes_kl$ and provide it with the tensor seminorm $\lVert \ \rVert_L$. Let $\tilk_\gr=\oplus_{r>0}\tilk_r$ be the $\bfR_{>0}$-graded reduction of $k$, and define $\tilL_\gr$, $\tilK_\gr$ and $\till_\gr$ similarly. Also, let $\tilL'_\gr$ be the $\bfR_{>0}$-graded ring associated with the filtration on $L$ induced by $\lVert \ \rVert_L$; in particular, $\tilL'_\gr=\tilL_\gr$ only when $\lVert \ \rVert_L$ coincides with the spectral seminorm. Note that $\lVert \ \rVert_L$ is not spectral if and only if $\lVert x^n\rVert_L<\lVert x\rVert^n_L$ for some $x$ and large $n$, and this happens if and only if $\tilx$ is a homogeneous nilpotent element of $\tilL'_\gr$. Thus, $\lVert \ \rVert_L$ is spectral if and only if the graded ring $\tilL'_\gr$ is reduced in the sense of graded commutative algebra of \cite[Section~1]{local-properties-II}. (This also reinforces Remark~\ref{unispecrem}(i).)

Since $K/k$ is defectless, $K$ possesses an orthogonal $k$-basis $a_1\..a_n$. Then, this basis is also an orthogonal basis of $L$ over $l$ and hence $\tila_1\..\tila_n$ is a basis of both $\tilK_\gr$ over $\tilk_\gr$ and $\tilL'_\gr$ over $\till_\gr$. In particular, we obtain that $\tilL'_\gr=\tilK_\gr\otimes_{\tilk_\gr}\till_\gr$.

Now, let us prove the lemma. The extension $K/k$ is tame if and only if $\tilK/\tilk$ is separable and $\lvert K^\times\rvert /\lvert k^\times\rvert$ has no $p$-torsion. By \cite[Proposition~2.10 ]{tamediscs} the latter happens if and only if the extension of graded fields $\tilK_\gr/\tilk_\gr$ is separable (in the graded setting). This implies that for any extension of analytic fields $l/k$, the graded $\till_\gr$-algebra $\tilK_\gr\otimes_{\tilk_\gr}\till_\gr$ is reduced, and as we saw above
this happens if and only if the tensor seminorm on $K\otimes_kl$ is spectral.

Conversely, if $K/k$ is wild then we showed in the above paragraph that $\tilK_\gr/\tilk_\gr$ is not separable. For simplicity take $l=\whka$ (in fact, $l=K$ would suffice). Then $\till_\gr$ is an algebraically closed graded field and $\tilK_\gr\otimes_{\tilk_\gr}\till_\gr$ is an inseparable finite $\till_\gr$-algebra. By \cite[1.14.3]{tamediscs}, this implies that $\tilK_\gr\otimes_{\tilk_\gr}\till_\gr$ is not reduced, and hence $\lVert \ \rVert_L$ is not spectral.
\end{proof}

\begin{cor}\label{univspeccor}
Any tame extension is universally spectral.
\end{cor}
\begin{proof}
This follows from the lemma since any tame extension is defectless.
\end{proof}

\subsubsection{Residually tame morphisms}\label{restamesec}
We say that a morphism $g\:S'\to S$ is {\em residually tame} (resp. {\em residually unramified}) {\em at} $s'\in S'$ if the extension of completed residue fields $\calH(s')/\calH(s)$, where $s=g(s')$, is finite and tame (resp. unramified). A morphism is {\em residually tame} or {\em residually unramified} if it is so at all points of the source.

\subsubsection{Compatibility with base changes}
In view of Remark~\ref{domrem}(i), one has to restrict base change morphisms in order to ensure compatibility of K\"ahler seminorms with the base change. Here is a natural way to impose such a restriction.

\begin{theor}\label{compatth}
Let $f\:X\to S$ and $g\:S'\to S$ be two morphisms of analytic $k$-spaces. Assume that for any point $s'\in S'$ with $s=g(s')$ the field $\calH(s')$ is finite and universally spectral over $\calH(s)$ (by Corollary~\ref{univspeccor} this includes the case of a residually tame $g$). Then the pullback of the K\"ahler seminorm on $\Omega_{X/S}$ coincides with the K\"ahler seminorm on $\Omega_{X\times_SS'/S'}$.
\end{theor}
\begin{proof}
Fix a point $x'\in X'=X\times_SS'$ and let $x\in X$, $s'\in S'$ and $s\in S$ be its images. Set also $K=\calH(s)$, $K'=\calH(s')$, $L=\calH(x)$, and $L'_1=\calH(x')$. We should prove that if $\phi\in\Omega_{X/S,x}$ and $\phi'\in\Omega_{X'/S',x'}$ is its pullback then $\lVert \phi\rVert_x=\lVert \phi'\rVert_{x'}$. By Corollary~\ref{inducecor}, the values of the seminorms can be computed at $\whOmega_{L/K}$ and $\whOmega_{L'_1/K'}$, respectively. This reduces the question to proving that the map $\psi\:\whOmega_{L/K}\otimes_L L'_1\to\whOmega_{L'_1/K'}$ is an isometry.

Consider the seminormed ring $L'=L\otimes_KK'$ and note that $\lam\:\whOmega_{L/K}\otimes_L L'\toisom\whOmega_{L'/K'}$ is an isometry by Lemma~\ref{basechangelem}. In addition, $L'$ is reduced since it is spectral, and $L'$ is a finite $L$-algebra since $K'/K$ is finite. Hence $L'\toisom\prod_{i=1}^nL'_i$, where $L'_i$ are extensions of $L$ and $L'_1$ is as defined earlier. We claim that this isomorphism is also an isometry. Indeed, the right-hand side is provided with the sup seminorm hence this follows from \cite[Theorem~3.8.3/7]{bgr}. Thus, $\Omega_{L'/K'}=\prod_{i=1}^n\Omega_{L'_i/K'}$ and hence $\psi'\:\whOmega_{L'/K'}\otimes_{L'} L'_1\to\whOmega_{L'_1/K'}$ is an isometric isomorphism. Therefore, the composition $\psi=\psi'\circ(\lam\otimes_{L'}L'_1)$ is an isometry.
\end{proof}

Let us record the most important particular cases of the theorem.

\begin{cor}\label{compatcor2}
Assume that $X\to S$ is a morphism of $k$-analytic spaces, $s\in S$ is a point and $X_s$ is the fiber over $s$. Then $\lVert \ \rVert_{\Omega,X_s/s}$ equals to the restriction of $\lVert \ \rVert_{\Omega,X/S}$ onto the fiber.
\end{cor}

\begin{cor}\label{compatcor}
Assume that $f\:X\to S$ is a morphism of $k$-analytic spaces and $l/k$ is a finite extension such that $l$ is universally spectral over $k$ (e.g., $l/k$ is tame). Set $X_l=X\otimes l$ and $S_l=Y\otimes l$. Then $\lVert \ \rVert_{\Omega,X_l/S_l}$ is the pullback of $\lVert \ \rVert_{\Omega,X/S}$.
\end{cor}

\begin{rem}
(i) In fact, instead of finiteness of $l/k$ it suffices to assume that the extension is algebraic. The proof is based on passing to a colimit and completing and we leave the details to the interested reader.

(ii) In particular, the K\"ahler seminorm is preserved when we replace $k$ with its completed tame closure. In principle, this reduces the study of K\"ahler seminorms to the case of a tamely closed ground field $k$.
\end{rem}

\subsubsection{Geometric K\"ahler seminorm}
Given a morphism $f\:X\to S$, let $\of\:\oX\to\oS$ denote its ground field extension with respect to $\whka/k$ and let $g$ denote the morphism $\oX\to X$. Provide $g_*\Omega_{\oX/\oS}$ with the pushout quasi-norm $g_*(\lVert \ \rVert_{\Omega,\oX/\oS})$ (see \S\ref{pushsec}) and let $\lVert \ \rVert_{\oOmega,X/S}$ be the quasi-norm it induces on $\Omega_{X/S}$ via the embedding $\Omega_{X/S}\into g_*\Omega_{\oX/\oS}$. By definition, if $\phi\in\Omega_{X/S}(U)$ then $\lVert \phi\rVert_{\oOmega,V}=\lVert g^*\phi\rVert_{\Omega,g^{-1}(V)}$. It follows easily that $\lVert \ \rVert_{\oOmega,X/S}$ is an analytic seminorm and $\lVert \phi\rVert_{\oOmega,x}=\lVert g^*\phi\rVert_{\Omega,x'}$ for any $x'\in \oX$ and $x=g(x')$. We call $\lVert \ \rVert_{\oOmega,X/S}$ the {\em geometric K\"ahler seminorm} on $\Omega_{X/S}$ and use $\oOmega$ in the notation to stress that it is geometric.

Absolutely in the same way one defines the {\em geometric K\"ahler seminorm} $\lVert \ \rVert_{(\oOmega_{X/S}^l)^{\otimes m}}$ on the sheaf of pluriforms $(\Omega_{\oX/\oS}^l)^{\otimes m}$.

\begin{rem}
If $k$ possesses non-trivial wild extensions, K\"ahler seminorms on $k$-analytic spaces can behave rather weird (see \S\ref{rignonperf} and \S\ref{discnonperf}). So, in this case it is often more useful to work with the geometric K\"ahler seminorm.
\end{rem}

\section{PL subspaces}\label{PLsec}
Main results about K\"ahler seminorms are related to the PL structure of analytic spaces, so in the current section we describe what this structure is.

\subsection{Invariants of a point}
We start with recalling basic results about invariants $t$ and $s$ associated to points of $k$-analytic spaces, see \cite[\S9]{berbook}. To stress the valuative-theoretic origin of these invariants we prefer to denote them $F$ and $E$, the transcendental analogues of $e$ and $f$.

\subsubsection{Invariants of extensions of valued fields}
To any extension of valued fields $L/K$ one can associate two cardinals that measure the ``transcendence size" of the extension: the {\em residual transcendence degree} $F_{L/K}=\trdeg_\tilK(\tilL)$ and the {\em rational rank} $E_{L/K}=\dim_{\bfQ}(\lvert L^\times\rvert /\lvert K^\times\rvert \otimes_\bfZ\bfQ)$. Both invariants are additive in towers of extensions $L'/L/K$, i.e. $E_{L'/K}=E_{L'/L}+E_{L/K}$ and $F_{L'/K}=F_{L'/L}+F_{L/K}$.

\subsubsection{Invariants of points}
For a point $x$ of a $k$-analytic space $X$ we define the {\em residual transcendence degree} $F_{X,x}=F_{\calH(x)/k}$ and the {\em rational rank} $E_{X,x}=E_{\calH(x)/k}$. Usually, we will omit the space $X$ in this notation. Additivity of the invariants can now be expressed as follows. Assume that $f\:X\to Y$ is a morphism of $k$-analytic spaces, $x\in X$ is a point with $y=f(x)$, and the $\calH(y)$-analytic space $Z=X\times_Y\calM(\calH(y))$ is the fiber over $y$. Then $F_{X,x}=F_{Y,y}+F_{Z,x}$ and $E_{X,x}=E_{Y,y}+E_{Z,x}$. In particular, if $f$ is finite then $F_{X,x}=F_{Y,y}$ and $E_{X,x}=E_{Y,y}$.

\subsubsection{Classification of points on a curve}\label{curvesec}
Starting with \cite[\S1.4.4]{berbook} and \cite[\S3.6]{berihes}, points of $k$-analytic curves are classified to four types: (1) $\calH(x)\subseteq\whka$, (2) $F_x=1$, (3) $E_x=1$, (4) the rest. In all cases, it is easy to see that $E_x+F_x\le 1$, i.e., $E_x=0$ for type 2 points, $F_x=0$ for type 3 points, and $E_x=F_x=0$ for type 1 and 4 points.

If $x$ is of type 2 or 3 then the following three claims hold: $\wHx$ is finitely generated over $\tilk$, $\lvert \calH(x)^\times\rvert$ is finitely generated over $\lvert k^\times\rvert$, if $k$ is stable then $\calH(x)$ is stable. The first two are simple, and we refer to \cite[Corollary~6.3.6]{temst} for the stability theorem. Note that all three claims can fail for types 1 and 4.

\subsubsection{Monomial points}
Fibering $X$ by curves and using the additivity of $E$ and $F$ and induction on dimension, it is easy to see that any point $x\in X$ satisfies the inequality $E_x+F_x\le\dim_x(X)$. A point $x\in X$ is called {\em monomial} or {\em Abhyankar} if $F_x+E_x=\dim_x(X)$. The set of all monomial points of $X$ will be denoted $X^\mon$.

\begin{rem}\label{monomrem}
(i) Monomial points are adequately controlled by the invariants $E_x$ and $F_x$. This often makes the work with them much easier than with general points. For example, see Corollary~\ref{fibcor} below.

(ii) Analogues of monomial points in the theory of Riemann-Zariski spaces are often called Abhyankar valuations. They are much easier to work with too; for example, local uniformization is known for such points.
\end{rem}

\subsection{PL subspaces}
The set $X^\mon$ is huge and, at first glance, may look a total mess when $\dim(X)>1$. Nevertheless, it possesses a natural structure of an ind-PL space that we are going to recall.

\subsubsection{The model case}\label{modelcasesec}
Recall that points of the affine space $$\bfA_k^n=\bigcup_r\calM(k\{r^{-1}t_1\..r^{-1}t_n\})$$ with coordinates $t_1\..t_n$ can be identified with real semivaluations on $k[t_1\..t_n]$ that extend the valuation of $k$. The semivaluations that do not vanish at $t_1\..t_n$ form the open subspace $\bfG_m^n$. Thus, for each tuple $r\in\bfR_{>0}^n$ the generalized gauss valuation $|\ |_\ur$ (see \S\ref{genGausssec}) defines a point $p_\ur\in\bfG_m^n$, and the correspondence $\ur\mapsto p_\ur$ provides a topological embedding $\alp\:\bfR_{>0}^n\into \bfG_m^n$. Let $S$ be the image of $\alp$.

\begin{rem}\label{modelrem}
(i) One often calls $S$ the {\em skeleton} of $\bfG_m^n$; this terminology is justified by (ii) and (iii) below. Sometimes one refers to the points of $S$ as {\em $\ut$-monomial valuations} because they are determined by their restriction to the monoid $\prod_{i=1}^nt_i^\bfN$.

(ii) Any semivaluation $x\in \bfG_m^n$ is dominated by the $\ut$-monomial valuation $p_{\lvert \ut(x)\rvert }$. The map $x\mapsto p_{\lvert \ut(x)\rvert }$ is a retraction $r\:\bfG_m^n\onto S$. Moreover, Berkovich constructs in \cite[\S6]{berbook} a deformational retraction of $\bfG_m^n$ onto $S$ whose level at 1 is the above map.

(iii) All $t$-monomial valuations are distinguished by invertible functions (in fact, by monomials $t^a$), hence they are incomparable with respect to the domination. Thus, $S$ is the set of all points of $\bfG_m^n$ that are maximal with respect to the domination. In particular, it is independent of the coordinates.
\end{rem}

\subsubsection{Monomial charts}
By a monomial chart of an analytic space $X$ we mean a morphism $f\:U\to\bfG_m^n$ such that $U$ is an analytic domain in $X$ and $f$ has zero-dimensional fibers. Clearly, $f$ is determined by invertible functions $t_1\..t_n\in\calO^\times_X(U)$. A point $x\in U$ of the chart is called {\em $\ut$-monomial} or {\em $f$-monomial} if the restriction of $\lvert \ \rvert_x$ to $k[t_1\..t_n]$ is monomial. In this case, we also say that $t_1\..t_n$ is a {\em family of monomial parameters} at $x$. Note that $\calH(f(x))=\wh{k(\ut)}$ is provided with a generalized Gauss valuation and $\calH(x)$ is its finite extension.

In addition, we say that the chart is residually tame or unramified (at a point $x\in U$) if the morphism $f$ is so, see \S\ref{restamesec}. In this case we also say that the family of monomial parameters $t_1\..t_n$ is {\em residually tame} or {\em unramified} (at $x$). The following result shows that monomial charts adequately describe the whole $X^\mon$.

\begin{lem}\label{monchartlem}
A point $x\in X$ is monomial if and only if there exists a monomial chart $f\:U\to\bfG_m^n$ such that $x$ is $f$-monomial.
\end{lem}
\begin{proof}
Set $L=\calH(x)$. If there exists a monomial chart $f$ such that $f(x)$ is a monomial point then the induced map $f\:U\to\bfG_m^n$ has zero-dimensional fibers and hence $L$ is finite over $K=\calH(y)$, where $y=f(x)$. Since $E_y+F_y=n$, we obtain that $E_x+F_x=n$, i.e. $x$ is monomial.

Conversely, assume that the sum of $E=E_x$ and $F=F_x$ equals to $n=\dim_x(X)$. Choose $t_1\..t_F\in\kappa_G(x)$ such that $\tilt_1\..\tilt_F$ is a transcendence basis of $\tilL/\tilk$, and choose $\{t_{F+1}\..t_n\}$ such that its image in $(\lvert L^\times\rvert /\lvert k^\times\rvert )\otimes\bfQ$ is a basis. Then the valuation on $K=\wh{k(t_1\..t_n)}$ is a generalized Gauss valuation. Take an analytic domain $U\subseteq X$ containing $x$ such that $\dim(U)=n$ and $t$ induces a morphism $f\:U\to\bfG_m^n$. Then $y=f(x)$ is a monomial point of $\bfG^m_n$ (even a point of its skeleton). Consider the fiber $U_y=f^{-1}(y)$. By \cite[Corollary~8.4.3]{flatness} $\dim_y(U_y)=0$, hence using \cite[Theorem~4.9]{Ducros} we can shrink $U$ around $y$ so that $f$ has zero-dimensional fibers.
\end{proof}

\begin{cor}\label{fibcor}
Assume that $x\in X$ is a monomial point, then
\begin{itemize}
\item[(0)] The ring $\calO_{X_G,x}$ is Artin. In particular, if $X$ is reduced then $m_{G,x}=0$ and $\calO_{X_G,x}=\kappa_G(x)$.
\item[(1)] The extension $\wHx/\tilk$ is finitely generated.
\item[(2)] The group $\lvert \calH(x)^\times\rvert /\lvert k^\times\rvert$ is finitely generated.
\item[(3)] If $k$ is stable then $\calH(x)$ is stable.
\end{itemize}
\end{cor}
\begin{proof}
By Lemma~\ref{monchartlem} there exists a chart $f\:U\to\bfG_m^n$ that takes $x$ to a point $y$ corresponding to a generalized Gauss valuation. Note that $\calO_{X_G,x}$ is finite over $\calO_{Y_G,y}$. All four properties are satisfied by $y$, hence they also hold for $x$.
\end{proof}

We can now strengthen Lemma~\ref{monchartlem} in the case of an algebraically closed $k$.

\begin{cor}\label{monchartcor}
Assume $X$ is a $k$-analytic space, $k$ algebraically closed and $x\in X$ is a monomial point. Then there exists a monomial chart $f$ such that $x$ is $f$-monomial and $f$ is residually unramified at $x$.
\end{cor}
\begin{proof}
The proof repeats that of Lemma~\ref{monchartlem}, but we will choose $t_i$ more carefully. Set $L=\calH(x)$, $E=E_{L/k}$, $F=F_{L/k}$ and $n=E+F$. Since $k=k^a$ the field $\tilk$ is algebraically closed and the group $|k^\times|$ is divisible. In particular, the extension $\tilL/\tilk$ is separable and the group $\lvert L^\times\rvert /\lvert k^\times$ is torsion free. Since $\tilL/\tilk$ is finitely generated by Corollary~\ref{fibcor}(1), it possesses a separable transcendence basis. Choose $t_1\..t_F\in\kappa_G(x)$ so that $\tilt_1\..\tilt_F$ is such a basis. Since the group $\lvert L^\times\rvert /\lvert k^\times\rvert$ is finitely generated by Corollary~\ref{fibcor}(2), it is isomorphic to $\bfZ^E$. Choose $t_{F+1}\..t_n\in\kappa_G(x)$ so that their images form a basis of $\lvert L^\times\rvert /\lvert k^\times$.

The elements $t_1\..t_n$ define a morphism $f\:U\to\bfG_m^n$ for a small enough analytic domain $U$ containing $x$, and we claim that $f$ is as required. Indeed, it suffices to check that $L$ is unramified over $K=\calH(f(x))=\wh{k(t_1\..t_n)}$. By our choice $\lvert L^\times\rvert =\lvert K^\times\rvert$ and $\tilL/\tilK$ is separable. Since $K$ is stable by Corollary~\ref{fibcor}(3), $L/K$ is unramified.
\end{proof}

\subsubsection{Skeletons of monomial charts}
The set of all $f$-monomial points of $U$ will be denoted $S(f)$ and called the {\em skeleton} of the chart. Note that $S(f)$ is nothing else but the preimage of the skeleton of $\bfG_m^n$ under $f$. By Lemma~\ref{monchartlem}, $X^\mon=\cup_f S(f)$, where $f$ runs over all monomial charts of $X$. We warn the reader that, in general, $S(f)$ does not have to be a retract of $U$.

\subsubsection{$R_S$-PL structures}\label{plspacesec}
We will usually abbreviate ``piecewise linear" as PL. We refer to \cite[\S1]{bercontr2} for the definition of an $R_S$-PL space $Q$ for a ring $R\subseteq\bfR$ and its exponential module $S\subseteq\bfR_{>0}$. Here we only recall that $Q$ has an atlas $\{P_i\}$ of $R_S$-PL polytopes, i.e. polytopes in $\bfR_{>0}^n$ given by finitely many inequalities $st_1^{e_1}\dots t_n^{e_n}\le 1$ with $s\in S$, $e_i\in R$ and provided with the family of $R_S$-PL functions.

\subsubsection{Rational PL-subspaces}\label{plsec}
Absolute values of the coordinates of $\bfG_m^n$ induce coordinates $t_i\:S\to\bfR_{>0}$ on its skeleton $S$. The latter are unique up to the action of $\GL(n,\bfZ)\ltimes\lvert k^\times\rvert^n$ combined from the action of $\GL(n,\bfZ)$ on $\prod_{i=1}^nt_i^\bfZ$ and rescaling the coordinates by elements of $\lvert k^\times\rvert$. Therefore, $S$ acquires a canonical $\bfZ_{\lvert k^\times\rvert }$-PL structure.

The following facts were proved by Ducros: (1) for any monomial chart $f\:U\to\bfG^n_m$ the skeleton $S(f)$ possesses a unique $\bfQ_H$-PL structure such that the map $S(f)\to S$ is $\bfQ_H$-PL, see \cite[Th. 3.1]{Ducpl}, (2) for any two monomial charts $f$ and $g$ the intersection $S(f)\cap S(g)$ is $\bfQ_H$-PL in both $S(f)$ and $S(g)$, see \cite[Th. 5.1]{Ducpl2}, and so $S(f)\cup S(g)$ acquires a natural $\bfQ_H$-PL structure and the whole $X^\mon$ acquires a natural structure of an ind-$\bfQ_H$-PL space. By a {\em rational $\bfQ_H$-PL subspace} of $X$ we mean a subset of the form $\cup_{i=1}^nS({f_i})$ with its induced $\bfQ_H$-PL structure.

\subsubsection{Integral structure}
One may wonder if the $\bfQ_H$-PL structure on a rational PL subspace $P$ of $X$ can be refined to an integral one. Ducros and Thuillier showed in \cite{skeletons} that this is indeed the case: there is a $\bfZ_H$-PL structure on $P$ such that a function $f\:P\to\bfR_{>0}$ is $\bfZ_H$-PL if and only if $G$-locally it is of the form $r\lvert f\rvert$, where $f$ is an analytic function on $X$ and $r\in H$ (see \cite[3.7]{skeletons} and note that $H^\bfQ=H$). Obviously, such a structure is unique but consistency of the definition requires an argument. In addition, they show that the associated $\bfQ_H$-PL space, obtained by adjoining integral roots of all $\bfZ_H$-PL functions on $P$, coincides with the original $\bfQ_H$-PL space. In the sequel, we provide $P$ with this $\bfZ_H$-PL structure and call it a {\em $\bfZ_H$-PL subspace} of $X$.

\subsection{Semistable formal models}\label{semistablesec}

\subsubsection{Semistable formal schemes}
We say that a formal $\kcirc$-scheme is {\em strictly semi\-stable} if locally it admits an \'etale morphism to a formal scheme of the form
$\gtZ_{n,a}=\Spf(\kcirc\{T_0\.. T_n\}/(T_0\dots T_n-a))$ with $0\neq a\in\kcirc$. A formal $\kcirc$-scheme is called {\em semistable} if it is \'etale-locally strictly semistable .

\begin{rem}
(i) Sometimes one does not require that $a\neq 0$, thereby obtaining a wider class of semistable formal schemes. If $\gtX$ is semistable in this sense then $\gtX_\eta$ is quasi-smooth if and only if one can find charts with $a\neq 0$. Since we will only be interested in formal models of quasi-smooth spaces, we use the definition that includes the condition $a\neq 0$.

(ii) If the valuation is discrete and $a$ is a uniformizer, one often considers schemes $\gtZ=\Spf(\kcirc\{T_0\.. T_n\}/(T_0^{l_0}\dots T_n^{l_n}-a))$. They are regular with snc closed fiber. If $l_i>1$ then the closed fiber is not reduced and hence $\gtZ$ is not semistable. Note also that if $(l_1\.. l_n)\in\tilk^\times$ then $\gtZ$ is log smooth, see Remark~\ref{logsmoothrem} below.
\end{rem}

\subsubsection{Skeletons associated to semistable formal models}\label{modelsec}
To any semistable formal $\kcirc$-scheme $\gtX$ with generic fiber $X=\gtX_\eta$ Berkovich associated in \cite[Section~5]{bercontr} the skeleton $S(\gtX)\subset X$ and constructed a deformational retraction $X\onto S(\gtX)$. Moreover, these constructions are compatible with any \'etale morphism $\phi\:\gtY\to\gtX$, i.e. $S(\gtY)=\phi_\eta^{-1}(S(\gtX))$ and the retraction is compatible with $\phi_\eta$, see \cite[Theorem~5.2(vii)]{bercontr}.

Since any semistable formal $\kcirc$-scheme is connected with semistable formal schemes $\gtZ_{n,a}$ by a zigzag of two \'etale morphisms, description of the skeleton and the retraction reduces to the model case, and the latter is induced from $\bfG^n_m$. Namely, let $\bfG^n_m$ be the $n$-dimensional torus with coordinates $T_1\..T_n$. The affinoid subdomain given by $\lvert T_1\dots T_n\rvert \ge\lvert a\rvert$ and $\lvert T_i\rvert \le 1$ for $1\le i\le n$ equals to $Z=\calM(k\{T_0\..T_n\}/(T_0\dots T_n-l)$ and hence can be identified with the generic fiber of $\gtZ=\gtZ_{n,a}$. Then $T_1\..T_n$ give rise to the monomial chart $f\:Z\into\bfG^n_m$ with skeleton $S(f)=S(\bfG^m_n)\cap Z$, and $S(\gtZ)$ coincides with $S(f)$. The retraction of $Z$ onto $S(\gtZ)$ is the restriction of the retraction of $\bfG^m_n$ onto $S(\bfG^m_n)$.

\begin{rem}\label{logsmoothrem}
(i) Slightly more generally, Berkovich makes the above constructions for {\em polystable} models, i.e. those models that are \'etale-locally isomorphic to products of semistable ones. In fact, this can be extended further to log smooth formal models with trivial generic log structure -- these are formal schemes that \'etale-locally admit smooth morphisms to formal schemes of the form $\Spf(\kcirc\{P/\Pi\})$, where $\Pi$ and $P$ are sharp fs monoids, $\alp\:\Pi\into\kcirc\setminus\{0\}$ is an embedding such that the composition $\Pi\to\lvert \kcirc\setminus\{0\}\rvert$ is injective, $\phi\:\Pi\to P$ is an injective homomorphism with no $p$-torsion in the cokernel and such that $\Pi^\gp P=P^\gp$, and $\kcirc\{P/\Pi\}$ is the quotient of $\kcirc\{P\}$ by the ideal generated by elements $\alp(\pi)-\phi(\pi)$ for $\pi\in\Pi$. The details will be worked out elsewhere.

(ii) The main motivation for considering log smooth formal models is that it is believed (at least by the author) that any quasi-smooth compact strictly analytic space possesses a log smooth formal model. This is absolutely open when $\cha(\tilk)>0$, but one may hope to prove this by current techniques when $\cha(\tilk)=0$. Currently, the latter is only known for a discretely valued $k$: desingularization of excellent formal schemes implies that there even exists a semistable formal model. If the $\bfQ$-rank of $\lvert k^\times\rvert$ is larger than one, then it is easy to give examples when semistable models do not exist, e.g. $\Spf(k\{at^{-1},t\})\times\Spf(k\{bt^{-1},t\})$, where $\lvert b\rvert \notin\lvert a\rvert^\bfQ$ (see also \cite[Remark~1.1.1(ii)]{altered}). The situation with polystable models is unclear: it is still an open combinatorial question (in dimension at least 4) whether any log smooth formal model possesses a polystable refinement. A tightly related question was raised by Abramovich and Karu in \cite[Section~8]{AK}: log smooth (resp. polystable) models are analogues of weakly semistable (resp. semistable) morphisms in the sense of \cite[Section~0]{AK}.
\end{rem}

\section{Metrization of pluricanonical forms}\label{lastsec}
Throughout Section~\ref{lastsec}, $X$ is assumed to be quasi-smooth of pure dimension $n$. In particular, $\Omega_X=\Omega^1_{X/k}$ is a locally free sheaf of rank $n$ and the pluricanonical sheaves $\omega^{\otimes m}_X=(\bigwedge^n\Omega_{X})^{\otimes m}$ are invertible.

\subsection{Monomiality of K\"ahler seminorms}

\subsubsection{Stalks at monomial points}
Recall that if $k=k^a$ then any monomial point possesses a family of residually tame monomial parameters by Corollary~\ref{monchartcor}. This allows to describe $\lVert \ \rVert_\omega$ as follows.

\begin{theor}\label{diffieldprop}
Assume that $k$ is algebraically closed. Let $X$ be a quasi-smooth $k$-analytic space of dimension $n$ with a point $x\in X$, and let $t_1\..t_n$ be invertible elements of $\calO_{X_G,x}$. Then $\lVert \frac{dt_1}{t_1}\wedge\dots\wedge\frac{dt_n}{t_n}\rVert_{\omega,x}\le 1$ and the following conditions are equivalent:

(i) $\lVert \frac{dt_1}{t_1}\wedge\dots\wedge\frac{dt_n}{t_n}\rVert_{\omega,x}=1$.

(ii) $\frac{dt_1}{t_1}\..\frac{dt_n}{t_n}$ form an orthonormal basis of $\hatOmega_{\calH(x)/k}$.

(iii) $t$ is a family of residually tame monomial parameters at $x$.
\end{theor}
\begin{proof}
Since $\lVert dt_i\rVert_{\Omega,x}\le\lvert t_i\rvert_x$, it follows that $\lVert \frac{dt_1}{t_1}\wedge\dots\wedge\frac{dt_n}{t_n}\rVert_{\omega,x}\le 1$. The equivalence (i)$\Longleftrightarrow$(ii) is proved precisely as the equivalence (iii)$\Longleftrightarrow$(iv) in the proof of Theorem~\ref{tmonth}. We will complete the proof by showing that (ii)$\Longleftrightarrow$(iii).

Set $l=k(t)$ and $L=\hatl$. First, assume that $t_i$ form a family of residually tame monomial parameters. Then $\frac{dt_1}{t_1}\..\frac{dt_n}{t_n}$ is an orthonormal basis of $\hatOmega_{l/k}$ by Theorem~\ref{tmonth}. It remains to use that $\hatOmega_{l/k}=\hatOmega_{L/k}$ by Corollary~\ref{densecor2}, and $\hatOmega_{L/k}=\hatOmega_{\calH(x)/k}$ because $\calH(x)/L$ is tame.

Conversely, assume that $\frac{dt_1}{t_1}\..\frac{dt_n}{t_n}$ is an orthonormal basis of $\hatOmega_{\calH(x)/k}$. Since the isomorphism $\psi_{\calH(x)/L/k}\:\hatOmega_{L/k}\otimes_L\calH(x)\to\hatOmega_{\calH(x)/k}$ is non-expansive and $\lVert \frac{dt_i}{t_i}\rVert_{\Omega,{L/k}}\le 1$, we obtain that $\psi_{\calH(x)/L/k}$ is an isometry and $\frac{dt_1}{t_1}\..\frac{dt_n}{t_n}$ is an orthonormal basis of $\hatOmega_{L/k}$. The module $\Omega^\rmlog_{\calH(x)^\circ/\kcirc}$ is almost torsion free by Theorem~\ref{torth}, hence $\calH(x)/L$ is almost tame by Theorem~\ref{kahlerextth}.

It remains to show that $t_1\..t_n$ is a family of monomial parameters because then $L$ is stable by Corollary~\ref{fibcor}(3) and hence $\calH(x)/L$ is tame by Theorem~\ref{altameextth}(ii). Note that $t_1\..t_n$ are algebraically independent over $k$ because $dt_1\..dt_n$ are linearly independent in $\hatOmega_{\calH(x)/k}$, hence $t_1\..t_n$ is a separable transcendence basis of $l/k$. Since $\Omega_{l/k}$ is finite-dimensional, the completion $\Omega_{l/k}\to\hatOmega_{l/k}$ is surjective, and comparing the dimensions we see that it is an isomorphism. Thus, $\Omega_{l/k}\otimes_lL\to\hatOmega_{L/k}$ is an isometric isomorphism by Corollary~\ref{densecor2}. This implies that $\frac{dt_1}{t_1}\..\frac{dt_n}{t_n}$ is an orthonormal basis of $\Omega_{l/k}$, and hence the valuation on $l$ is $t$-monomial by Theorem~\ref{tmonth}.
\end{proof}

\begin{cor}\label{monomnormcor2}
Assume that $k$ is algebraically closed, $X$ is quasi-smooth, elements $t_1\..t_n\in\calO_{X_G,x}$ form a family of residually tame monomial parameters at a point $x$, and $\calF=(\Omega_X^l)^{\otimes m}$. Then
$$B=\left\{\left(\frac{dt_{i_1}}{t_{i_1}}\wedge\dots\wedge\frac{dt_{i_l}}{t_{i_l}}\right)\otimes\dots
\otimes\left(\frac{dt_{j_1}}{t_{j_1}}\wedge\dots\wedge\frac{dt_{j_l}}{t_{j_l}}\right)\right\}$$
is an orthonormal basis of $\calF_x$. In particular, any pluriform $\phi\in\Gamma((\Omega_X^l)^{\otimes m})$ can be represented as $\phi=\sum_{e\in B}\phi_e e$ locally at $x$, and the following equality holds $$\lVert \phi\rVert_{(\Omega_X^l)^{\otimes m},x}=\max_{e\in B}\lvert \phi_e\rvert_x.$$
\end{cor}
\begin{proof}
By Theorem~\ref{diffieldprop}, $\frac{dt_1}{t_1}\..\frac{dt_n}{t_n}$ is an orthonormal basis of $\Omega_{X,x}$. This reduces the claim to simple multilinear algebra.
\end{proof}

\begin{cor}\label{monomnormcor}
Keep the assumptions of Corollary~\ref{monomnormcor2} and assume that $\calF=\omega_X^{\otimes m}$. Then $e=(\frac{dt_1}{t_1}\wedge\dots\wedge\frac{dt_n}{t_n})^{\otimes m}$ is a basis of $\calF_x$, and if $\phi=fe$ is the representation of a pluricanonical form $\phi$ at $x$ then $\lVert \phi\rVert_{\omega^{\otimes m},x}=\lvert f\rvert_x$.
\end{cor}

\subsubsection{Piecewise monomiality}
Now, we can prove that norms of pluriforms induce $\bfZ_H$-PL functions on $\bfZ_H$-PL subspaces of $X$ (see \S\ref{plsec}). In particular, when restricted to $\bfZ_H$-PL subspaces of $X$, spectral seminorm and K\"ahler seminorms demonstrate similar behavior, although their global behavior is very different.

\begin{theor}\label{rplth}
Assume that $X$ is a reduced $k$-analytic space and $\phi\in\Gamma((\Omega_X^l)^{\otimes m})$ is a pluriform on $X$, and consider the function $\lVert \phi\rVert \:X\to\bfR_{\ge 0}$ that sends $x$ to the value $\lVert \phi\rVert_{(\oOmega_X^l)^{\otimes m},x}$ of the geometric K\"ahler seminorm. Then for any $\bfZ_H$-PL subspace $P\subset X$ the restriction of $\lVert \phi\rVert$ onto $P$ is a $\bfZ_H$-PL function.
\end{theor}
\begin{proof}
Since $X$ is reduced, any monomial point satisfies $m_x=0$ and hence $X$ is quasi-smooth at $x$. Therefore, replacing $X$ by a neighborhood of $P$ we can assume that $X$ is quasi-smooth.

First, we consider the case when $k=k^a$. By Theorem~\ref{tamechart} that will be proved in the end of the paper, we can cover $S$ by finitely many skeletons of residually tame (even unramified) monomial charts, hence it suffices to consider the case when $P$ itself is the skeleton of a residually tame monomial chart $f\:U\to\bfG^n_m$ given by $t_1\..t_n\in\Gamma(\calO_U^\times)$. By Corollary~\ref{monomnormcor2}, locally at a point $x\in P$ we can represent $\lVert \phi\rVert$ as the maximum of $\bfZ_H$-PL functions $\lvert \phi_i\rvert$. Hence $\lVert \phi\rVert$ is $\bfZ_H$-PL too.

Assume, now, that $k$ is arbitrary. Set $\oX=X\wtimes_kk^a$ and let $\ophi\in\Gamma((\Omega_\oX^l)^{\otimes m})$ be the pullback of $\phi$. The preimage $\oP\subset\oX$ of $P$ is a $\bfZ_H$-PL subspace: just take the charts of $P$ and extend the ground field to $\whka$. Also, it follows from the description with charts that the map $\oP\to P$ is $\bfZ_H$-PL. Since the pullback of $\lVert \phi\rVert_{(\oOmega_X^l)^{\otimes m},x}$ to $\oP$ is $\lVert \ophi\rVert_{(\Omega_\oX^l)^{\otimes m},x}$ and the latter function is $\bfZ_H$-PL by the case of an algebraically closed ground field, $\lVert \phi\rVert_{(\oOmega_X^l)^{\otimes m},x}$ is a $\bfZ_H$-PL function as well.
\end{proof}

\begin{rem}
Most probably, the assumption that $X$ is reduced can also be removed. Also, it seems probable that the theorem holds for the K\"ahler seminorm too.
\end{rem}

\subsection{Maximality locus of pluricanonical forms}\label{maxsec}
Given a pluricanonical form $\phi$, let $M_\phi$ be the maximality locus of the K\"ahler seminorm of $\phi$ and let $\oM_\phi$ be the maximality locus of the geometric K\"ahler seminorm of $\phi$. Our next aim is to study $\oM_\phi$; recall that it is closed by Theorem~\ref{analyticomega} and Lemma~\ref{rankonelem}. The main results of this section, including Theorems \ref{semistableprop} and \ref{plmaxth}, do not hold for the maximality locus $M_\phi$, at least when the residue field is not perfect; the counterexamples being as in Section~\ref{discnonperf} and Remark~\ref{mixedrem}.

\subsubsection{The torus case}
We start with studying the standard pluricanonical form on a torus.

\begin{lem}\label{toruscor}
Let $X$ be the $k$-analytic torus $\bfG_m^{n,\an}$ with coordinates $t_1\..t_n$. Consider the pluricanonical form $\phi=(\frac{dt_1}{t_1}\wedge\dots\wedge\frac{dt_n}{t_n})^{\otimes m}$. Then $\lVert \phi\rVert_{\oomega^{\otimes m},x}\le 1$ for any $x\in X$ and the equality takes place if and only if $x$ is a generalized Gauss point (see \S\ref{modelcasesec}). In particular, the maximum locus of $\lVert \phi\rVert_{\oomega^{\otimes m}}$ is the skeleton $\bfR_{>0}^n$ of $X$.
\end{lem}
\begin{proof}
This immediately follows from Theorem~\ref{diffieldprop}.
\end{proof}

\subsubsection{The semistable case}
Recall that the skeleton $S(\gtX)\subset X$ associated with a strictly semistable formal model has a natural structure of a simplicial complex.

\begin{theor}\label{semistableprop}
Assume that $X$ is a quasi-smooth compact strictly $k$-analytic space, $\phi\in\Gamma(\omega^{\otimes m}_X)$ is a pluricanonical form on $X$, $\gtX$ is a strictly semistable formal model of $X$, and $S(\gtX)\subset X$ is the skeleton associated with $\gtX$. Then the maximality locus $\oM_\phi$ is a union of faces of $S(\gtX)$.
\end{theor}
\begin{proof}
Set $\lVert \ \rVert =\lVert \ \rVert_{\oomega^{\otimes m}}$ for shortness. First, let us prove that $\oM_\phi\subseteq S(\gtX)$. Working locally on $\gtX$ we can assume that there exists an \'etale morphism $$\gtg\:\gtX\to\gtY=\Spf(\kcirc\{t_0\..t_n\}/(t_0\dots t_n-a)),$$ where $0\neq a \in\kcirc$. Since $Y=\gtY_\eta$ is a domain in $\bfG_m^n$ with coordinates $t_1\..t_n$ (see \S\ref{modelsec}), $e=(\frac{dt_1}{t_1}\wedge\dots\wedge\frac{dt_{n}}{t_{n}})^{\otimes m}$ is a nowhere vanishing pluricanonical form on $Y$ and therefore $\phi=he$ for a function $h\in\Gamma(\calO_X)$. By \cite[Theorem~5.3.2(ii)]{bercontr}, the retraction $r\:X\to S(\gtX)$ is compatible with domination, in particular, $\lvert h\rvert_x\le\lvert h\rvert_{r(x)}$. Since $\lVert \phi\rVert_{x}=\lvert h\rvert_x\lVert e\rVert_{x}$, it remains to prove that $\lVert e\rVert_{x}\le\lVert e\rVert_{r(x)}$ and the equality holds if and only if $x=r(x)$, i.e. $x\in S(\gtX)$.

When working with $e$, we can also view it as a form on $Y$. Let $g\:X\to Y$ be the generic fiber of $\gtg$. Since $\gtg$ is \'etale, if $x\in X$ and $y=g(x)$ then the extension $\calH(x)/\calH(y)$ is unramified by \cite[Lemma~1.6]{bercontr}. So, Theorem~\ref{compatth} and Lemma~\ref{univspeclem} imply that $\lVert e\rVert_{x}=\lVert e\rVert_{y}$. By Corollary~\ref{toruscor}, $\lVert e\rVert_{y}\le 1$ for any point $y\in Y$ and the equality holds precisely for the points of $S(\gtY)$. So, $\lVert e\rVert_{x}\le 1$ for any $x\in X$ and the equality holds precisely for the points of $g^{-1}(S(\gtY))=S(\gtX)$.

It remains to show that if $\Delta$ is a face of $S(\gtX)$ then either $\Delta\subset\oM_\phi$ or the interior $\Delta^\circ$ is disjoint from $\oM_\phi$. This claim is local on $\gtX$ so we can assume that $\gtX=\Spf(\calAcirc)$ is affine and there is an \'etale morphism $\gtg\:\gtX\to\gtY$ as above. In particular, we can assume that $\phi=he$, as above, and so $\lVert \phi\rVert_{x}=\lvert h\rvert_x$ at any point $x\in S(\gtX)$. Note that $\gtX$ is pluri-nodal in the sense of \cite[Section 1]{bercontr}, hence $\lvert h\rvert_\calA\in\lvert k^\times\rvert$ by \cite[Proposition~1.4]{bercontr}. Thus, multiplying $\phi$ by an element of $k^\times$ we can achieve that $\lvert h\rvert_\calA=1$, and then $h\in\calAcirc$ is a function on $\gtX$. Note that $\Delta^\circ$ is the preimage in $S(\gtX)$ of a point $\gtx\in\gtX$. If $\tilh(\gtx)=0$ then $\lvert h\rvert_x<1$ at any point in the preimage of $\gtx$ in $X$, and hence $\Delta^\circ\cap\oM_\phi=\emptyset$. Otherwise, $\lvert h\rvert_x=1$ for any $x$ as above, hence $\Delta^\circ\subseteq\oM_\phi$ and by the closedness of $\oM_\phi$ we obtain that $\Delta\subseteq\oM_\phi$.
\end{proof}

\begin{rem}
(i) Theorem~\ref{semistableprop} is an analogue of \cite[Th. 4.5.5]{Mustata-Nicaise}, though it applies to a wider context (e.g., $X$ is only assumed to be quasi-smooth). Note, however, that these results consider different seminorms when $\cha(\tilk)>0$ (see Section~\ref{comparsec} below).

(iii) The lemma can be extended to general semistable models at cost of considering generalized simplicial complexes. Moreover, it should extend to arbitrary log smooth formal models, once the foundations are set (see Remark~\ref{logsmoothrem}).
\end{rem}

\subsubsection{Residually tame coverings}
Recall that residual tameness was defined in \S\ref{restamesec}.

\begin{lem}\label{semistablecor}
Assume that $X$ is a quasi-smooth compact strictly $k$-analytic space admitting a residually tame quasi-\'etale covering $Y\to X$ such that $Y$ possesses a strictly semistable formal model. Then for any pluricanonical form $\phi\in\Gamma(\omega^{\otimes m}_X)$ on $X$ the geometric maximality locus $\oM_\phi$ is a compact $\bfZ_H$-PL subspace of $X$.
\end{lem}
\begin{proof}
Let $\psi\in\Gamma(\omega_Y)$ be the pullback of $\phi$. By Theorem~\ref{compatth}, the norm functions $\lVert \phi\rVert_{\omega^{\otimes m}}$ and $\lVert \psi\rVert_{\omega^{\otimes m}}$ are compatible, and hence $\oM_\phi=f(\oM_\psi)$. It remains to recall that $\oM_\psi$ is a compact $\bfZ_H$-PL subspace by Theorem~\ref{semistableprop}, and hence its image under $f$ is a $\bfZ_H$-PL subspace by \cite[Proposition~2.1]{skeletons}.
\end{proof}

\subsubsection{Residue characteristic zero}
In order to use the previous lemma, one should construct an appropriate covering $Y\to X$. The author conjectures that any quasi-smooth strictly analytic space possesses a quasi-net of analytic subdomains that admit a semistable model (this is an analogue of the local uniformization conjecture in the desingularization theory). However, this seems to be out of reach when $\cha(\tilk)>0$. Even the case of $\cha(\tilk)=0$, which should be relatively simple, is missing in the literature. We can avoid dealing with it here in view of the following generalization of a theorem of U. Hartl to ground fields with non-discrete valuations, see \cite[Theorem~3.4.1]{altered}: there exist a finite extension $l/k$ and a quasi-\'etale surjective morphism $Y\to X_l=X\otimes_kl$ such that $Y$ possesses a strictly semistable formal model. The theorem does not provide any control on residual tameness of $f$, but it is automatic whenever $\cha(\tilk)=0$.

\begin{theor}\label{plmaxth}
Assume that $\cha(\tilk)=0$ and $X$ is a quasi-smooth compact strictly $k$-analytic space. Then for any pluricanonical form $\phi\in\Gamma(\omega^{\otimes m}_X)$ on $X$, the maximality locus $M_\phi$ is a compact $\bfZ_H$-PL subspace of $X$.
\end{theor}
\begin{proof}
In this case, there is no difference between $\lVert \ \rVert_{\omega^{\otimes m}}$ and $\lVert \ \rVert_{\oomega^{\otimes m}}$. Take $Y$ and $l$ as in \cite[Theorem~3.4.1]{altered}, then the composition $Y\to X\otimes_k l\to X$ is a quasi-\'etale covering, which is automatically residually tame. It remains to use Lemma~\ref{semistablecor}.
\end{proof}

We us say that a point $x\in X$ is {\em divisorial} if $x$ is monomial and $E_{\calH(x)/k}=0$. Thus, $x$ is divisorial if and only if $\trdeg_{\wHx/\tilk}=\dim_x(X)$. We will not need this, but it is easy to see that if $X$ is strictly analytic then $x$ is divisorial if and only if there exists a formal model $\gtX$ such that $x$ is the preimage of a generic point of $\gtX_s$ under the reduction map.

\begin{cor}\label{plmaxcor}
Keep the assumptions of Theorem~\ref{plmaxth} and let $X^\div$ be the set of divisorial points of $X$. Then $\lVert \phi\rVert _{\omega^{\otimes m}}=\max_{x\in X^\div}\lVert \phi\rVert _{\omega^{\otimes m},x}$.
\end{cor}
\begin{proof}
By Theorem~\ref{plmaxth}, the maximality locus $M_\phi$ is a compact $\bfZ_H$-PL subspace. By Theorem~\ref{rplth}, the restriction of $\lVert \phi\rVert $ on $M_\phi$ is a $\bfZ_H$-PL function, hence it achieves maximum at a $\bfZ_H$-rational point $x$. Any such $x$ is a divisorial point of $X$.
\end{proof}

\subsection{Comparison with the weight norm of Musta\c{t}\u{a}-Nicaise}\label{comparsec}
We conclude Section~\ref{lastsec} by comparing $\lVert \ \rVert_{\omega}$ with the weight norm \`a la Musta\c{t}\u{a} and Nicaise, see \cite{Mustata-Nicaise}. Unless said to the contrary, $k$ is assumed to be discretely valued.

\subsubsection{Weight seminorm}
Assume that $K/k$ is a separable finitely generated extension of real-valued fields of transcendence degree $n$ such that $\trdeg(\tilK/\tilk)=n$. Note that $K$ is  discretely valued and such extensions correspond to divisorial valuations. Let us recall how a norm on $\omega_{K/k}=\Omega_{K/k}^n$ is defined in \cite{Mustata-Nicaise}. We will call it the {\em weight norm} and denote $\lVert \ \rVert_\rmwt$. Fix $t_1\..t_n\in\Kcirc$ such that $\tilt_1\..\tilt_n$ is a transcendence basis of $\tilK/\tilk$. We claim that replacing $t_i$ with elements $t'_i$ such that $\lvert t_i-t'_i\rvert <1$ one can in addition achieve that $t_1\..t_n$ is a separable transcendence basis of $K$. To prove this we will use the observation that the latter happens if and only if $dt_1\..dt_n$ is a basis of $\Omega_{K/k}$. Choose a separable transcendence basis $x_1\..x_n$, in particular, $dx_i$ form a basis. Then the elements $d(t_i+ax_i)=dt_i+adx_i$ form a basis for all but finitely many values of $a\in k$. In particular, we can choose $a\in k$ such that $d(t_i+ax_i)$ form a basis and $\lvert ax_i\rvert <1$ for any $i$, and then $t'_i=t_i+ax_i$ are as required.

Note that the induced valuation on $l=k(t_1\..t_n)$ is Gauss, and so $\lcirc$ is a localization of $\kcirc[t_1\..t_n]$. Since $\lcirc\into\Kcirc$ is a finite lci homomorphism, there exists a representation $\Kcirc=\lcirc[s_1\..s_m]/(f_1\..f_m)$ and then $\Kcirc$ is a localization of $\kcirc[\ut,\us]/(\uf)$. By \cite[4.1.4]{Mustata-Nicaise}, the canonical module $\omega_{\Kcirc/\kcirc}$ is generated by $\Delta^{-1}\phi$, where $\phi=dt_1\wedge\dots\wedge dt_n$ and $\Delta=\det\left(\frac{\partial f_i}{\partial s_j}\right)$. To describe $\lVert \ \rVert_\rmwt$ it suffices to compute the norm of $\phi$. The definitions of \cite[4.2.3--4.2.5]{Mustata-Nicaise} introduce a log-norm that we call weight: $\rmwt(\phi)=\frac{\nu_k(\Delta)+1}{e}$, where $e=e_{K/k}$ and $\nu_k\:k^\times\to\bfZ$ is the additive valuation of $k$ (the weight function $\rmwt_\phi$ considered in loc.cit. is a function of a point $x\in X$ where $K$ appears as the residue field of $x$; it is the logarithmic analogue of the function $\lVert \phi\rVert$ on $X$). So, we define the {\em weight norm} by $\lVert \phi\rVert_\rmwt=\lvert \Delta\pi_K\rvert$, where $\pi_K$ is a uniformizer of $K$. (The factor $1/e$ is only needed in the additive setting to make the group of values of $K$ equal to $e^{-1}\bfZ$, so that $\nu_k$ agrees with $\nu_K$.) The weight norm on $\omega_{K/k}^{\otimes m}$ is defined via $\lVert \phi^{\otimes m}\rVert_{\rmwt^{\otimes m}}=(\lVert \phi\rVert_\rmwt)^m$.

Note that $\Omega_{\Kcirc/\lcirc}$ is the cokernel of the map $$\oplus_{i=1}^mf_i\Kcirc\stackrel{d}{\to}\oplus_{j=1}^m\Kcirc ds_j,$$ where $df_i=\sum_{j=1}^n \frac{\partial f_i}{\partial s_j}ds_j$. It follows that $\Delta=\cont(\Omega_{\Kcirc/\lcirc})=\delta_{K/l}$.

\subsubsection{Comparison}
The norms $\lVert \ \rVert_\omega$ and $\lVert \ \rVert_\rmwt$ on the one-dimensional vector space $\omega_{K/k}$ differ by a factor, so to compare them it suffices to evaluate the K\"ahler seminorm at $\phi$. Since $\lvert t_i\rvert =1$, Theorem~\ref{tmonth} implies that $dt_1\..dt_n$ is a basis of the $\Kcirc$-module $\Omega^\rmlog_{\lcirc/\kcirc}\otimes_{\lcirc}\Kcirc$. The homomorphism $\psi^\rmlog_{\Kcirc/\lcirc/\kcirc}$ is an almost embedding by Lemma~\ref{philem}, and since its source is torsion free (even free), it is injective and we obtain an exact sequence
$$0\to\Omega^\rmlog_{\lcirc/\kcirc}\otimes_{\lcirc}\Kcirc\to\Omega^\rmlog_{\Kcirc/\kcirc}\to\Omega^\rmlog_{\Kcirc/\lcirc}\to 0.$$
Note that
$$\left(\Omega^\rmlog_{\Kcirc/\kcirc}\right)_\tf/\left(\Omega^\rmlog_{\lcirc/\kcirc}\otimes_{\lcirc}\Kcirc\right)=\Omega^\rmlog_{\Kcirc/\lcirc}/\left(\Omega^\rmlog_{\Kcirc/\kcirc}\right)_\tor$$
and denote this module by $M$.

By Theorem~\ref{unitballth}, $(\Omega^\rmlog_{\Kcirc/\kcirc})_\tf$ is an almost unit ball of $\lVert \ \rVert_{\omega,K/k}$. Since $\phi$ is a basis of $\det(\Omega^\rmlog_{\lcirc/\kcirc}\otimes_{\lcirc}\Kcirc)$ we have that $$\lVert \phi\rVert_{\omega,K/k}=\left[\left(\Omega^\rmlog_{\Kcirc/\kcirc}\right)_\tf:\left(\Omega^\rmlog_{\lcirc/\kcirc}\otimes_{\lcirc}\Kcirc\right)\right]^{-1},$$ and then Lemma~\ref{contentindex} implies that $\lVert \phi\rVert_{\omega,K/k}=\cont(M)$. By Theorem~\ref{contlem} $$\cont(M)=\cont\left(\Omega^\rmlog_{\Kcirc/\lcirc}\right)/\cont\left(\left(\Omega^\rmlog_{\Kcirc/\kcirc}\right)_\tor\right)=\delta^\rmlog_{K/l}/\delta^\rmlog_{K/k},$$
and we obtain that $\lVert \phi\rVert_\omega=\delta^\rmlog_{K/l}/\delta^\rmlog_{K/k}$. Since $\delta^\rmlog_{K/l}=\delta_{K/l}\lvert \pi_K\pi_l^{-1}\rvert$ and $\lvert \pi_l\rvert =\lvert \pi_k\rvert$, the equality rewrites as $\lVert \phi\rVert_\omega=(\delta^\rmlog_{K/k})^{-1}\Delta\lvert \pi_K\pi_k^{-1}\rvert$. Thus, $\lVert \ \rVert_\rmwt=\lvert \pi_k\rvert \delta^\rmlog_{K/k}\lVert \ \rVert_\omega$ and twisting by $m$ we obtain the following comparison result.

\begin{theor}\label{comparth}
If $k$ is discretely valued, $X$ is quasi-smooth and $x\in X$ is a divisorial point, i.e. a monomial point with discretely valued $K=\calH(x)$, then the K\"ahler and the weight norms on $m$-canonical forms are related by $$\lVert \ \rVert_{\rmwt^{\otimes m}}=\lvert \pi_k\rvert^m\left(\delta^\rmlog_{K/k}\right)^{m}\lVert \ \rVert_{\omega^{\otimes m}}.$$
\end{theor}

\begin{rem}
(i) If $X$ is the analytification of a smooth $k$-variety, Musta\c{t}\u{a} and Nicaise extend the weight norms $\lVert \ \rVert _{\rmwt,x}$ to a {\em weight seminorm} $\lVert \ \rVert _{X,\rmwt}$ on the whole $X$ by semicontinuity, i.e. $\lVert \ \rVert _{X,\rmwt}$ is the minimal seminorm that extends the family $\{\lVert \ \rVert _{\rmwt,x}\}_{x\in X^\div}$. If $\cha(\tilk)=0$ then $\lVert \ \rVert_{X,\omega^{\otimes m}}$ is an $X^\div$-seminorm by Corollary~\ref{plmaxcor}, hence Theorem~\ref{comparth} implies that $\lVert \ \rVert_{X,\rmwt^{\otimes m}}=\lvert \pi_k\rvert^m\lVert \ \rVert_{X,\omega^{\otimes m}}$. If $\cha(\tilk)>0$ then it is easy to see that $X=\bfA^1_k$ contains divisorial points $x$ with $\delta^\rmlog_{\calH(x)/k}<1$. Hence the seminorms differ already for $\bfA^1_k$.

(ii) The constant factor $\lvert \pi_k\rvert$ in the formula for $\lVert \ \rVert_\rmwt$ is analogous to the $-1$ shift in \cite[4.5.3]{Mustata-Nicaise}, while the log different factor is rather subtle (whenever $\cha(\tilk)>0$). It seems very probable that for any $K$ and quasi-smooth $X$, the function $\delta^\rmlog(x)=\delta^\rmlog_{\calH(x)/k}$ is upper semicontinuous on $X$. In particular, the seminorms $\lVert \ \rVert_{\rmwt,x}=\lvert \pi_k\rvert \delta^\rmlog_{\calH(x)/k}\lVert \ \rVert_{\omega,x}$ should define an analytic seminorm $\lVert \ \rVert_{X,\rmwt}$ on $\omega_X$. Also, I expect that $\lVert \ \rVert_{X,\rmwt}$ (as well as $\lVert \ \rVert_{X,\omega^{\otimes m}}$) is an $X^\div$-seminorm, and hence it coincides with the weight seminorm of Musta\c{t}\u{a}-Nicaise in the situation they considered.
\end{rem}

\section{The topological realization of $X_G$}\label{kansec}
Let $X$ be an analytic space. In Section \ref{kansec} we study the topological space $\lvert X_G\rvert $ associated to $X_G$ and prove Theorem~\ref{tamechart} that was used earlier in the paper. The section is independent of the rest of the paper, so there is no cycle reasoning here.

\subsection{Topological realization of $X_G$}\label{Gsec}

\subsubsection{Prime filters}
Let us recall the definition of completely prime filters on $G$-topological spaces (e.g., see \cite[p. 83]{Put-Schneider}). Let $p=\{U_i\}$ be a set of analytic domains of $X$ then
\begin{itemize}
\item[(0)] $p$ is {\em proper} if $\emptyset\notin p$ and $X\in p$.

\item[(1)] $p$ is {\em saturated} if for any analytic domains $U\subseteq V$ with $U\in p$ also $V\in p$.

\item[(2)] $p$ is {\em filtered} if for any $U,V\in p$ also $U\cap V\in p$.

\item[(3)] $p$ is {\em completely prime} if for any admissible covering $U=\cup_i U_i$ with $U\in p$ at least one $U_i$ is in $p$.
\end{itemize}
We say that $p$ is a {\em filter} if it is proper, saturated and filtered. Note that these three conditions are purely set-theoretic, while the complete primality condition involves the $G$-topology.

\begin{rem}
(i) A filter is called {\em prime} if it satisfies (3) for finite admissible coverings. Van der Put and Schneider considered in \cite{Put-Schneider} prime filters of {\em quasi-compact} analytic domains. In this case, primality and complete primality are equivalent. Moreover, any finite covering of a quasi-compact domain by quasi-compact domains is admissible, hence primality reduces to a set-theoretical condition.

(ii) Our definition deals with arbitrary analytic domains. In this case, prime filters do not form an interesting class and one has to work with completely prime ones.
\end{rem}

\subsubsection{The space $\lvert X_G\rvert$}
By a {\em point} $x$ of $X_G$ we mean a completely prime filter $\{U_i\}$ of analytic domains of $X$. Intuitively, this is the prime filter of all analytic domains ``containing" $x$. We denote by $\lvert X_G\rvert$ the set of all points of $X_G$. For any analytic domain $U\subseteq X$ saturation of a filter of $U$ in $X$ induces an embedding $\lvert U_G\rvert \into \lvert X_G\rvert$.

We provide $\lvert X_G\rvert$ with the topology whose base is formed by all sets of the form $\lvert U_G\rvert$. Obviously, any sheaf $\calF$ on $X_G$ extends to a sheaf $\calF'$ on $\lvert X_G\rvert$ by setting $\calF'(\lvert U_G\rvert )=\calF(U)$ and sheafifying, so we obtain a functor $\alp_X\:X_G^\sim\to\lvert X_G\rvert^\sim$ between the associated topoi, where, as in \cite{sga41}, given a site $\calC$ we denote by $\calC^\sim$ the topos of sheaves of sets on $\calC$. The stalk of $\calF'$ at $x\in\lvert X_G\rvert$ is simply $\colim_{U\in x}\calF(U)$. For shortness, we will denote this stalk as $\calF_x$.

\begin{rem}
We refer to \cite[Tag:00Y3]{stacks} for the definition of points of a general site. It is easy to see that for $G$-topological spaces this definition agrees with our definition given in terms of completely prime filters.
\end{rem}

\subsubsection{Abundance of points}\label{manypoints}
Since any point of $X$ possesses a compact neighborhood and any compact analytic space is quasi-compact in the $G$-topology, the site of $X$ is locally coherent in the sense of \cite[VI.2.3]{sga42}. Therefore, $X_G^\sim$ has enough points by Deligne's theorem, see \cite[VI.9.0]{sga42}.

\begin{theor}\label{equivth}
For any $k$-analytic space $X$ the topological space $\lvert X_G\rvert$ is sober and the functor $\alp_X\:X_G^\sim\to\lvert X_G\rvert^\sim$ is an equivalence of categories.
\end{theor}
\begin{proof}
In view of \cite[IV.7.1.9]{sga41} and \cite[VI.7.1.6]{sga41}, it suffices to show that $X_G^\sim$ is generated by subsheaves of the final sheaf $1_X$. For any analytic subdomain $U\subseteq X$, let $\calL_U\subseteq 1_X$ denote the extension of $1_U$, i.e. $\calL_U(V)=\{1\}$ if $V\subseteq U$ and $\calL_U(V)=\emptyset$ otherwise. If $\calF$ is a sheaf on $X$ then any section $s\in\calF(U)$ induces a morphism $\calL_U\to\calF$. In particular, we obtain an epimorphism $\phi\:\coprod_{s,U}\calL_U\to\calF$, where $U$ runs over all analytic subdomains and $s$ runs over $\calF(U)$.
\end{proof}

\subsubsection{Ultrafilters}
Our next aim is to classify the points of $X_G$ and we start with those corresponding to the maximal completely prime filters.

\begin{lem}\label{ultralem}
The completely prime filter $\calP_x$ of all analytic domains containing a point $x\in X$ is maximal, and any maximal completely prime filter is of the form $\calP_x$. In particular, we obtain an embedding of sets $X\into\lvert X_G\rvert$ whose image consists of all points that have no non-trivial generizations.
\end{lem}
\begin{proof}
If $\calP_x$ is not maximal then it can be increased to a larger completely prime filter $\calP$. Fix an analytic domain $U\in\calP\setminus\calP_x$. Take an affinoid domain $V$ containing $x$, then $W=U\cap V$ lies in $\calP$. Choose an admissible covering of $W$ by affinoid domains $W_i$. Then at least some $W'=W_i$ lies in $\calP$. Since $W'$ is a compact domain in $V$ not containing $x$, there exists a neighborhood $V'$ of $x$ in $V$ such that $V'\cap W'=\emptyset$. Since $V'\in\calP_x\subset\calP$, this contradicts $\calP$ being a filter, so $\calP_x$ is maximal.

Assume, now, that $\calP$ is a maximal completely prime filter. Since $X$ possesses an admissible covering $X=\cup_iV_i$ by affinoid domains we can fix $V=V_i\in\calP$. Assume that $\calP$ is not of the form $\calP_x$ with $x\in V$. By the maximality of $\calP$, for any $x\in V$ we have that $\calP\nsubseteq\calP_x$ and hence there exists an affinoid domain $V_x\subset V$ with $x\notin V_x$. Then $\cap_{x\in V}V_x=\emptyset$, and hence already the intersection of finitely many sets $V_x$ is empty. This contradicts $\calP$ being a filter.
\end{proof}

In the sequel we will freely consider $X$ as a subset of $\lvert X_G\rvert$.

\begin{rem}
One may wonder whether $X\into X_G$ is a topological embedding. Clearly, this may make sense only for the $G$-topology of $X$ since any analytic domain $U\subseteq X$ is the preimage of the open subset $U_G$ of $X_G$. Nevertheless, even for the $G$-topology the answer is negative simply because $X$ is not a topological space for the $G$-topology. Moreover, there exist analytic domains $U$ and $V$ such that $U\cup V$ is not an analytic domain (e.g. the closed polydisc of radii $(1,2)$ and the open polydisc of radii $(2,1)$). So, $U_G\cup V_G$ is open in $X_G$ but its restriction to $X$ is not an analytic domain.
\end{rem}

\begin{cor}\label{ultracor}
Any point $z\in\lvert X_G\rvert$ possesses a unique generization $\gtr(z)$ lying in $X$.
\end{cor}
\begin{proof}
We should prove that any completely prime filter is contained in a single filter of the form $\calP_x$. One such $\calP_x$ exists by Lemma~\ref{ultralem}. Assume that $\calP$ is contained in $\calP_x$ and $\calP_y$ with $x\neq y$. Choose any $V\in\calP$, then $V_x=X\setminus\{y\}$ and $V_y=V\setminus\{x\}$ form an open and, hence, admissible covering of $V$. Thus, either $V_x$ or $V_y$ lies in $\calP$ and we obtain a contradiction.
\end{proof}

\subsubsection{The retraction}
By Corollary~\ref{ultracor} we obtain a retraction $\gtr_X\:\lvert X_G\rvert \to X$ given by $z\mapsto\gtr(z)$. For each $x\in X$, the fiber $\gtr_X^{-1}(x)$ is the set of all specializations of $x$ in $\lvert X_G\rvert$. Thus, $\gtr_X^{-1}(x)$ is the closure of $x$ in $\lvert X_G\rvert$ and we will also denote it $\ox_{X,G}$.

\begin{theor}\label{retrth}
Let $X$ be a $k$-analytic space, $U\subseteq X$ a $k$-analytic subdomain and $x\in U$ a point. Then,

(i) $\gtr_X\:\lvert X_G\rvert \to X$ is a topological quotient map and $X$ is the maximal locally Hausdorff quotient of $\lvert X_G\rvert$.

(ii) $U$ is a neighborhood of $x$ if and only if the inclusion $\ox_{U,G}\subseteq\ox_{X,G}$ is an equality. In particular, $U$ is open if and only if the inclusion $\lvert U_G\rvert\subseteq\gtr_X^{-1}(U)$ is an equality.
\end{theor}
\begin{proof}
We start with (ii). For an analytic domain $V\subseteq X$ with $x\in V$ let $(V,x)$ denote the germ of $V$ at $x$. Define a presheaf of abelian groups on $\lvert X_G\rvert$ as follows: $F(V)$ is either $0$ or $\bfZ$, and the second case takes place if and only if $x\in V$ and $(V,x)$ is not contained in $(U,x)$, i.e. for any neighborhood $W$ of $x$ one has that $W\cap V\nsubseteq W\cap U$. The restriction maps are either identities or the map $\bfZ\to 0$. If $V_1\..V_n$ are domains containing $x$ and satisfying $\cup_{i=1}^n(V_i,x)=(V,x)$, then $(V,x)\nsubseteq(U,x)$ if and only if $(V_i,x)\nsubseteq(U,x)$ for some $i$. It follows that the presheaf $F$ is separated and hence the sheafification map $F\to\calF=\alp F$ is injective by \cite[II.3.2]{sga41}. In particular, $\calF=0$ if and only if $F=0$. Obviously, $F=0$ if and only if $U$ is a neighborhood of $x$.

On the other hand, the stalk of $\calF$ at a point $z\in\lvert X_G\rvert $ is given by $\calF_z=\colim_{z\in W} F(W)$, in particular, $\calF_z$ is either 0 or $\bfZ$. The second possibility holds if and only if for any $W$ with $z\in W$ one has that $x\in W$ and $(W,x)$ is not contained in $(U,x)$. The first condition means that $\gtr_X(z)=x$ and then the second condition holds if and only if $z\notin\lvert U_G\rvert $, i.e. $z\notin\lvert U_G\rvert \cap\ox_{X,G}=\ox_{U,G}$. Thus, $\calF$ has non-zero stalks if and only if the inclusion $\ox_{U,G}\subseteq\ox_{X,G}$ is not equality, and hence $\calF=0$ if and only if $\ox_{U,G}=\ox_{X,G}$. Combining this with the conclusion of the above paragraph we obtain (ii).

Let us prove (i). If $U\subseteq X$ is open then $\gtr_X^{-1}(U)=\lvert U_G\rvert$ by (ii). Thus, $\gtr_X^{-1}(U)$ is open in $\lvert X_G\rvert $ and we obtain that $\gtr_X$ is continuous.

To prove that $\gtr_X$ is a topological quotient map, assume that $U\subseteq X$ is not open and let us prove that $\gtr_X^{-1}(U)$ is not open. Choose a point $x\in U$ not lying in the interior of $U$. By (ii) there exists a point $z\in\ox_{X,G}\setminus\ox_{U,G}$, in particular, $z\notin\lvert U_G\rvert $ and $z\in\gtr_X^{-1}(U)$. We claim that $\gtr_X^{-1}(U)$ is not a neighborhood of $z$. Assume to the contrary that $z$ lies in the interior of $\gtr_X^{-1}(U)$. Then there exists an analytic domain $W\subseteq X$ such that $z\in\lvert W_G\rvert \subseteq \gtr_X^{-1}(U)$. Since $z\notin\lvert U_G\rvert $, we have that $W\nsubseteq U$. So there exists a point $y\in W\setminus U$, and observing that $y\in\lvert W_G\rvert$ and $y\notin \gtr_X^{-1}(U)$ we obtain a contradiction. Thus, $\gtr_X$ is a topological quotient map.

Finally, $X$ is the maximal locally Hausdorff quotient of $\lvert X_G\rvert$ because any locally Hausdorff quotient $\lvert X_G\rvert \to Z$ should identify each point $x\in X$ with any of its specialization $z\in\ox_{X,G}$.
\end{proof}

\subsubsection{Notation $X_G$}
Starting from this point we will not distinguish the site $X_G$ and the topological space $\lvert X_G\rvert$. In particular, we will usually write $x\in X_G$ instead of $x\in\lvert X_G\rvert$.

\subsubsection{Non-analytic points}
Our next aim is to describe the {\em non-analytic} or infinitesimal points of $X_G$, i.e. the points of $X_G\setminus X$. For this we have to recall some results about reductions of germs.

\subsubsection{Germ reduction}\label{germred}
By $\tilcalA_H$ we denote the $H$-graded reduction $\oplus_{h\in H}\tilcalA_h$. In \cite[Section~8]{descent} and \cite[Section~1.5]{flatness}, to any germ $(X,x)$ of an $H$-strict analytic space at a point one associates an $H$-graded reduction $\wt{(X,x)}_H$, which is an $H$-graded Riemann-Zariski space associated to the extension of $H$-graded fields $\wHx_H/\tilk_H$. In particular, any point $z\in\wt{(X,x)}_H$ induces a graded valuation on $\wHx_H$ and if $(X,x)$ is separated then $z$ is also determined by this valuation.

\begin{theor}\label{redth}
Let $X$ be an $H$-strict $k$-analytic space and $x\in X$ a point. Then the closure of $x$ in $X_G$ is canonically homeomorphic to $\wt{(X,x)}_H$.
\end{theor}
\begin{proof}
Let $L$ denote the set-theoretical lattice of subdomains $(U,x)\subseteq(X,x)$. Clearly, the closure of $x$ is a sober topological space and $L$ is its topology base. On the other hand, subdomains of $(X,x)$ are in a one-to-one correspondence with quasi-compact open subspaces of $\wt{(X,x)}_H$ by \cite[Theorem~4.5]{local-properties-II} and \cite[Theorem~8.5]{descent}. So, $L$ is also the lattice of a topology base of the sober topological space $\wt{(X,x)}_H$. Since, a sober topological space is determined by such a lattice (it can be reconstructed as points of the corresponding topos), we obtain the asserted homeomorphism.
\end{proof}

\subsection{Stalks of $\calO_{X_G}$ and $\calOcirc_{X_G}$}
Our next aim is to describe the stalks of the sheaves $\calO_{X_G}$ and $\calOcirc_{X_G}$ at non-analytic points. In particular, this will lead to an explicit description of the homeomorphism from Theorem~\ref{redth}.

\subsubsection{Spectral seminorm}
We provide each stalk $\calO_{X_G,x}$ with the stalk $\lvert \ \rvert_x$ of the spectral seminorm, i.e. $\lvert s\rvert_x=\inf_{x\in U_G}\lvert s\rvert_U$, where $\lvert \ \rvert_U$ is the spectral seminorm of $U$.

\begin{lem}\label{Gstalklem}
Let $X$ be an analytic space and let $x\in X_G$ be a point. Then,

(i) The seminorm $\lvert \ \rvert_x$ is a semivaluation and $\calO_{X_G,x}$ is a local ring whose maximal ideal is the kernel of $\lvert \ \rvert_x$.

(ii) If $y\in X_G$ generizes $x$ then the generization homomorphism $\phi_{x,y}\:\calO_{X_G,x}\to\calO_{X_G,y}$ is an isometry with respect to $\lvert \ \rvert_x$ and $\lvert \ \rvert_y$. In particular, $\phi$ is local.
\end{lem}
For the sake of comparison we note that in the case of schemes any non-trivial generization is not local.
\begin{proof}
Let $z=\gtr_X(x)$ be the maximal generization of $x$. We claim that $\lvert s\rvert_x=\lvert s\rvert_z$ for any $s\in\calO_{X_G,x}$. If $U$ is an analytic domain with $x\in U_G$ then $z\in U$ and so $\lvert s\rvert_x\ge\lvert s\rvert_z$. Conversely, set $r=\lvert s\rvert_z$ and note that $X_\veps=X\{\lvert s\rvert \le r+\veps\}$ is a neighborhood of $z$ for any $\veps>0$. Since $x\in(X_\veps)_G$ we obtain that $\lvert s\rvert_x\le r+\veps$ for any $\veps$, and so $\lvert s\rvert_x\le\lvert s\rvert_z$. In other words, $\phi_{x,z}$ is an isometry. In the same way, $\phi_{y,z}$ is an isometry, and therefore $\phi_{x,y}$ is an isometry.

It remains to prove (i). Since $\phi_{x,z}$ is an isometry and $\lvert \ \rvert_z$ is multiplicative, it follows that $\lvert \ \rvert_y$ is multiplicative. If $\lvert s\rvert_x=0$ for $s\in\calO_{X_G,x}$ then $s$ is not invertible. Conversely, if $\lvert s\rvert_x=r>0$ then $\lvert s\rvert_z >0$ and hence there exists a neighborhood $U$ of $z$ such that $s\in\calO_{X_G}(U)^\times$. Since $x\in U_G$, we obtain that $s$ is invertible.
\end{proof}

\subsubsection{Residue fields}
Once we know that $\calO_{X_G,x}$ is local, we denote its maximal ideal by $m_{G,x}$. The residue field will be denoted $\kappa_G(x)=\calO_{X_G,x}/m_{G,x}$. Since $m_{G,x}$ is the kernel of $\lvert \ \rvert_x$, the residue field acquires a real valuation and we denote its completion by $\calH(x)$. This extends the notation of \S\ref{localringsec} to non-analytic points.

\subsubsection{Generization homomorphisms}
Recall that by Lemma~\ref{Gstalklem}(ii), any generization homomorphism $\phi_{x,y}\:\calO_{X_G,y}\to\calO_{X_G,x}$ induces an embedding of real-valued fields $\kappa_G(y)\into\kappa_G(x)$.

\begin{lem}\label{genlem}
If $y\in X_G$ generizes $x\in X_G$ then the induced embedding $\kappa_G(x)\into\kappa_G(y)$ has dense image and so $\calH(x)=\calH(y)$.
\end{lem}
\begin{proof}
Note that we can replace $X$ with an analytic domain $X'$ such that $x\in X'_G$. In particular, we can assume that $X$ is affinoid. Let $z\in X$ be the maximal generization of $y$ and $x$. The local embeddings $\calO_{X,z}\into\calO_{X_G,x}\into\calO_{X_G,y}\into\calO_{X_G,z}$ induce embeddings of the residue fields $\kappa(z)\into\kappa_G(x)\into\kappa_G(y)\into\kappa_G(z)$. It remains to use that $\kappa(z)\into\kappa_G(z)$ has a dense image.
\end{proof}

\subsubsection{Stalks of $\calOcirc_{X_G}$}
As in \cite[Section 2.1]{temrz}, by a semivaluation ring $A$ with semifraction ring $B$ we mean the following datum: a local ring $(B,m)$ and a subring $A\subseteq B$ such that $m\subset A$ and $A/m$ is a valuation ring of $k=B/m$. Such a datum defines an equivalence class of semivaluations $\nu\:B\to\Gamma\cup\{0\}$ whose kernel is $m$, and $A=\nu^{-1}(\Gamma_{\le 1}\cup\{0\})$ is the ring of integers of $\nu$.

\begin{lem}\label{Ocirclem}
For any $x\in X_G$, the stalk $\calOcirc_{X_G,x}$ is a semivaluation ring with semifraction ring $\calO_{X_G,x}$. In particular, $\calOcirc_{X_G,x}/m_{G,x}$ is a valuation ring of $\kappa_G(x)$.
\end{lem}
\begin{proof}
If $s\in m_{G,x}$ then $\lvert s\rvert_x=0$ and hence $\lvert s\rvert <1$ in a neighborhood of $x$. In particular, $s\in\calOcirc_{X_G,x}$. It remains to show that if $u\in\calO_{X_G,x}^\times$ then either $u$ or $u^{-1}$ lies in $\calOcirc_{X_G,x}$. Shrinking $X$ we can assume that $u\in\calO_{X_G}(X)^\times$. Then $X$ is the union of $Y=X\{u\}$ and $Z=X\{u^{-1}\}$, hence $x$ lies in either $Y_G$ or $Z_G$, and then $u\in \calOcirc_{X_G,x}$ or $u^{-1}\in \calOcirc_{X_G,x}$, respectively.
\end{proof}

\subsubsection{Valuation $\nu_x$}\label{nusec}
Let $\nu_x$ denote both the semivaluation induced by $\calOcirc_{X_G,x}$ on $\calO_{X_G,x}$ and the valuations induced on $\kappa_G(x)$ and $\calH(x)$. If $x\in X$ then an element $s\in\calO_{X_G,x}$ lies in $\calOcirc_{X_G,x}$ if and only if $\lvert s\rvert_x\le 1$. Thus, $\calOcirc_{X_G,x}/m_{G,x}=\calH(x)^\circ$, i.e. $\nu_x$ is the standard real valuation of $\calH(x)$. If $x$ is arbitrary, we only have an inclusion $\calOcirc_{X_G,x}/m_{G,x}\into\calH(x)^\circ$, so $\nu_x$ is composed from the real valuation of $\calH(x)$ and the residue valuation $\tilnu_x$ on $\wHx$.

The following remark clarifies the relation between these valuations and germ reductions. It will not be used in the sequel so we just formulate the results.

\begin{rem}\label{Ocircrem}
(i) The construction of $\tilnu_x$ can be extended by associating to $x$ a graded valuation ring of $\wHx_H$ (see \S\ref{germred}) whose component in degree 1 is $\calOcirc_{X_G,x}/m_{G,x}$, and the argument is essentially the same. Set $\calO=\calO_{X_G}$ for shortness and let $\calOcirc_r$ and $\calO^\circcirc_r$ be the subsheaves of $\calO$ whose sections on $U$ satisfy $\lvert s\rvert_U\le r$ and $\lvert s\rvert_U<r$, respectively. Then $A_x=\oplus_{h\in H}(\calOcirc_h)_x/(\calO^\circcirc_h)_x$ is a graded valuation ring of $\wHx_H=\oplus_{h\in H}(\calO_x)^\circ_h/(\calO_x)^\circcirc_h$ that coincides with the graded valuation ring induced by the image of $x$ under the homeomorphism $\gtr_X^{-1}(y)=\wt{(X,y)}_H$ from Theorem~\ref{redth}, where $y\in X$ is the maximal generization of $x$. In particular, $x\in X$ if and only if $A_x=\wHx_H$, i.e. the graded valuation is trivial.

(ii) Assume that $H=\sqrt{\lvert k^\times\rvert }$ and hence $X_G$ is the usual strictly analytic G-topology. Then the $H$-graded reduction coincides (as a topological space) with the ungraded reduction $\wt{(X,x)}$ because taking the degree-1 components provides a one-to-one correspondence between graded valuation $\tilk_H$-rings in $\wHx_H$ and valuation $\tilk$-rings in $\wHx$. In particular, $x$ is determined by its maximal generization $y$ and a point of $\wt{(X,y)}$, or, that is equivalent, $x$ is determined by the valuation $\nu_x$. This implies that $X_G$ coincides with the Huber adic space $X^\ad$ corresponding to $X$, and $x\in X$ if and only if $\nu_x$ is of height one. Furthermore, if $H=\bfR_{>0}$ then $X_G$ coincides with the so-called reified adic space introduced by Kedlaya in \cite{Kedlaya-reified}, and (perhaps) the case of a general $H$ will be set in details in \cite{skeletons}.
\end{rem}

\subsection{Topological realization of PL spaces and skeletons}\label{Psec}

\subsubsection{The $n$-dimensional affine $R_S$-PL space}
Consider the $R_S$-PL space $A=\bfR_{>0}^n$ with coordinates $t_1\..t_n$. It is provided with the $G$-topology $A_G$ of $R_S$-PL subspaces. Recall that $U\subseteq A$ is an $R_S$-PL subspace if the $R_S$-polytopes contained in $U$ form a quasi-net (hence also a net) of $U$, and a covering $U=\cup_i U_i$ by $R_S$-PL subspaces is admissible if $\{U_i\}_i$ is a quasi-net of $U$.

For any polytope $P\subset A$ the topological realization $\lvert P_G\rvert$ is a quasi-compact topological space. As in the case of analytic spaces (see \S\ref{manypoints}), Deligne's theorem implies that $A_G$ has enough points, and, moreover, $(A_G)^\sim$ is equivalent to $\lvert A_G\rvert^\sim$. For shortness, we will not distinguish $A_G$ (resp. $P_G$) and its topological realization $\lvert A_G\rvert$ (resp. $\lvert P_G\rvert$).

\subsubsection{$G$-skeletons}
If $X$ is an analytic space with a $\bfZ_H$-PL subspace $P$ then the embedding $i\:P\into X$ is continuous with respect to the $G$-topologies, see \cite[Theorem~6.3.1]{bercontr2}. Since the functor that associates to a sober topological space the lattice of its open subsets is fully faithful, this implies that $i$ extends to a continuous embedding $i_G\:P_G\into X_G$ and we say that $P_G$ is a {\em $\bfZ_H$-PL subspace} of $X_G$.

\begin{rem}\label{PGrem}
(i) The main advantage of working with $P_G$ is that it is a honest topological space, so one can use local arguments. One has to describe the new points but, as we will see, this is simple: points of $P_G$ correspond to valuations on abelian groups, and points of $i_G(P_G)$ correspond to $\ut$-monomial valuations.

(ii) It seems that the use of model theory in \cite{skeletons} is mainly needed for the same aim. One interprets $P_G$ in terms of definable sets and types in the theory of ordered groups, and Deligne's theorem on points of locally coherent sites is replaced with G\"odel's completeness theorem. This is not so surprising, since it is known that G\"odel's theorem and Deligne's theorem are equivalent, when appropriately translated (for example, see \cite{Frot}).
\end{rem}

\subsubsection{Monomiality}
Notions of $\ut$-monomial and generalized Gauss valuations naturally extend to general valuations. Namely, let $L/l$ be an extension of valued fields and let $(t_1\..t_n)$ be a tuple of elements of $L$. We say that the valuation on $l(\ut)$ is a generalized Gauss valuation (with respect to $l$) if for any polynomial $a=\sum_{i\in\bfN^n} a_i\ut^i\in l[\ut]$ the equality $\lvert a\rvert =\max_i\lvert a_i\ut^i\rvert$ holds. Such a valuation on $l(\ut)$ is uniquely determined by its restrictions onto $l$ and the monoid $\ut^\bfZ:=\prod_{j=1}^n t_j^\bfZ$. If, in addition, $L$ is finite over the closure of $l(\ut)$ in $L$ then we say that $L$ and its valuation are {\em $\ut$-monomial}. The following result is proved in \cite[Proposition~1.8.3]{skeletons}.

\begin{lem}\label{monomlem}
Assume that $X$ is an analytic space, $f\:U\to\bfG^n_m$ is a monomial chart given by $t_1\..t_n\in\calO_{X_G}^\times(U)$, and $x\in X_G$ is a point whose maximal generization $y$ is contained in the $\bfZ_H$-PL subspace $P=S(f)$. Let $l$ denote the field $\wHy$ provided with the valuation $\tilnu_x$ (see Remark~\ref{Ocircrem}(i)), and let $d=\trdeg(l/\tilk)$. Assume that $\lvert t_i\rvert_x=1$ for $1\le i\le d$, and set $s_i=\tilt_i$ for $1\le i\le d$. Then $x\in P_G$ if and only if the extension $l/\tilk$ is $s$-monomial.
\end{lem}

Consider a monomial chart $f\:U\to\bfG^n_m$ given by $u_1\..u_n\in\calO_{X_G}^\times(U)$. Let $g\:U\to\bfG^n_m$ be given by $t_i=c_i\prod_{j=1}^nu_i^{l_{ij}}$, where $c_i\in k^\times$ and $(l_{ij})$ is an $n$-by-$n$ integer matrix with non-zero determinant. Then it is easy to see that $g$ is another monomial chart and $S(f)=S(g)$. Note also that if $x\in U_G$ is a point whose maximal generization $y$ is monomial then choosing monomials $t_i$ appropriately we can achieve that $\lvert t_i\rvert_x=1$ for $1\le i\le F_y=\trdeg(\wHx/\tilk)$ (in terms of \cite[Section~1.8]{skeletons}, $t_i$ are well presented at $x$). Therefore, Lemma~\ref{monomlem} implies the following corollary.

\begin{cor}\label{monomcor}
Assume that $P$ is a $\bfZ_H$-PL subspace of $X$ and $x\in P_G$ is a point, and let $l$ denote the residue field $\wHx$ with the valuation $\tilnu_x$. Then the extension $l/\tilk$ is Abhyankar.
\end{cor}

\subsubsection{Structure sheaf of $A$}
In the remaining part of Section \ref{Psec} we provide the promised valuation-theoretic description of the points of $P_G$. This will not be used in the sequel, so the reader can skip to Section~\ref{unrammonom}.

Until the end of Section \ref{Psec}, we fix arbitrary $R$ and $S$, and $L=S\oplus(\oplus_{i=1}^nt_i^R)$ denotes the group of $R_S$-monomial functions on $A$. We provide $A$ with the sheaf $\calOcirc_A$ such that if $P$ is an $R_S$-polyhedron then $\calOcirc_A(P)$ is the monoid of $R_S$-PL functions with values in $(0,1]$.

\subsubsection{Combinatorial valuations}
By an {\em $R_S$-valuation} on $L$ we mean any homomorphism $\lvert \ \rvert \:L\to\Gamma$ to an ordered group such that if $r\in R_+=R\cap\bfR_{\ge 0}$ and $x\in L$ with $\lvert x\rvert \le 1$ then $\lvert x^r\rvert \le 1$, and the restriction of $\lvert \ \rvert$ onto $S\subseteq L$ is equivalent to the embedding $S\into\bfR_{>0}$. We say that a valuation is {\em bounded} if for any $x\in L$ there exists $s\in S$ with $\lvert x\rvert \le \lvert s\rvert$.

\subsubsection{Valuation monoids}
Analogously to ring valuations, the valuation is determined up to an equivalence by the {\em valuation monoid} $\Lcirc$ consisting of all elements $x\in L$ with $\lvert x\rvert \le 1$. It is bounded if and only if $S\Lcirc=L$. In addition, $(\Lcirc)^\gp=L$ and $\Lcirc$ is an $(R_S)^\circ$-monoid, i.e. it contains $\Scirc=S\cap(0,1]$ and is closed under the action of $R_+$. In this case, we say that $\Lcirc$ is a {\em bounded valuation $R_+$-monoid of $L$}.

\subsubsection{Valuative interpretation of points}
Similarly to the points of $X_G$, points of $P_G$ admit a simple valuative-theoretic interpretation.

\begin{theor}\label{PGth}
Let $A$ and $L$ be as above. Then for any point $x\in A_G$ the stalk $\calOcirc_{A_G,x}$ is a bounded valuation $R_+$-monoid of $L$, and this establishes a one-to-one correspondence between points of $P_G$ and bounded $R_S$-valuations on $L$.
\end{theor}
\begin{proof}
Choose an $R_S$-polytope $P$ with $x\in P_G$. Clearly, $M=\calOcirc_{A_G,x}$ is an $R_+$-monoid. Also, $M^\gp=L=SM$ because these equalities hold already for the monoid $M'=\calOcirc_{A_G}(P)$. To prove that $M$ is a valuation monoid it suffices to show that if $a\in L$ then either $a\in M$ or $a^{-1}\in M$, but this follows from the fact that $P=P\{a\le 1\}\cup P\{a^{-1}\le 1\}$.

It remains to show that any bounded valuation $R_+$-monoid $\Lcirc$ of $L$ equals to $\calOcirc_{A_G,x}$ for a unique point $x$. Consider the set $\calF$ of all $R_S$-polytopes given by finitely many inequalities $a_i\le 1$ with $a_i\in\Lcirc$. Then $\calF$ is a completely prime filter and the stalk at the corresponding point $x$ is $\Lcirc$. Uniqueness of $\calF$ is also clear from the construction.
\end{proof}

\subsubsection{G-skeleton of the torus}
Now let us assume that $R_S=\bfZ_H$ and so $L=H\oplus\ut^\bfZ$. Set $\bfT=\bfG_m^n$, and consider the embedding $i\:A\into\bfT$ sending $\us$ to the generalized Gauss valuation $\lvert \ \rvert_\us$ and the retraction $r\:\bfT\to A$ from Remark~\ref{modelrem}(ii). Both are continuous in the G-topology, see \cite[Theorems 6.3.1 and 6.4.1]{bercontr2}, hence extend to continuous maps $i_G\:A_G\into \bfT_G$ and $r_G\:\bfT_G\to A_G$. Furthermore, the description of the maps $i$ and $r$ can be naturally extended to $i_G$ and $r_G$. For simplicity, we explain this only in the case when $H=\lvert k^\times\rvert^\bfQ$, and the reduction is ungraded.

Given a point $x\in\bfT_G$ consider its maximal generization $y\in\bfT$ and provide $\calH(x)$ with the valuation $\nu_x$, see Remark~\ref{Ocircrem}. Then the restriction of $\nu_x$ onto $\lvert k^\times\rvert \oplus\ut^\bfZ$ is a bounded $\bfZ_H$-valuation, so we obtain a map $r_G$. Conversely, given a bounded $\bfZ_H$-valuation $\mu\:L\to\Gamma$, we extend it to a generalized Gauss valuation on $k[\ut]$ by the max formula $\lvert \sum_i a_i\ut^i\rvert_\mu=\max_i\lvert a_i\rvert \mu(\ut^i)$. Since $\lvert \ \rvert_\mu$ is composed from a real valuation and a valuation on its residue field, we obtain a point of $\bfT_G$. Clearly, $r_G\circ i_G=\Id$, so $r_G$ is a retraction onto $S(\bfT)_G:=i_G(A)$.

\subsection{Residually unramified monomial charts}\label{unrammonom}
Our last goal is to prove a theorem on existence of residually unramified monomial charts that was used earlier in the paper.

\subsubsection{Residual unramifiedness at non-analytic points}\label{unramtamesec}
Assume that $f\:Y\to X$ is a morphism of $k$-analytic spaces, $y\in Y_G$, $x=f(y)$ and provide $\calH(y)$ and $\calH(x)$ with the valuations $\nu_y$ and $\nu_x$, respectively (see \S\ref{nusec}). We say that $f$ is {\em residually unramified at} $y$ (resp. {\em residually tame at} $y$) if the extension of valued fields $(\calH(y),\nu_y)/(\calH(x),\nu_x)$ is finite and unramified (resp. tame). This extends the analogous notion from the case of analytic points to the whole $Y_G$.

\begin{lem}\label{composlem}
Keep the above notation. Then $f$ is residually unramified at $y$ if and only if the extension of the real-valued fields $(\calH(y),\lvert \ \rvert _y)/(\calH(x),\lvert \ \rvert _x)$ and the extension of the valued fields $(\tilcalH(y),\tilnu_y)/(\tilcalH(x),\tilnu_x)$ are unramified.
\end{lem}
\begin{proof}
This follows from a criterion for an extension of composed valued fields to be unramified, see \cite[Proposition~2.2.2]{temst}.
\end{proof}

\subsubsection{Generators of unramified extensions}
Assume that $L/K$ is a finite unramified extension of valued fields. We say that $u\in L$ is an {\em integral generator} of $L$ over $K$ if $\Lcirc$ is a localization of $\Kcirc[u]$. Note that in this case $g'(u)$ is invertible in $\Lcirc$, where $g(T)$ is the minimal polynomial of $u$ over $K$.

\begin{lem}\label{unramlem}
Assume that $L/K$ is a finite extension of valued fields. Then,

(i) If $L/K$ is unramified then it possesses an integral generator.

(ii) An element $u\in L$ is an integral generator if and only if $L=K[u]$, $u\in\Lcirc$ and $g'(u)\in(\Lcirc)^\times$, where $g$ is the minimal polynomial of $u$ over $K$.
\end{lem}
\begin{proof}
The first claim is a (simple) special case of Chevalley's theorem \cite[$\rm IV_4$, 18.4.6]{ega}. Only the inverse implication needs a proof in (ii), so let us establish it. Consider the subring $A=\Kcirc[u,g'(u)]$ of $\Lcirc$. By \cite[$\rm IV_4$, 18.4.2(ii)]{ega}, $A$ is \'etale over $\Kcirc$. Therefore, $A$ is a semilocal {\em Pr\"ufer} ring (i.e. all its localizations are valuation rings) and it follows that any intermediate ring $A\subseteq R\subseteq\Frac(A)$ is a localization of $A$. In particular, $\Lcirc$ is a localization of $A$.
\end{proof}

\subsubsection{Residually unramified locus}
Given a morphism $f\:Y\to X$, by the {\em residually unramified locus} of $f$ we mean the set of all points $y\in\lvert Y_G\rvert $ such that $f$ is residually unramified at $y$.

\begin{theor}\label{Gopenlem}
Assume that $f\:Y\to X$ is a quasi-\'etale morphism. Then the residually unramified locus of $f$ is open in $\lvert Y_G\rvert $.
\end{theor}
\begin{proof}
Assume that $f$ is residually unramified at $y\in Y_G$. We should prove that $f$ is residually unramified in a neighborhood of $y$. Set $x=f(y)$, $L=(\calH(y),\nu_y)$ and $K=(\calH(x),\nu_x)$. Also, let $y_0\in Y$ and $x_0\in X$ be the maximal generizations of $y$ and $x$, respectively, and consider the real-valued fields $L_0=\calH(y_0)$ and $K_0=\calH(x_0)$.

Step 1. {\it We can assume that $X=\calM(\calA)$ is affinoid and $f$ is finite \'etale.} The question is $G$-local at $x$ and $y$ hence we can assume that $X$ and $Y$ are affinoid. It follows easily from \cite[Theorem~3.4.1]{berihes} that replacing $X$ and $Y$ with neighborhoods of $y_0$ and $x_0$, we can achieve that $f$ factors as $Y\into\oY\to X$, where $Y$ is an affinoid domain in $\oY$ and $\oY\to X$ is finite \'etale (see \cite[3.12]{skeletons}). So, replacing $Y$ with $\oY$ we can assume that $f$ is finite \'etale.

Note that $X$ and $Y$ are good. We will use the usual local rings $\calO_{X,x_0}$, $\calO_{Y,y_0}$ and the residue fields $\kappa(x_0)$, $\kappa(y_0)$ in the sequel. Also, we will freely replace $X$ with a neighborhood of $x_0$ when needed.

Step 2. {\it We can assume that $Y=\calM(\calB)$, where $\calB=\calA[u]$ and $u(y)$ is an integral generator of $\Lcirc$ over $\Kcirc$.} By our assumption, $\Lcirc/\Kcirc$ is \'etale, hence by Lemma~\ref{unramlem}(i) there exists an integral generator $v$ of $L$ over $K$. We claim that using Lemma~\ref{unramlem}(ii) one can slightly move $v$ achieving that $v\in\kappa(y_0)$. Indeed, $L_0/K_)$ is separable hence $K_0[v]$ is preserved under small deformations of $u$ by Krasner's lemma applied to $L_0/K_0$. Similarly, if $g_v$ is the minimal polynomial of $v$ then $g'_v(v)$ changes slightly under small deformations of $v$. In particular, a slight change of $v$ preserves the reduction $\wt{g'_v(v)}$ and hence also the equality $\tilnu_y(\wt{g'_v(v)})=1$, which is equivalent to the inclusion $g'_v(v)\in(\Lcirc)^\times$.

In the sequel, $v\in\kappa(y_0)$. Since $\kappa(x_0)$ is henselian by \cite[Theorem~2.3.3]{berihes}, it is separably closed in $\calH(x)$ and therefore the minimal polynomial $g_v(T)$ of $v$ lies in $\kappa(x_0)[T]$. Choose a monic lifting $G(T)\in\calO_{X,x_0}[T]$ and shrink $X$ around $x_0$ so that $G(T)\in\calA[T]$.

Set $\calB=\calA[T]/(G(T))$ and $Y'=\calM(\calB)$. Then $Y'\to X$ is a finite map of degree $d=\deg(G)$ and the preimage of $x_0$ is a single point $y'_0$ such that $\calH(y'_0)/\calH(x_0)$ is the extension $L_0/K_0$. Therefore, $Y'\to X$ is \'etale at $y'_0$ and the maps of germs $(Y',y'_0)\to(X,x_0)$ and $(Y,y_0)\to(X,x_0)$ are isomorphic by \cite[Theorem~3.4.1]{berihes}. In particular, after shrinking $X$ around $x_0$ the morphism $Y'\to X$ becomes \'etale, and then we can replace $Y$ with $Y'$. Then the image $u\in\calB$ of $T$ is as required since $u(y)=v$.

Step 3. {\it The domain $U=Y\{G'(u)^{-1}\}$ is as required.} Clearly $y\in U$, so we should only check that for any $z\in U$ with $F=\calH(z)$ and $E=\calH(f(z))$, the extension $F/E$ is unramified. Set $w=u(z)$, then $F=E[w]$ and the minimal polynomial $h(T)$ of $w$ over $E$ is a factor of $\oG(T)$, where $\oG=G(z)$ is obtained from $G$ by evaluating its coefficients at $z$. Since $h(w)=0$ and $\oG'(w)$ is invertible, we obtain that $h'(w)$ is invertible. By Lemma~\ref{unramlem}(ii), $w$ is an integral generator of $F$ over $E$, in particular, $F/E$ is unramified.
\end{proof}

\subsubsection{Construction of charts}
Now we are in a position to prove the main result of Section~\ref{unrammonom}.

\begin{theor}\label{tamechart}
Assume that $k$ is algebraically closed. Let $X$ be a $k$-analytic space and $P$ a compact $\bfZ_H$-PL subspace of $X$. Then there exist finitely many residually unramified monomial charts $f_i\:U_i\to\bfG_m^{n_i}$ such that $P=\cup_i S({f_i})$.
\end{theor}
\begin{proof}
First, we observe that it suffices to show that for any point $x\in P_G$ there exists a residually unramified monomial chart $f\:U\to\bfG_m^n$ such that $x\in S(f)_G$. Indeed, intersecting this chart with a chart $g$ such that $x\in S(g)_G$ and $S(g)\subseteq P$ we can also achieve that $S(f)\subseteq P$, and then the assertion follows from the quasi-compactness of $P_G$.

By Lemma~\ref{tamechartlem} below, there exists a monomial chart such that $x\in S(f)_G$ and $f$ is residually unramified at $x$. Applying Theorem~\ref{Gopenlem} we can find an analytic domain $V\subseteq U$ such that $x\in V_G$ and the map $f\vert_V$ is residually unramified. Hence $f\vert_V$ is a required monomial chart and we are done.
\end{proof}

\begin{lem}\label{tamechartlem}
Assume that $X$ is an analytic space and $x\in X_G$ is a monomial point. Then there exists a monomial chart $f\:U\to\bfG_m^n$ such that $x\in S(f)_G$ and $f$ is residually monomial at $x$.
\end{lem}
\begin{proof}
The argument is similar to that in Corollary~\ref{monchartcor}, but this time we should choose the transcendence basis of the residue field more carefully. Let $x_0\in X$ be the maximal generization of $x$ and consider the real-valued field $L_0=(\calH(x_0),\lvert \ \rvert _x)$ and the valued fields $L=(\calH(x),\nu_x)$ and $l=(\tilL_0,\tilnu_x)$. Then $x_0$ is monomial and $l/\tilk$ is Abhyankar by Corollary~\ref{monomcor}. Set $E=E_{x_0}$ and $F=F_{x_0}$, and choose elements $a_{F+1}\..a_{F+E}$ such that their images in $\lvert L_0^\times\rvert /\lvert k^\times\rvert$ form a basis. Next, choose a transcendence basis $b_1\..b_F$ of $l/\tilk$ such that $l/\tilk(b_1\..b_F)$ is unramified; this is possible by \cite[Theorem~1.3]{Kuhlmann}.

\begin{rem}
Existence of such a basis is the valuation-theoretic ingredient of local uniformization of Abhyankar valuations. It is an immediate consequence of the difficult theorem that $l$ is stable, see  \cite[Theorem~1.1]{Kuhlmann} or \cite[Remark~2.1.3]{insepunif}. In fact, one can take $b_1\..b_F$ to be any basis such that $b_1\..b_\tilE$ is mapped to a basis of $\lvert l^\times\rvert \otimes\bfQ$, where $\tilE=E_{l/\tilk}$, and $b_{\tilE+1}\..b_F$ is mapped to a separable transcendence basis of $\till/\tilk$.
\end{rem}

Choose any lifts $a_1\..a_F\in\kappa_G(x)$ of $b_1\..b_F$. We obtain elements $a_1\..a_n$, where $n=E+F=\dim_{x_0}(X)$, and let $t_1\..t_n\in\calO_{X_G,x}$ be any lifts of $a_1\..a_n$. Let $U$ be such that $t_i$ are defined on $U$ and $x\in U_G$. Then $\ut$ induces a morphism $f\:U\to\bfG_m^n$ and the fiber $f^{-1}(f(x_0))$ is zero-dimensional at $x_0$ by \cite[Corollary~8.4.3]{flatness}. It then follows from \cite[Theorem~4.9]{Ducros} that replacing $U$ by a Zariski open neighborhood of $x_0$ we can achieve that $f$ has zero-dimensional fibers and hence is a monomial chart. By our construction, the induced valuations on $K=k(a_1\..a_n)$ and $\tilK=\tilk(b_1\..b_F)$ are generalized Gauss valuations. The first implies that $x_0\in S(f)$, hence the second implies that $x\in S(f)_G$ by Lemma~\ref{monomlem}.

The residue field extension $l/\tilk(b_1\..b_F)$ is unramified and hence separable. In addition, $L_0$ is unramified over $\calH(f(x_0))=\wh{k(a_1\..a_F)}$ because the fields are stable and have the same group of values and the extension of the residue fields is separable. Hence $f$ is residually unramified at $x$ by Lemma~\ref{composlem}, and we are done.
\end{proof}

\bibliographystyle{amsalpha}
\bibliography{topform}

\end{document}